\documentclass[10pt,twoside]{amsart}
\usepackage{amsmath}
\usepackage{amssymb}
\usepackage{wasysym}
\usepackage{psfrag}
\usepackage{graphicx,epsf,amsmath}  
\usepackage{epsf,graphicx}
\usepackage{comment}

\newcommand{\dpy}{\displaystyle}
\newtheorem{thm}{Theorem}[section]
\newtheorem{pro}[thm]{Proposition}
\newtheorem{lem}[thm]{Lemma}
\newtheorem{lemma}[thm]{Lemma}
\newtheorem{cor}[thm]{Corollary}

\newtheorem{defi}[thm]{Definition}

\newtheorem{preremark}[thm]{Remark}
\newenvironment{remark}{\begin{preremark}\rm}{\end{preremark}}
\def\NN{{\mathbb N}}
\def\nat{{\mathbb N}}
\def\ZZ{{\mathbb Z}}
\def\Z{{\mathbb Z}}
\def\integer{{\mathbb Z}}
\def\RR{{\mathbb R}}
\def\R{{\mathbb R}}
\def\real{{\mathbb R}}
\def\T{{\mathbb T}}
\def\TT{{\mathbb T}}
\def\torus{{\mathbb T}}

\def\C{{\mathbb C}}
\def\CC{{\mathbb C}}

\def\M{{\mathcal M}}
\def\L{{\mathcal L}}
\def\E{{\mathcal E}}
\def\A{{\mathcal A}}
\def\U{{\mathcal U}}
\def\B{{\mathcal B}}
\def\J{{\mathcal J}}
\def\K{{\mathcal K}}
\def\F{{\mathcal F}}
\def\S{{\mathcal S}}
\def\D{{\mathcal D}}

\def\Tau{{\mathcal T}}
\def\ep{\varepsilon}
\def\th{\theta}
\def\hth{\hat{\theta}}

\def\Id{{\rm Id}}
\def\avg{{\rm avg\,}}
\def\Im{{\rm Im\,}}
\def\dist{ {\rm dist}}
\def\meas{ {\rm meas} }

\title[Quasi-periodic breathers]{Construction of invariant whiskered tori by a
  parameterization method. Part II: Quasi-periodic and almost periodic breathers in coupled map lattices. }
\author[E. Fontich]{Ernest Fontich}
\address{Departarmnent de Matem\`atica Aplicada i An\`alisi, 
Via 585. 08007 Barcelona. Spain}
\email{fontich@maia.ub.es}
\author[R. de la Llave]{Rafael de la Llave}
\address{Shool of Mathematics, Georgia Institute of Technology, 686 Cherry St. Atlanta, GA 30332-0160, USA} 
\email{rafael.delallave@math.gatech.edu}
\thanks{Supported by NSF grants.}
\author[Y. Sire]{Yannick Sire}
\address{
Universit\'e Paul C\'ezanne, Laboratoire LATP UMR 6632, Marseille, France
}
\email{sire@cmi.univ-mrs.fr}

\begin{document}
\maketitle
\begin{abstract}
We construct quasi-periodic and almost periodic solutions
for coupled Hamiltonian systems on an infinite lattice     which is translation 
invariant.   The couplings can be 
long range, provided that they decay moderately fast with 
respect to
the distance.

For  the solutions  we construct, most of the sites are  moving 
in a neighborhood of
a hyperbolic fixed  point, but there are oscillating sites clustered
around a sequence of  nodes. The amplitude of these oscillations
does not need to tend to zero. In particular, the almost periodic solutions 
do not decay at infinity. 

The main result is an \emph{a-posteriori} theorem. 
We formulate an invariance equation. Solutions of
this equation are embeddings of an invariant torus
on which the motion is conjugate to a rotation.
We show that, if there is an approximate solution of 
the invariance equation that satisfies some non-degeneracy 
conditions, there is a true solution close by. 

This 
does not require that the system 
is close to integrable, hence it can be used to 
validate numerical 
calculations or formal expansions.

The proof  of this \emph{a-posteriori} theorem is based on a
Nash-Moser iteration, which does not use transformation theory.
Simpler versions of the scheme were developed in
E. Fontich, R.  de la Llave,Y.  Sire \emph{J. Differential. Equations.} {\bf 246}, 
3136 (2009). 

One  technical tool, important for our purposes, is the use of 
weighted spaces that capture the idea that the maps under consideration
are local interactions. Using these weighted spaces, 
the estimates of 
iterative steps are similar  to those in finite dimensional
spaces. In particular, the estimates are independent of the number of
nodes that get excited.  
Using these techniques, given  two breathers, we can place them 
apart and obtain an approximate solution, which leads to 
a true solution nearby. By repeating the process infinitely often, 
we can get solutions with infinitely many frequencies which 
do not tend to zero at infinity.

\end{abstract}
\tableofcontents
\section{Introduction}
The goal of this paper is to prove theorems on persistence of
invariant tori in some lattice systems. These
models describe copies of identical systems placed on the nodes of
a lattice and interacting with all the other systems in the lattice. The
interaction can be of infinite range, but it has to 
decay sufficiently fast with the distance. 
We will assume that the dynamics is Hamiltonian
and, for simplicity, we will also assume that the dynamics is
analytic. We will consider \emph{``whiskered tori''}. These
are invariant tori such that the motion on them is a rotation
and which are as hyperbolic as possible, compatible with the 
fact that the motion is an irrational rotation (it is well 
known that the directions symplectically conjugate to the tangent of 
the tori have to be neutral). See Definition \ref{NDloc-hyp}.

The main technical tool we will develop is a theorem of persistence 
of finite dimensional whiskered 
 tori, namely Theorem~\ref{existenceembedloc} below, 
which has sufficiently good properties to allow us to use it recursively 
to construct tori with infinite frequencies.

The tori we consider in Theorem~\ref{existenceembedloc} 
are finite dimensional whiskered tori with
some local character. The motion on the torus is a rigid rotation
with a Diophantine frequency. The preservation of the
symplectic structure and the rotational motion on the tori force
that there are some neutral directions in the normal directions.
We will assume that, except for these directions, the normal
directions are hyperbolic (they expand at exponential rates either
in the future or in the past). In particular, the hyperbolic 
spaces are infinite dimensional. 

The main technique to prove Theorem \ref{existenceembedloc} is to 
derive an equation that implies invariance of the torus
and that the motion on it is a rotation and to develop a theory 
for solutions of the equation. 

Given  a map $F$ on a phase space $\M$  and 
a frequency $\omega \in \RR^l$,  it is easy to see 
that $K: \torus^l \rightarrow \M$ is a parameterization of 
a torus with a rotation $\omega$, if and only if 
\begin{equation}\label{invarmap1} 
F \circ K = K \circ T_\omega\, ,
\end{equation}
where $T_\omega$ denotes the rotation on the torus by 
$\omega$. Similarly, for a vector field $X$, we seek 
parameterizations $K$ satisfying 
\begin{equation}\label{invarflow1}
X \circ K = \partial_\omega K,
\end{equation}
where $\partial_\omega$ is the derivative along the direction $\omega$.

Our main result will be
Theorem~\ref{existenceembedloc}, which shows that if we have an
approximate solution of the invariance equation which is also not
too degenerate,  there is a true solution which is close to the
approximate one.  Theorems of this form,
that validate an approximate solution, will be 
called \emph{a posteriori}, following the language in numerical analysis. 

We emphasize that  Theorem~\ref{existenceembedloc} 
does not assume that the system is
close to integrable, so that the approximate solution
could be produced in any way.  Of course, when the system is close to
integrable, we can take as approximate solutions the solutions of
the integrable system, so that we recover the standard
formulations of KAM theorems for 
quasi-integrable systems.  The approximate solutions can be
produced by a variety of methods, including Lindstedt series or
numerical computations. In finite dimensions, some whiskered
tori are generated by resonant averaging \cite{LlaveW04, Treschev94a}
or by homoclinic tangencies \cite{Duarte08}.
In such cases, Theorem~\ref{existenceembedloc} 
leads to justifications of the expansions or the numerical computations. 
We also note that  Theorem~\ref{existenceembedloc} does 
not assume that the system is translation invariant (it assumes 
only the existence of some uniform bounds). 

The {\sl a posteriori} approach to KAM theorem was emphasized in
\cite{Moser66a,Moser66b,Zehnder75a,Zehnder75b,Zehnder76b}. 
There, it was pointed
out that this {\sl a posteriori} approach automatically allows
to deduce results for finitely differentiable systems as well
as to prove smooth dependence on parameters or analyticity of 
perturbative series.
We refer the reader to \cite{Llave01c} for
a comparison of different KAM methods.

In this paper, we use the {\sl a posteriori} format to construct more 
complicated quasi-periodic solutions by juxtaposing two simpler 
solutions separated by a sufficiently long distance. 
The {\sl a posteriori} format of Theorem~\ref{existenceembedloc}, 
allows us to control the limit of the solutions, which will 
be an almost periodic solution. The ability to superimpose solutions far 
apart is greatly facilitated by assuming translation invariance, which will be an assumption in the second part of this work.  One could assume 
significantly less (e.g. some amount of uniformity). 
Nevertheless, this seems a natural assumption.

One important technical tool in this paper is the use of 
spaces of decay functions following \cite{JL00,FLM}. 
These are spaces of functions whose norms quantify the effect that 
the motion of one particle does not affect much the motion of 
particles far apart. Besides that, they also enjoy certain Banach 
algebra properties so that the hyperbolic directions can be 
dealt with in the same manner than in the finite dimensional 
spaces. 

Using  spaces of decay functions,  we can make quantitative the 
observation that, since
the oscillations at one site almost do not affect those further apart,  
superimposing oscillations centered around sites far apart produces 
a very approximate solution. 
We will call the localized oscillating solutions \emph{``breathers''}.
The error in the invariance equation
(measured in the sense of an appropriate space of 
decay functions) is arbitrarily small if the centers are placed 
far enough. A rather simple calculation shows that the 
non-degeneracy conditions deteriorate also by an arbitrarily small 
amount. In summary, if the frequencies of the oscillations are 
jointly Diophantine (even if the constant is bad), we can 
satisfy all the requirements of the theorem by displacing the breathers far
apart.  If one makes appropriate choices -- placing the subsequent 
centers of oscillation far enough apart -- we will show that the process can be repeated 
infinitely often 
and that it converges
 in a sense which is strong enough to justify that  the limit is a 
solution of the system.  This solution contains infinitely many frequencies.

The process of coupling the breathers does not require 
any smallness conditions in the coupling (it suffices 
to place the breathers far enough apart). On the other 
hand, establishing the 
existence of breathers by perturbing from those of the uncoupled system, 
does require some smallness conditions. We also require some mild
smallness conditions on the perturbations to ensure that the 
system remains non-degenerate.

In the solutions that we construct most of the sites are 
near a hyperbolic equilibrium. These solutions, therefore, have 
an average energy close to that of the equilibrium solutions and  are at the border of chaos (in particular, 
they are dynamically unstable). There are indications that these
solutions play an important role in instability. 

The results of the paper were summarized in 
\cite{FontichLS09}, which perhaps can be used as a reading 
guide to the present paper. 

We also note that, after this paper was finished, the work 
\cite{BlazevskiL14}, used the results of this paper to construct the 
whiskers of the whiskered tori constructed in this paper in 
a very similar functional formulation, so that the whiskers also 
have decay properties.

To provide some motivation, we now mention several models 
found in the literature for which our method
applies. These models can be 
described by the following formal Hamiltonian
\begin{equation}\label{generalmodel}
H(p,q) = \sum_{i \in \ZZ^N} \Big(\frac{1}{2} |p_i|^2 +
W(q_i)\Big) + \sum_{k\in \Z^N}\ \sum_{i \in \ZZ^N} V_k (
q_i-q_{i + k})
\end{equation}
under some assumptions on the potentials $W$ and $V_k$. Here the formal Hamiltonian structure is
\begin{equation*}
\Omega_\infty=\sum_{i \in \ZZ^N} d p_i  \wedge d q_i .
\end{equation*}
Note that, even if the sum defining the Hamiltonian 
and the symplectic form are  formal 
and not meant to converge, Hamilton's equations are a well 
behaved system of differential equations
(if the $V_k$ decay fast enough, e.g. if they are finite range). In fact the equations of motion are 
\begin{equation*}
\begin{split} 
\dot q_i &= p_i \\
\dot p_i &= - \nabla W( q_i) - \sum_{k \in \ZZ^N} \Big \{ \nabla V_k(q_i - q_{i +k} ) 
	- \nabla V_k(q_{i-k} - q_i) \Big \}.  
\end{split} 
\end{equation*}

The model \eqref{generalmodel} involves a local potential $W$ for 
each particle and interaction potentials among pairs of 
particles. Of course, the interaction potentials are assumed 
to decay with $|k|$ fast enough. The method of proof 
also accommodates many body interactions. 
One important feature of the method is that, in some 
appropriate weighted spaces, the estimates we obtain 
are independent of the number and the position of the centers of 
oscillation.

If we take the lattice to be with one degree of freedom,
 the potential $V$ to be just nearest neighbor
(i.e. $V_k = 0$  for $|k| > 1$) and set
$V_{1}(s) = \frac{\gamma}{2} s^2$, 
we obtain the so called 1-D Klein-Gordon system
described by the formal  Hamiltonian 
\begin{equation*}
H(q, p)=\sum_{n=-\infty}^{+\infty} \Big(\frac{1}{2}
p_n^2+W(q_n)+\frac{\gamma}{2}(q_{n+1}-q_n)^2\Big),
\end{equation*}
and whose equations of motion are 

\begin{equation}\label{KG}
\ddot{q}_n+W'(q_n)=\gamma (q_{n+1}+q_{n-1}-2q_n),\,\,\,\,\,n\in
\ZZ . 
\end{equation}

We note that the method we present applies to higher dimensional 
lattices and higher dimensional systems. We also do not need 
to assume that the symplectic form is the standard one.
This is convenient when the symplectic form is degenerate. 
Changes in the symplectic form correspond to magnetic fields
\cite{Thirring97}. Note that the systems with magnetic
fields are not reversible.

For a review of the physical relevance of these models
we refer the reader to \cite{FW98}. Concerning the existence proof
of periodic breathers, we refer to
\cite{MKA94,AGT96,AG96,AKK01}. In the latter papers, the technique
is based on a variational argument whereas in
\cite{MKA94}, the authors use an implicit function theorem. For
quasi-periodic breathers in finite -- but arbitrarily large 
systems,  we mention
\cite{BV02,GengY07,GengVY08,ChungY07}. The paper  \cite{Yuan02} proves the 
existence of quasi-periodic breathers in the Fermi-Pasta-Ulam
lattice.
In all the cases above, 
 the breathers are normally  elliptic or dissipative. 
Quasi-periodic and almost periodic 
breathers for lattices of reversible systems with 
dissipation are considered in \cite{ChenY07}. 

\begin{remark}
There is a variety of results showing that for 
hyperbolic PDE's there are no quasi-periodic solutions 
of finite energy \cite{Pyke96,SofferW99,KomechK09, KomechK10}. 
Since some of the models we consider are obtained as 
discretizations of nonlinear wave equations, it is 
interesting to understand why the results above 
do not apply to the discretized model, even if 
they apply to the PDE. 

The reason is that the mechanism behind the proofs 
in the above papers is 
that quasi-periodic solutions of non-linear
PDE's have to  radiate and send energy to infinity. 

In the models we consider,  
there is no radiation because most of the media is 
near the hyperbolic regime. 

We can understand the lack of radiation in 
the model but representative problem
\begin{equation} \label{linearKG}
\ddot{q}_n+ A q_n =\gamma (q_{n+1}+q_{n-1}-2q_n),\,\,\,\,\,n\in
\ZZ, \, \gamma > 0.
\end{equation}
where $A <   0$. 
The equation \eqref{linearKG} is a linearization of 
\eqref{KG} near the point $q = 0$, which is  
a maximum of the potential (or a mimimum of $W$ 
in the notation of \eqref{KG}). 

We see that if we substitute solutions of the form 
$q_n = \exp( i (\omega t + k n))$  in 
\eqref{linearKG}, we are
lead to the  dispersion relation
\[
-\omega^2 + A = \gamma(2 \cos( k)  - 2).
\]

If $|\gamma|$ is small enough, this dispersion relation does not have 
any real solutions  for $\omega$ and the only  square roots are imaginary.

In this model, near the hyperbolic fixed points the equations do 
not propagate waves, so that there is no radiation and the arguments
excluding quasi-periodic solutions in the above papers 
do not apply. 

However, for PDE's, the dispersion relation would be 
$-\omega^2 + A = -\gamma k^2$. The unboudedness of the $k^2$ factor
makes it possible to have propagating waves no matter how small $|\gamma|$ is.

Note also that in the model in \cite{FrolichSW86}, there is no
propagation either because of the random nature of the media. 
\end{remark}

\section{Basic setup and preliminaries}
\label{sec:preliminaries}

The main goal of this  paper
is to extend the method 
introduced in \cite{FontichLS09}  for the study of
whiskered tori
to  some systems on 
infinite dimensional manifolds. 
The systems we will consider  consist of  infinitely many 
finite dimensional Hamiltonian  systems, each of them corresponding to 
a site on a lattice,  subject to some 
coupling.
We will  assume that the  coupling
decays fast enough with respect to the distance among the 
sites. These are standard models in many applied fields and 
there is a large mathematical theory, which we cannot survey 
systematically (but we will make some indication of the results 
we use or the one closer to our goals).

 An important tool for us will be 
appropriate function spaces for these interactions. There 
are many other methods to establish the existence of
whiskered tori \cite{Graff74, Zehnder75b,You99}. The present method 
has the advantage that it depends much less on the subtle geometric 
properties, so that it applies easily to infinite dimensional contexts.

Since  we are interested in translation invariant 
problems and want to produce solutions that do not go to 
zero at infinity, it is natural to model the functional 
analysis in $\ell^\infty$, which has some subtle points 
that require attention. 

The goal of this section is to set up the functional analysis 
spaces modelled after $\ell^\infty$ and capturing the idea that 
changes in one site have very small effect in sites that are
far away. We anticipate that we need two types of spaces. One family of 
spaces  for 
the mappings from the infinite dimensional space to itself 
and another kind of spaces for the mappings from a finite dimensional 
torus to the infinite dimensional phase space.  
This corresponds to the $F$, $K$ in \eqref{invarmap1}
or the $X$, $K$ in \eqref{invarflow1}. 
The spaces we choose are patterned after the choices in 
\cite{JL00,FLM}. Other  Banach spaces of functions in 
lattice systems are in \cite{Rugh02,BricmontK95}.  Indeed, the 
choice of topologies in these infinite dimensional 
systems is rather subtle and arguments  in ergodic theory 
which rely more on measure theory  than on geometry find useful 
toplogies in which the phase space is compact.

\subsection{Phase spaces} 
In this paper 
we will assume that the phase space at each node 
is given by an Euclidean exact symplectic manifold $(M=\T^l \times
\RR^{2d-l}, \Omega=d\alpha)$, 
where $\T= \RR / \ZZ$. We will not assume 
that the symplectic form is given in the standard form 
of action-angle variables. For some calculations we consider 
$\tilde M = \real^l \times \real^{2d -l}$, the universal
covering of $M$ or the complex extensions of the above. 
This is natural since the KAM method requires to consider 
Fourier series.

It is possible to adapt
our method to the case of a non-Euclidean manifold $M$ using 
connectors and exponential mappings. This just requires some 
typographical effort. See the discussion in \cite{FontichLS09}. 

Then,  the phase space of the lattice system will be a subset of 
\begin{equation}
M^{\ZZ^N}=\displaystyle{\prod_{j \in \ZZ^N}} M.
\end{equation}

Since $M$ is unbounded we will take as phase space
\begin{equation}\label{mani}
\M=\ell^\infty(\Z^N,M) = \Big\{ x\in M^{\ZZ^N}\mid\,
\sup_{i\in \Z^N} |x_i| <\infty \Big\}
\end{equation}
which is a strict subset of $M^{\ZZ^N}$. We will endow $\M$
with the distance
\[
d(x, y) = \sup_{i \in \ZZ^N} d(x_i, y_i),
\]
where $d(x_i, y_i)$ is the distance on the finite dimensional 
manifold $M$.

When $M$ is $\mathbb{R}^{2d}$, $\M$ is a Banach space
with the norm
\begin{equation*}
\|x\|_\infty=\sup_{i \in \ZZ^N} |x_i|.
\end{equation*}
When
$M=\T^l\times \mathbb{R}^{2d-l}$, $\M$ is a Banach manifold
modelled on $\ell^\infty(\ZZ^N)$. 

Notice that because $M$ is an Euclidean space, the tangent space of $\mathcal M$
is trivial and can be identified with $\ell^\infty$. Given 
$x \in \M$ and $\xi \in \ell^\infty$, we can define 
$x + \xi$ by just adding the components. 
We see that if $\| \xi \|_\infty < 1/2$, 
the mapping $\xi \to x + \xi$ is injective, so that this 
defines a chart in $\M$. 

\begin{remark}
The fact that we assume that the manifold has the structure
$M=\T^l \times \RR^{2d-l}$ is important here since it implies
that $H^1(M) \sim H^1(\T^l)$ is non-trivial. This allows us to
perform the construction in Appendix \ref{maps}. If the manifold
was such that its first de-Rham cohomology group were trivial, all
the symplectic maps would be exact symplectic and the construction
would not work without changes.
To deal with manifolds such that $H^1(M)$ is trivial
($M=\RR^{2d}$ for instance), one can use the method developed by
the authors in the finite dimensional case (see \cite{FontichLS09,FontichLS09b}). 
This consists in perturbing the invariance equation for the tori
by a translation term and prove, at the end of the convergence
scheme, that the geometry implies that this term is zero.
The method of \cite{FontichLS09} allows to deal with secondary tori 
(i.e. tori which are contractible to tori of lower dimension) 
directly. The present method would require to make some preliminary 
changes of variables. 
\end{remark}

The choice of $\ell^\infty$ is dictated by the fact that we want 
to deal with solutions that neither grow nor decrease at $\infty$. 
This, however,
will lead to some complications,
the functional analysis in $\ell^\infty$ 
being rather delicate.  On the other hand, 
obtaining estimates in  $\ell^\infty$ 
for several of our objects will be relatively easy. 

Since we are going to deal with analytic functions,
one has to define what is the complex extension of the manifold
$\M$. By assumption, the manifold $M$ 
is  an Euclidean manifold, hence, it admits a complex extension
$M^{\mathbb{C}}$. We define the complex extension of $\M$
as a subspace of the product of the complex extensions of $M$,
i.e.
\begin{equation*}
\M^{\mathbb{C}} =\Big\{z\in \prod_{j \in \ZZ^N}
M^\C \ \Big\vert  \ \sup_{i\in\ZZ^N} \vert z_i\vert <\infty\Big\}.
\end{equation*}
In the following,  we will be considering mostly 
$\M^{\CC}$ but to simplify the
notation we will not write the superscript $\CC$, if it does 
not lead to confusion. 

\subsection{Some functional analysis in $\ell^\infty(\ZZ^N)$ and the spaces of decay functions}

As emphasized in \cite{FLM}, $\ell^\infty(\ZZ^N)$ 
has a very complicated dual space
which cannot be identified with a space of 
sequences since there is no Riesz-representation theorem. As
a consequence, we have that the matrix elements of 
an operator do not characterize the operator and, relatedly, 
the differential of a map is not represented by its partial
derivatives.
The
physical meaning is that one has to take into account ``boundary
conditions  at infinity''.

For example, consider the functional
$\Tau$ defined on the  closed subspace 
of $\ell^\infty(\ZZ)$ consisting of convergent sequences
by the formula
\begin{equation*}
\Tau(u)=\lim_{n \rightarrow +\infty} u_n.
\end{equation*}
By the Hahn-Banach theorem, $\Tau$  extends to
$\ell^\infty(\ZZ)$. The extended functional $\Tau$ is
non-trivial but we have
\begin{equation*}
\partial_{u_i} \Tau(u)=0
\end{equation*}
since the limit does not depend on $u_i$.
Of course, the functional $\Tau$ is linear but it is 
not represented by a matrix. 
Similar phenomena have been known in statistical 
mechanics for a while under the name \emph{observables at 
infinity}.

This phenomenon can be eliminated by restricting our 
attention 
to functions whose derivative is a linear functional 
which  is given  by the
matrix of  partial 
derivatives. We will develop some technology that allows 
to verify this assumption rather comfortably in the cases of interest. 
A much more thorough treatment can be found in 
\cite{FLM}. 

\subsubsection{Weighted norms to formulate decay properties}
To formulate quantitatively the approximate locality of the maps
we will consider 
Banach spaces whose norm 
makes precise that changing one coordinate affects 
little the outcome of other coordinates far away.

We will make use of the so-called decay
functions introduced in \cite{JL00}.
\begin{defi}\label{defDecay}
We say that a function $\Gamma: \ZZ^N \rightarrow
\RR_+$,  is a decay function when it satisfies
\begin{enumerate}
\item $\displaystyle{\sum_{j \in \ZZ^N}} \Gamma(j) \leq 1,$
\item $\displaystyle{\sum_{j \in \ZZ^N}} \Gamma(i-j)\Gamma(j-k)
  \leq \Gamma(i-k),\quad\,\,\,i,k \in \ZZ^N$.
\end{enumerate}
\end{defi}

The algebraic property $(2)$ in definition \ref{defDecay} is
important since it is the one that allows us to construct Banach
algebras.

The following elementary  proposition 
is proved in detail  in \cite{JL00} and provides an
example of a  decay function.
\begin{pro}\label{exDecay}
Given $\alpha>N$, $\th \geq 0$, there exists $a>0$, depending on
$\alpha,\th$, $N$ such that  the function defined by
\begin{equation*}
\Gamma(i)=\left \{
\begin{array}[c]{ll}
a |i|^{-\alpha} e^{-\th |i|}\qquad &\mbox{if $\;i \neq 0$,} \\
a & \mbox{if $\;i = 0$}
\end{array}\right .
\end{equation*}
is a decay function on $\ZZ^N$.
\end{pro}

We note, as it is easily verified in \cite{JL00}, 
that $\Gamma(i) = C \exp( - \beta |i| )$ is not a decay function for 
any $\beta, C > 0$.

If one considers other sets $\Lambda$ in place of $\ZZ^N$,
such as the Bethe lattice which also admit decay functions
(see \cite{JL00}),
many of the results of the present paper can be adapted 
with little change. 

\begin{defi} \label{dominated} 
Given two decay functions $\Gamma, \Gamma'$ we say that 
$\Gamma$ dominates $\Gamma'$ and write 
$\Gamma' \ll  \Gamma$ when 
\[
\lim_{k \to \infty} \Gamma'(k)/\Gamma(k) = 0 .  
\]

We say that a family of decay functions 
$\Gamma_\beta$,  $\beta \in [0,1]$, is an ordered family 
when $\tilde \beta < \beta$ implies 
$\Gamma_{\tilde \beta} \ll \Gamma_\beta$. 
\end{defi}
Of course the examples in Proposition~\ref{exDecay}  
constitute an ordered family. For some of the arguments later, 
in the proof of Theorem~\ref{thgl},
when we are increasing the scales increasing the number of 
breathers, it will be useful to have a full scale so 
that the longer scales have a weaker decay. 
This is the reason why
Theorem~\ref{thgl} is only stated for these functions.

Of course, the examples in Proposition~\ref{exDecay} enjoy several
other nice properties, for example that $\Gamma(i)$ is a decreasing
function of $|i|$.   We refer the 
reader to Appendix \ref{decayfunctions} where  
a deeper study of spaces 
of decay functions is performed. In the following, we just give the
definitions needed to state our main results.

\begin{remark}
Prof. L. Sadun pointed out that there is a very natural 
physical interpretation of the definition of 
decay functions. We note that a site $i$ can affect 
another site $j$ either directly or by affecting another 
site  $k$ which in turn affects the site $j$. Of course, 
more complicated effects involving longer chains of 
intermediate sites are also possible. If the direct interaction between two sites is bounded 
by a decay function, it follows that the effect mediated 
through intermediate sites is bounded by the same function. 
This makes it possible to comfortably carry out
perturbation calculations. 
\end{remark} 

\subsubsection{Banach spaces of functions with good localization 
properties} 

We now introduce the functional spaces needed for our purposes.  We introduce:

\begin{itemize}
\item The Banach space of decay linear operators 
\begin{equation}
\mathcal{L}_\Gamma (\ell^\infty(\ZZ^N))= \left \{ 
\begin{array}{ccc}
A \in \mathcal{L}(\ell^\infty(\ZZ^N)) \,\,|\,\,\exists \,\,\left \{ A_{ij} \right  \}_{i,j \in \ZZ^N}, A_{i,j} \in \mathcal{L}(M)\\
\,\,(Au)_i =\sum_{j \in \ZZ^N} A_{ij}u_j, \,\,\,\,i \in \ZZ^N, \\
 \sup_{i,j \in \ZZ^N} \Gamma(i-j)^{-1}|A_{ij}| < \infty \end{array}\right \},
 \end{equation}
 where $\mathcal L(\ell^\infty(\ZZ^N))$ denotes the space of continuous linear maps from $\ell^\infty(\ZZ^N)$ into itself.  We endow $ \mathcal L_\Gamma (\ell^\infty(\ZZ^N))$ with the norm 
 \begin{equation}
 \|A\|_\Gamma= \sup_{i,j \in \ZZ^N} \Gamma(i-j)^{-1}|A_{ij}| 
 \end{equation}
 \item The space of $C^1$ functions on an open set $\mathcal B \subset \mathcal M$
 \begin{equation*}
C^1_\Gamma(\B)=\left \{
\begin{array}[c]{cc}
F: \B \to \M\mid\, F \in
  C^1(\B),\,\,DF(x) \in
  \L_\Gamma(\ell^\infty(\ZZ^N)), \, \, \forall x \in \B\\
\sup_{x \in \B} \|F(x)\| < \infty  ,\ \sup_{x \in \B}
\|DF(x)\|_\Gamma < \infty
\end{array} \right \}
\end{equation*}
with the norm 
\[
\|F\|_{C^1_\Gamma}=\max
\big(\sup_{x\in\B} \|F(x)\|\, ,\break \sup_{x\in\B} \|DF(x) \|_\Gamma\big).
\]

For $r \in \NN$, we define
\begin{equation*}
C^r_\Gamma(\B)=\left \{
\begin{array}[c]{cc}
F: \B \to \M\mid\, F \in
  C^r(\B),\, D D^{j-1}  F\in
  C^{1}_\Gamma(\B),\\
  0 \le j \le r=1
\end{array} \right \}.
\end{equation*}
Of course, we can give an equivalent recursive definition of the $C^r$ as 
the set of fucntions whose derivative is given by a matrix
valued function which is in $C^{r-1}$. 

We define a notion of analyticity for maps on 
lattices. 
\begin{defi}\label{analyt}
Let $\B$ be an open set of $\M$.
We say that $F: \B \rightarrow \M$ is analytic if
it is in $C^1_\Gamma (\mathcal B)$ with the derivatives understood in the 
complex sense. 
\end{defi}

\item The space of analytic embeddings on a strip 
\begin{equation*}
D_{\rho}=\left \{ z\in \mathbb{C}^l/\ZZ^l| \,\,|\mbox{Im}\,z_i|
  < \rho,\,\,i=1,\dots ,l\right \}.
\end{equation*}
Let $R \geq 1$ be an integer and consider $\underline{c} \in (\ZZ^N)^R$, i.e.
\begin{equation*}
\underline{c}=(c_1,\dots ,c_R).
\end{equation*}
We introduce the following quantity
\begin{equation*}
\|f\|_{\rho,\underline{c},\Gamma}=\displaystyle{\sup_{i \in \ZZ^N} \min_{j=1,\dots ,R}}
\Gamma^{-1}(i-c_j) \|f_i\|_{\rho},
\end{equation*}
where
\begin{equation*}
\|f_i\|_{\rho}=\sup_{ \th \in D_\rho} |f_i(\th)|.
\end{equation*}
We denote

\begin{equation}\label{espace}
\mathcal{A}_{\rho,\underline{c},\Gamma}= \left \{
\begin{array}[c]{cc}
f: D_\rho \to \M\mid f \in
C^0(\overline{D}_\rho),\,f\mbox{ analytic in} \,D_\rho,\\
\|f\|_{\rho,\underline{c}, \Gamma} < \infty
\end{array}\right \}.
\end{equation}

This space, with the norm $\|\cdot \|_{\rho,\underline{c},
\Gamma}$, is a Banach space. If we consider a map  $A$ from $D_{\rho}$ into the set of
linear maps $\L_\Gamma(\ell^\infty(\ZZ^N))$, the associated norm is
\begin{equation*}
\|A\|_{\rho,\Gamma}=\displaystyle{\sup_{i,j\in
  \ZZ^N}}\ \displaystyle{ \sup_{\th \in D_\rho}}\Gamma^{-1}(i-j)|A_{ij}(\th)| = \sup_{\th \in D_\rho} \|A(\th)\|_\Gamma
\end{equation*}

\end{itemize}

\subsection{Symplectic geometry on lattices}

In this section, we introduce the little geometry we need on the
manifold $\M$ to be able to perform the iteration. We refer the reader
to Appendix \ref{sec:symplectic} where a more systematic description
and properties of the objects is performed. We basically need
symplectic geometry for the KAM step on the center manilfolds -- which
will be finite dimensional --- and we will need the exactness
properties for the vanishing lemma \ref{vanishingloc} in Section
\ref{secvanishing}.  These uses can be accomplished by just saying that 
the pullback of the symplectic form by decay embeddings 
from a finite dimensional torus make sense. 
It is also very useful that the proof presented does 
not require transformation theory and, hence, we do not need 
to discuss a systematic theory of symplectic mappings.

Consider our finite dimensional exact symplectic manifold $(M,\Omega=d\alpha)$ and the associated lattice
\begin{equation*}
\M=\ell^\infty(\ZZ^N,M).
\end{equation*}

Define $\alpha_\infty$ and $\Omega_\infty$ to be the formal sums
(later, we will give them some precise meaning)
\begin{equation*}
\alpha_\infty=\sum_{j \in \ZZ^N} \pi^*_j\alpha, \qquad
\Omega_\infty=\sum_{j \in \ZZ^N} \pi^*_j\Omega,
\end{equation*}
where $\pi_j$ are the standard projections from $\mathcal M$ to $M$ at the node $ j \in \ZZ^N$. Let $J$ be the symplectic matrix associated to the symplectic
two-form $\Omega$ on $M$. We denote $ J_\infty$ the operator
defined on $T \mathcal M$ by
$$J_\infty (z)= {\rm diag} \big(\dots, J(\pi_iz),\dots\big), \qquad\ z\in\M\,.$$

We introduce the following definitions. 

\begin{defi}\label{Symploc}
We say that a $C^1_\Gamma$ function $F:\M \rightarrow
\M$ is symplectic if the following identity holds for any $z \in \mathcal M$
$$DF^\top(z) J_\infty (F(z)) DF(z)=J_\infty(z),$$
where the product of two operators $A$ and $B$ in $C^1_\Gamma$ is given component-wise by 
$(AB)_{i,j}= \sum_{k \in \Z^N}A_{ik}B_{kj},\quad i,j \in \ZZ^N.$
\end{defi}
Note that, due to the decay, properties, the products involved in the 
definition of a symplectic matrix are absolutely convergent sums.

Similarly, we have the following definition. Let $\hat A$ be the linear operator associated to the Liouville form $\alpha$ on $M$. We denote $ \hat A_\infty$ the operator
defined on $T \mathcal M$ by
$$\hat A_\infty (z)= {\rm diag} \big(\dots, \hat A( \pi_i z),\dots\big), \qquad\ z\in\M\,.$$

\begin{defi}\label{exSymploc}
We  say that a $C^1_\Gamma$ function $F:\M \rightarrow \M$ is
exact symplectic on $\M$ if  there exists a one-form $\tilde \alpha$ defined on $T \mathcal M$  with  matrix  $ \tilde A$ such that
\begin{itemize}
\item For every $j \in \ZZ^N$, there exists a smooth function $W_j$ on $M$ such that 
$$\tilde \alpha_j = dW_j,  $$ 
where $d$ is the exterior differentiation on $M$. 
\item The following formula holds component-wise on the lattice 
\begin{equation*}
DF(z)^\top \hat A_\infty(F(z)) =\hat A_\infty(z)+ \tilde A(z).
\end{equation*}

\end{itemize}
\end{defi}
The previous definitions are completely equivalent to the standard
definitions of symplectic and exact symplectic 
maps in the finite dimensional 
case, but they are among the mildest ones that 
we can imagine in infinite dimensions. 

We anticipate that the symplectic structure, will only 
enter in this paper in two places: 1) The automatic reducibility 
in the center directions, 2) The vanishing lemma to show 
that for exact symplectic mappings several averages vanish. 
These applications are very finite dimensional. 

The following lemma will be usefull for us (see Appendix 11). 
\begin{lemma}\label{sympDecay}
Consider a function $\psi$ defined on $\T^l$ (or a subset of it) with values in $\M$ and belonging to $\mathcal A_{\rho,\underline c, \Gamma} $ for some $\rho >0$. Then the bilinear form 
$$\psi^*\Omega_\infty= \sum_{j \in Z^N}\psi^* \Omega(\pi_j)$$
is a two-form on the torus $\TT^l$. 

\end{lemma}

\subsection{Diophantine properties} 
\label{sec:diophantine} 
KAM relies on approximation properties of the frequencies
by rational numbers.
In this section, we recall some well known  notions. For diffeomorphisms, the relevant notion of 
Diophantine properties 
is given by the following

\begin{defi}\label{rotVect}
Given $\kappa>0$ and $\nu \geq l$, we define $D(\kappa,\nu)$ as
the set of frequency vectors $\omega \in \real^l$ satisfying the
Diophantine condition:

\begin{equation*}
|\omega \, \cdot \,k-n|^{-1} \leq \kappa
|k|^{\nu},\,\,\,\,\,\,\mbox{for all $k\in\integer^l-\left \{
0\right \}$ and $n\in \integer$}
\end{equation*}
with $|k|=|k_1|+\dots +|k_l|$, where $k_i$ are the coordinates of
$k$.
\end{defi}

For vector fields, one uses the following
\begin{defi}\label{rotVF}
Given $\kappa>0$ and $\nu \geq l-1$, we define $D_h(\kappa,\nu)$
as the set of frequency vectors $\omega \in \real^l$ satisfying
the Diophantine  condition:
\begin{equation*}
|\omega \, \cdot \,k|^{-1} \leq \kappa
|k|^{\nu},\,\,\,\,\,\,\mbox{for all $k\in\integer^l-\left \{
0\right \}$},
\end{equation*}
where $|k|=|k_1|+\dots +|k_l|$.
\end{defi}

Given $f \in L^1(\torus^l)$, we denote
\begin{equation*}
\mbox{$\avg(f)$}=\int_{\torus^l}f(\th)d\th.
\end{equation*}

We also denote  by $T_\omega$ the  rotation on  
$\torus^l$  by  $\omega$ :
\begin{equation*}
T_\omega(\th)=\th+\omega.
\end{equation*}

In Section~\ref{sec:diophantineinfinite} we will discuss extensions of 
these definitions to infinite dimensional vectors which are well 
adapted to our applications.

\section{Formulation of the results}

We will
first obtain a translated tori result, i.e. a KAM theorem for
parameteri\-zed families of maps $F_\lambda$ which are symplectic
for all $\lambda$ and such that $F_0$ is {\sl exact} symplectic.
This will allow us to avoid the considerations of vanishing of 
averages at each stage of the iteration. Then, we will prove a 
simple vanishing lemma (see Section~\ref{secvanishing} )
that shows that the added extra parameter vanishes. 
This yields to the desired invariant tori theorem. Going through translated curve theorems has become quite standard 
in KAM theory (see \cite{Moser67, Russmann76})
especially since \cite{Sevryuk99} pointed out that it 
deals with very degenerate situations. In our case, it is 
particularly advantageous since the parameters we need are 
finite dimensional and it avoids many infinite dimensional considerations.

The problem is the following:  given an exact symplectic map $F$
and  a vector of frequencies $\omega \in D(\kappa,\nu)$ we wish to
construct an invariant torus for $F$ such that the dynamics of $F$
restricted on it is conjugated to  the translation $T_\omega$. To
this end, we search for an embedding $K: D_\rho \supset \torus^l
\rightarrow \M$ in
$\mathcal{A}_{\rho,\underline{c},\Gamma}$ such that for all
$\th \in D_\rho$, $K$ satisfies the functional equation
\eqref{invarmap1}.

Notice that if \eqref{invarmap1} is satisfied,
the image under $F$ of a point in the range of
$K$ will also be in the range of $K$.
If the range of $DK(\theta)$ is $l$-dimensional 
for all $\theta$, then $K(\T^l)$  is an $l$-dimensional 
invariant torus.
(Similarly, the geometric interpretation of 
\eqref{invarflow1} is that the vector field $X$ at a 
point in the range of $K$ is tangent to the range of $K$.)

The assumptions are that we are given a mapping $K$ that satisfies
\eqref{invarmap1} up to a very small error and that fullfills some
non-degeneracy assumptions. We prove that the embedding $K$ exists
and also that the solution is unique up to composition on the
right with translations.

Actually we are going to prove a more general result which works
for parameterized families of symplectic maps $F_\lambda$, such
that $F_0$ is exact symplectic, but only provides translated (and
not invariant) tori. That is, given $\omega\in D (\kappa,\nu)$
and an approximate solution $K$ of $F_{\lambda_0}\circ K-K\circ
T_\omega=0$ satisfying a set of non-degeneracy conditions, we
search for an embedding $K: D_\rho \supset \torus^l \rightarrow
\M$ in $\mathcal{A}_{\rho,\underline{c},\Gamma}$ such
that
\begin{equation}\label{translatedloc}
F_{\lambda} \circ K =K\circ T_{\omega}
\end{equation}
for some $\lambda$ close to $\lambda_0$. The geometric interpretation of the 
invariance equations is illustrated in 
Figure~1. 
\begin{figure}
\begin{center} 
\includegraphics[height = 2.0 in]{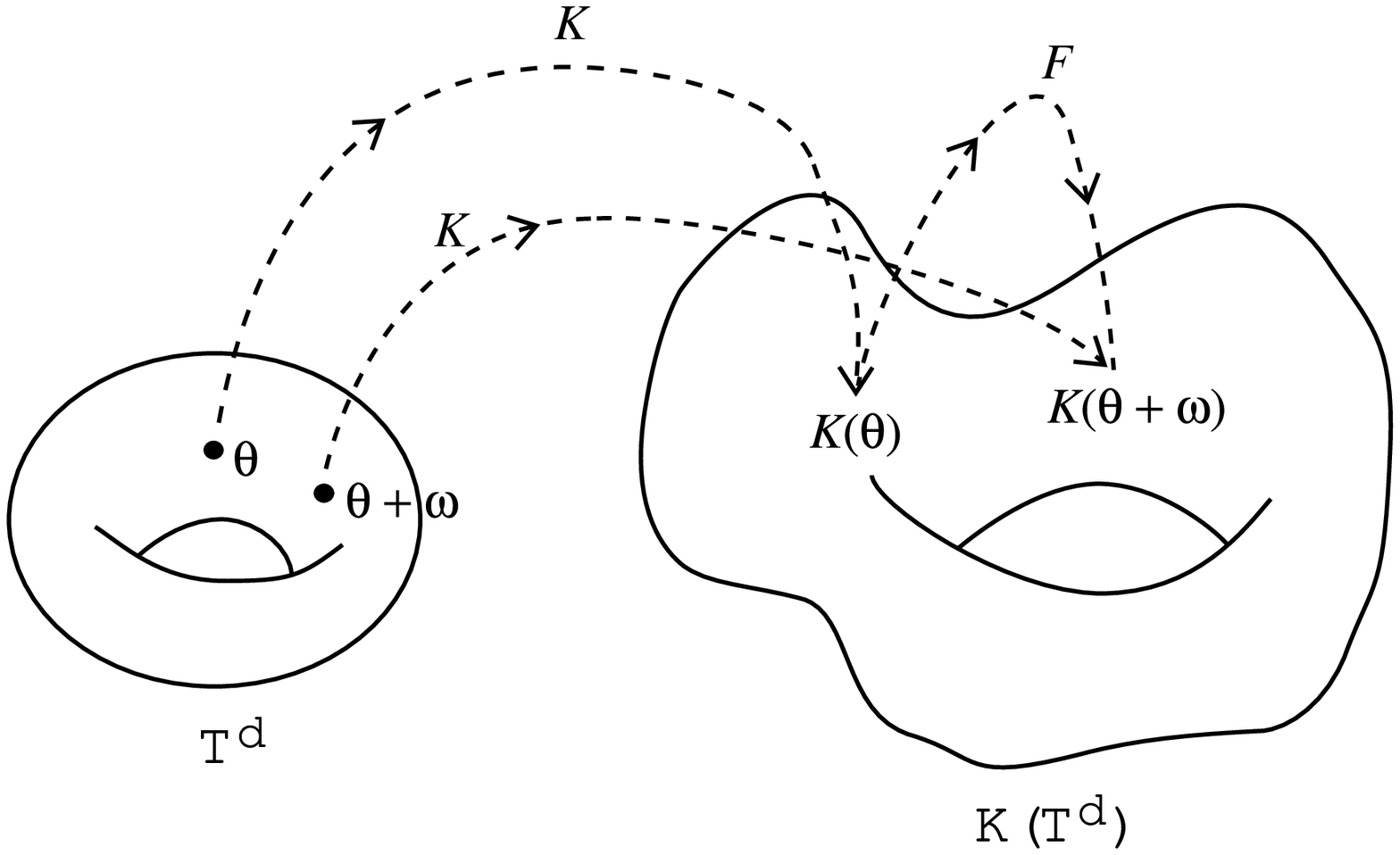}

\includegraphics[height = 2.0 in]{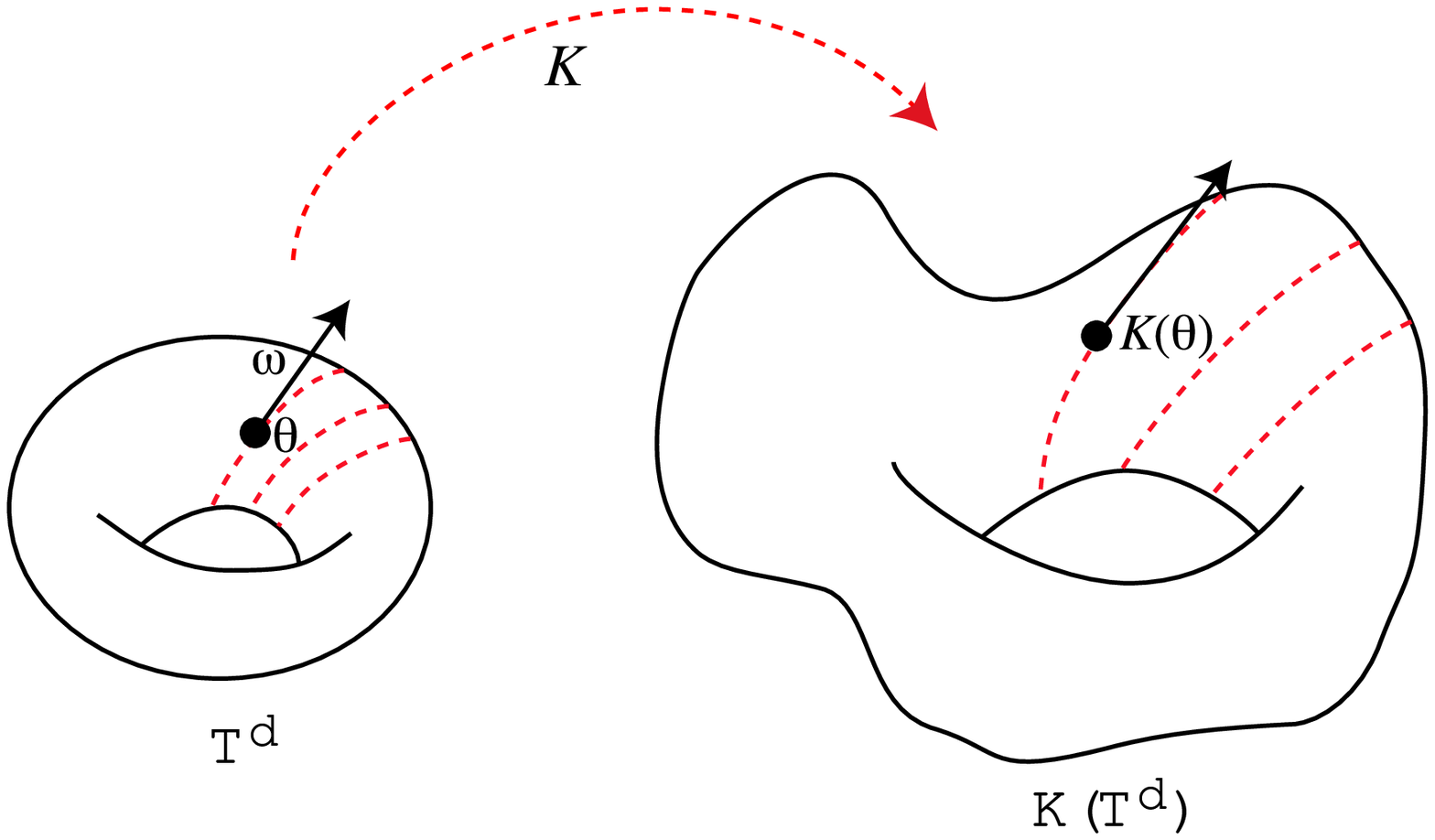}

\caption{Illustration of the invariance 
equations \eqref{invarmap1}, \eqref{invarflow1}. }
\end{center}
\end{figure} 

We go through a Newton scheme to prove the existence of such a
pair $(\lambda,K)$. To this end, we introduce the operator
$\mathcal{F}_{\omega}$
\begin{equation*}
\mathcal{F}_{\omega}(\lambda,K)=F_\lambda\circ K -K\circ
T_{\omega}.
\end{equation*}

In the paper \cite{FontichLS09}, the authors constructed 
invariant tori using {\sl a posteriori} KAM theorems
in finite dimensional systems. The general principles of this
method remain valid in some infinite dimensional systems such as
lattices. We first introduce some notations and several non-degeneracy
conditions.

\begin{defi}\label{NDloc-hyp} 
Consider $\rho>0,\ \omega\in\R^l,\ \underline c=(c_1,\dots, c_R)\in
(\Z^N)^R,\ R\ge 1,$ a decay function $\Gamma,\ \lambda \in\R^l$ and $F_\lambda :\M\to \M$ be a $C^2_\Gamma$ map. 

We say that $K:D_\rho\to \M \in \mathcal A_ {\rho,\underline{c}, \Gamma}$  is a whiskered embedding 
for $F_\lambda$ when 
we have:

The tangent space 
$T_{K(\th)}\M$ has an invariant analytic splitting for all $\th \in D_\rho$
\begin{equation}\label{splittingReseau}
T_{K(\th)}\M=\mathcal{E}^s_{{K(\th)}}\oplus
\mathcal{E}^c_{{K(\th)}}\oplus \mathcal{E}^u_{{K(\th)}}\,,
\end{equation}
where $\mathcal{E}^s_{{K(\th)}}$, $\mathcal{E}^c_{{K(\th)}}$ and
$\mathcal{E}^u_{{K(\th)}}$ are the stable, center and unstable
invariant spaces respectively, 
which satisfy: 
\begin{itemize} 
\item 
The projections $\Pi^s_{K(\th)}$, $\Pi^c_{K(\th)}$ and
$\Pi^u_{K(\th)}$ associated to this splitting are analytic with
respect to $\th$ considered as operators in $\L_\Gamma(\ell^\infty(\Z^N))$.

\item
The splitting \eqref{splittingReseau} is characterized by
asymptotic growth conditions (co-cycles over $T_{\omega}$): there
exist $0<\mu_1,\mu_2 <1$, $\mu_3 >1$ such that
$\mu_1 \mu_3 <1$, $\mu_2 \mu_3 <1$ and $C_h>0$
such that for all $n \geq 1$, $\th \in D_\rho$ and $\lambda \in \mathbb{R}^l$
\begin{equation}\label{ndeg1loc}
\begin{split}
\|DF_{\lambda}&\circ K\circ T^{n-1}_{\omega}\times\dots \times DF_{\lambda}\circ K
v\|_{\rho,\underline{c}, \Gamma} \leq C_h\mu_1^n \|v\|_{\rho,\underline{c},\Gamma}\\
&\iff v \in \mathcal{E}^s_{{K(\th)}}
\end{split}
\end{equation}
and
\begin{equation}\label{ndeg2loc}
\begin{split}
\|DF_{\lambda}^{-1}&\circ K\circ T^{-(n-1)}_{\omega}\times\dots \times
DF_{\lambda}^{-1}\circ K v\|_{\rho,\underline{c},\Gamma} \leq C_h\mu_2^n
\|v\|_{\rho,\underline{c},\Gamma}\\
&\iff v \in \mathcal{E}^u_{{K(\th)}}.
\end{split}
\end{equation}

\item
The center subspace
$\mathcal{E}_{K(\th)}^c$ is {\sl finite dimensional,} has
dimension $2l$ and it is characterized by: 
\begin{equation}\label{ndeg3loc}
\begin{split}
\|DF_{\lambda}&\circ K\circ T^{n-1}_{\omega}(\th)\times\dots \times
DF_{\lambda}\circ K(\th) v\|_{\rho,\underline{c},\Gamma} \leq C_h\mu_3^n \|v\|_{\rho,\underline{c},\Gamma}\\
\|DF_{\lambda}^{-1}&\circ K\circ T^{-(n-1)}_{\omega}(\th)\times\dots \times
DF_{\lambda}^{-1}\circ K(\th) v\|_{\rho,\underline{c},\Gamma} \leq C_h\mu_3^n \|v\|_{\rho,\underline{c},\Gamma}\\
&\iff v \in \mathcal{E}^c_{{K(\th)}}.
\end{split}
\end{equation}
\end{itemize} 
\end{defi}

It is important for applications that 
the spectral condition in Definition~\ref{NDloc-hyp} 
is implied by a condition 
that can be verified by 
a finite calculation (see Definition~\ref{NDloc2} below).
Approximate invariance of the splitting is sufficient (see Proposition~\ref{NDmoveloc} below) to ensure there is a truly invariant splitting. 
So,  the final version of our results will have 
as a hypothesis the existence of approximately invariant tori 
with approximately invariant splitting (Definition~\ref{NDloc2}). 
The final version of the results will 
have as a  conclusion the existence of exactly invariant tori 
with exactly invariant splittings (Definition~\ref{NDloc-hyp}).

\begin{defi}\label{NDloc2}
Consider $\rho>0,\ \omega\in\R^l,\ \underline c=(c_1,\dots, c_R)\in
(\Z^N)^R,\ R\ge 1,$ a decay function $\Gamma,\ \lambda \in\R^l$ and $F_\lambda :\M\to \M$ be a $C^2_\Gamma$ map. 

We say that  $\tilde K:D_\rho\to \M \in \mathcal A_ {\rho,\underline{c}, \Gamma}$  satisfies the $\eta$-hyperbolic condition (or has an $\eta-$invariant splitting), 
if there exists an analytic  splitting
of $T_{\tilde{K}(\torus^l)}\mathcal{M}$,
\begin{equation}\label{spscu}
T_{\tilde{K}(\th)}\mathcal{M}=\mathcal{E}^s_{{\tilde{K}(\th)}}\oplus
\mathcal{E}^c_{{\tilde{K}(\th)}}\oplus
\mathcal{E}^u_{{\tilde{K}(\th)}}
\end{equation}
such that, denoting
$\Pi^{s,c,u}_{\tilde K(\th)}$ be the corresponding projections, we have 
\begin{enumerate}
\item The splitting is approximately invariant under the co-cycle 
$DF \circ{ \tilde K}$ over $T_\omega$ in the sense that
$$
\dist \Big(DF_\lambda(\tilde K(\th)) \mathcal{E}^{s,c,u}_{\tilde
K(\th)}, \mathcal{E}^{s, c, u}_{\tilde K(\th + \omega )}\Big) <
\eta.$$
\item There exists $N\in \NN, 0<\tilde\mu_1, \tilde\mu_2< 1$ and $\tilde\mu_3 >
1$  such that $\tilde\mu_1, \tilde\mu_3 < 1,\ \tilde\mu_2
\tilde\mu_3 < 1$ and
\begin{equation}\label{ndeg1Iterloc}
\begin{split}
\|DF_{\lambda}&\circ \tilde{K}\circ
T^{N-1}_{\omega}(\th)\times\dots \times
DF_{\lambda}\circ \tilde{K}(\th) v\|_{\rho,\underline{c},\Gamma}
\leq \tilde{\mu}_1^N \|v\|_{\rho,\underline{c},\Gamma},\\
&\forall v \in \mathcal{E}^s_{{\tilde{K}(\th)}},
\end{split}
\end{equation}
\begin{equation}\label{ndeg2Iterloc}
\begin{split}
\|DF_{\lambda}^{-1}&\circ \tilde{K}\circ
T^{-(N-1)}_{\omega}(\th)\times\dots \times DF_{\lambda}^{-1}\circ
\tilde{K}(\th) v\|_{\rho,\underline{c},\Gamma}
 \leq \tilde{\mu}_2^N \|v\|_{\rho,\underline{c},\Gamma},\\
&\forall v \in \mathcal{E}^u_{{\tilde{K}(\th)}}
\end{split}
\end{equation}
and
\begin{equation}\label{ndeg3Iterloc}
\begin{split}
\|DF_{\lambda}&\circ \tilde{K}\circ
T^{N-1}_{\omega}(\th)\times\dots \times
DF_{\lambda}\circ \tilde{K}(\th) v\|_{\rho,\underline{c},\Gamma}
 \leq \tilde{\mu}_3^N \|v\|_{\rho,\underline{c},\Gamma}\\
\|DF_{\lambda}^{-1}&\circ \tilde{K}\circ
T^{-(N-1)}_{\omega}(\th)\times\dots \times
DF_{\lambda}^{-1}\circ \tilde{K}(\th) v\|_{\rho,\underline{c},\Gamma}
 \leq \tilde{\mu}_3^N \|v\|_{\rho,\underline{c},\Gamma}\\
&\forall v \in \mathcal{E}^c_{{\tilde{K}(\th)}}.
\end{split}
\end{equation}
\end{enumerate}
\end{defi}

\begin{remark} 
Note that in Definition~\ref{NDloc2} we are using that the 
phase space is Euclidean. On a general manifold, 
the products used in \eqref{ndeg1Iterloc}, 
\eqref{ndeg2Iterloc}, 
\eqref{ndeg3Iterloc} cannot be defined because, in general, 
$DF(x): T_x \M \rightarrow T_{F(x)}\M$. Hence, in a general manifold,
if 
$F\circ K(\th) \ne  K(\th + \omega)$, 
we cannot define $DF \circ K(\th + \omega) DF\circ K(\th)$. 
In \cite{FontichLS09}  one can find 
a definition of approximately invariant cocycles
for general manifolds. In this paper, we will not consider 
such generality. 
\end{remark}

We will define 
\[ 
\Omega^c_{K(\th)} = \Omega|_{\E^c_{K(\th)}} \quad \forall \theta \in \torus^\ell
\]
and we introduce the symplectic linear map $J^c(K(\th)): 
\E^c_{K(\th)} \rightarrow \E^c_{K(\th)}$ by 
\[
\Omega^c_{K(\th)}(u,v) = 
\langle u, J^c_{K(\th)} v \rangle \quad \forall u,v \in \E^c_{K(\th)}.
\]
Obviously, we have $J^c(K(\th))^\top = -J^c(K(\th))$. 
We also have (See Lemma~\ref{lem:restriction}) that 
$\Omega^c$ is non-degenerate, hence $J^c$ is invertible.

\begin{defi}\label{NDloc}
Given $\rho>0,\ \omega\in\R^l,\ \underline c=(c_1,\dots, c_R)\in
(\Z^N)^R,\ R\ge 1,$ a decay function $\Gamma,\ \lambda \in\R^l$ and
an embedding $K:D_\rho\rightarrow\M \in \mathcal A_{\rho,\underline{c},\Gamma}$, a pair $(\lambda,K)$ is said
to be non-degenerate (and we denote $(\lambda,K) \in
ND_{loc}(\rho,\Gamma)$) if it satisfies the following conditions
\begin{itemize}
\item {\sl Non degeneracy of the embedding: } 
We have that the $l \times l$ matrix
$
DK^{\top}(\th) DK(\th) 
$ 
is invertible for all $\th$ in $D_{\rho}$. 
We denote $N(\th) = \big( DK^\top(\th) DK(\th)\big)^{-1}$ 
and we assume that 
\[
\| N\|_{\rho,\Gamma} < \infty
\]
\item {\sl Twist condition:} let $P(\th)= DK(\th)N(\th)$. 

The average
  on $\torus^l$ of the $l \times l-$matrix
\begin{equation}\label{Adefined}
A_\lambda(\th)=P(\th+\omega)^\top\Big
([DF_{\lambda}(K) (J^c\circ K)^{-1}P](\th) - [(J^c\circ K)^{-1}P](\th+\omega)\Big
)
\end{equation}
is non-singular.

\item{\sl Parameter cohomological non-degeneracy:} 
The average
  on $\torus^l$ of the $l \times l-$matrix
\begin{equation}
\label{Qdefined}
Q_{\lambda}(\th)=
\Big((DK^\top(\omega+\th) J^c(K(\omega+\th)) \frac{\partial F_\lambda (K(\th))}{\partial \lambda}\Big)
\end{equation}
is non-singular.

\end{itemize}
\end{defi}

It is clear that  the meaning of $\| N\|_{\rho,\Gamma}$ is a measure of 
the quality of the embedding. It grows if the embedding comes close 
to having a singularity. 
During the proof  it will become clear that the meaning of 
$A_\lambda$ is the change of the rotation when we move in the 
direction transversal to the torus. 
As we will see in calculations, the meaning of 
the invertibility of the average of $Q_\lambda$ is that, 
by changing $\lambda$, we can adjust the obstructions to the 
cohomology equations.

For applications, it is important to note that the non-degeneracy
hypothesis only depend on the approximate solution considered and 
that they are readily computable algebraic expressions. 
They are quite analogous to the condition numbers in numerical 
analysis.
\medskip

First we state our main theorem, which provides the existence of a
solution $(\lambda,K)$ to the functional equation
\eqref{translatedloc}. This is the translated tori KAM theorem.

\begin{thm}\label{existencetranslatedloc}
Let $F_{\lambda}:\M \rightarrow \M$ be a family
of symplectic maps parameterized by $\lambda \in \mathbb{R}^l$,
$\omega \in D(\kappa,\nu)$ for some $\kappa>0, \nu\ge l$,
$\rho_0>0,\ \Gamma$ a decay function and
$\underline{c}=(c_1,\dots , c_R) \in (\ZZ^N)^R$. Assume we have
$\lambda_0\in\R^l$ and $K_0:D_\rho \supset \T^l\rightarrow\M$
satisfying the following hypotheses
\begin{itemize}
\item
For all $\lambda \in
\RR^l$, the maps $F_\lambda$ belong to $C^2_\Gamma$ and satisfy
$\sup_{i \in \ZZ^N}\Gamma^{-1}(i)(F_\lambda(0))_i <\infty$. 

\item The map $F_{\lambda}$ is real analytic and it
can be extended holomorphically to some complex neighborhood of
the image under $K_0$ of $D_{\rho_0}$:
\begin{equation*}
B_r=\left \{ z \in \M |\; \exists \th \quad s.t. \quad
|\rm{Im }\,\, \th| < \rho_0 \ , \ |z-K_0(\th)| <r \right \},
\end{equation*}
for some $r>0$ and
such that $\|DF_{\lambda}\|_{C^2_\Gamma(B_r)}$ is finite.

\item $(\lambda_0,K_0) \in ND_{loc}(\rho_0,\Gamma)$ i.e ,  the
embedding $K_0$ is  non-\-dege\-nerate in the 
sense of Definition~\ref{NDloc}. 

\item
The embedding $K_0$ is $\eta_0$-hyperbolic 
in the sense of Definition~\ref{NDloc2} with $\eta_0$ sufficiently small 
(depending on $\| \Pi^{s,c,u}\|_{\rho_0, \underline{c}, \Gamma}$, 
$\mu_{1,2,3}$, $N$, $\| F\|_{C^2_\Gamma(B_r)}$). 
\end{itemize}

Define the error $E_0$ by
\begin{equation*}
E_0=F_{\lambda_0}\circ K_0-K_0\circ T_{\omega}.
\end{equation*}
Denote also 
\def\tep{{\tilde \ep}}
\[
\tep = \max( \| E_0 \|_{\rho_0, \underline c, \Gamma}, \eta_0). 
\]

There exists a constant $C>0$ depending on $l$, $\kappa$, $\nu$,
$\rho_0$, $\|DF_{\lambda}\|_{C_\Gamma^2(B_r)}$,
$\|DK_0\|_{\rho_0,\underline{c},\Gamma}$, $\|N_0\|_{\rho_0}$,
$\|\frac{\partial F_{\lambda}(K_0)}{\partial
  \lambda}\|_{\rho_0,\underline{c},\Gamma}$,
$\|A_{\lambda_0}^0\|_{\rho_0}$, $|\avg (A_{\lambda_0}^0)|^{-1}$,
$|\avg (Q_{\lambda_0})|^{-1}$ (where $A_{\lambda_0}^0$,
$Q_{\lambda_0}^0$ and
$N_0$ are as in Definition \ref{NDloc},
replacing $K$ with $K_0$)
and on $\|\Pi^{c,s,u}_{K_0(\th)}\|_{\rho_0,\Gamma}$ such that, if
for some $\delta$, 
 $0 < \delta <\min(1,\rho_0/12)$, we
have the following conditions satisfied  
\begin{equation*}
C\kappa^4 \delta^{-4\nu} 
\tep < 1
\end{equation*}
and
\begin{equation*}
C\kappa^2 \delta^{-2\nu} \tep <r
\end{equation*}

then, there exist
an embedding $K_{\infty} \in
ND_{loc}(\rho_{\infty}=\rho_0-6\delta,\Gamma)$ and a vector
$\lambda_{\infty} \in \mathbb{R}^l$ such that
\begin{equation}\label{transSolloc}
F_{\lambda_{\infty}}\circ K_{\infty}=K_{\infty}\circ T_{\omega}.
\end{equation}
Furthermore, we have the following estimates
\begin{equation}\label{changeestimates}
\begin{split}
&\|K_{\infty}-K_0\|_{\rho_{\infty},\underline{c},\Gamma} \leq C
\kappa^{2} \delta^{-2\nu} \tep,\\
&|\lambda_0-\lambda_{\infty}| < C \kappa^{2} \delta^{-2\nu}
\tep . 
\end{split}
\end{equation}

Additionnally, we have that the invariant embedding $K_\infty$ 
admits invariant splittings, satisfying Definition \eqref{NDloc}. 

Denoting the non-degeneracy constants corresponding to 
$K_\infty$ by index $\infty$, we have: 
\begin{equation} \label{hyperbolicitychange} 
\begin{split}
& \| \Pi^{s,c,u}_\infty \circ K_\infty 
-  \Pi^{s,c,u} \circ K_0  \|_{\rho_\infty,  \Gamma} 
\le
C \kappa^{2} \delta^{-2\nu}\tep, \\
& |\mu_{1,2,3}^\infty - \mu_{s,c,u} | \le
C \kappa^{2} \delta^{-2\nu}\tep
\end{split} 
\end{equation} 
and
\begin{equation} \label{degeneracychange} 
\begin{split}
&|\| N_0\|_{\rho_0} - \| N_\infty\|_{\rho_\infty} | \le 
C \kappa^{2} \delta^{-2\nu}\tep, \\
&|\| A_{\lambda_0}^0\|_{\rho_0} - \| A_{\lambda_\infty}^\infty\|_{\rho_\infty} | \le 
C \kappa^{2} \delta^{-2\nu}\tep ,\\
&|\| Q_{\lambda}^0\|_{\rho_0, \Gamma} - \| Q_{\lambda_\infty}^\infty\|_{\rho_\infty, \Gamma} | \le 
C \kappa^{2} \delta^{-2\nu}\tep .\\
\end{split}
\end{equation} 
\end{thm}

The previous theorem will allow us to construct increasingly complicated 
solutions, the solutions of one stage being an approximate 
solution for the next stage. We will however be able to maintain 
enough control of the non-degeneracy conditions. 

Of course, \eqref{degeneracychange} is an easy consequence of 
\eqref{changeestimates} since the objects that enter 
in the degeneracy estimates are  algebraic expressions of 
$K$.

We now come to the result on the existence of invariant tori. They
correspond to localized quasi-periodic orbits on the manifold
$\M$. These orbits are known as ``breathers".

\begin{thm}\label{existenceembedloc}
Let $F_{\lambda}:\M \rightarrow \M$ be a family
of symplectic maps parameterized by $\lambda \in \mathbb{R}^l$,
$\omega \in D(\kappa,\nu)$ for some $\kappa>0, \nu\ge l$,
$\rho_0>0,\ \Gamma$ a decay function and
$\underline{c}=(c_1,\dots , c_R) \in (\ZZ^N)^R$. Assume we have
$\lambda_0\in\R^l$ and $K_0:D_\rho \supset \T^l\rightarrow\M$
satisfying the following hypotheses
\begin{itemize}
\item The map $F_{\lambda_0}$ is exact symplectic and $F_{\lambda_0}(0)=0$. 
\item
For all $\lambda \in
\RR^l$, the maps $F_\lambda$ belong to $C^2_\Gamma$ and satisfy
$\sup_{i \in \ZZ^N}\Gamma^{-1}(i)(F_\lambda(0))_i <\infty$. 
\item The map $F_{\lambda}$ is real analytic and it
can be extended holomorphically to some complex neighborhood of
the image under $K_0$ of $D_{\rho_0}$:
\begin{equation*}
B_r=\left \{ z \in \M |\; \exists \th \quad s.t.  \quad
|\rm{Im }\,\, \th| < \rho_0 \ , \ |z-K_0(\th)| <r \right \},
\end{equation*}
for some $r>0$ and
such that $\|DF_{\lambda}\|_{C^2_\Gamma(B_r)}$ is finite.

\item $(0,K_0) \in ND_{loc}(\rho_0,\Gamma)$ i.e ,  the
embedding $K_0$ is  non-\-dege\-nerate in the 
sense of Definition\ref{NDloc}. 

\item
The embedding $K_0$ is $\eta_0$-hyperbolic 
in the sense of Definition~\ref{NDloc2} with $\eta_0$ sufficiently small 
(depending on $\| \Pi^{s,c,u}\|_{\rho_0,  \Gamma}$, 
$\mu_{1,2,3}$, $N$, $\| F\|_{C^2_\Gamma(B_r)}$). 
\end{itemize}
Define the error $E_0$ by
\begin{equation*}
E_0=F_{\lambda_0}\circ K_0-K_0\circ T_{\omega}.
\end{equation*}
Denote also 
\def\tep{{\tilde \ep}}
\[
\tep = \max( \| E_0 \|_{\rho_0, \underline c, \Gamma}, \eta_0). 
\]

There exists a constant $C>0$ depending on $l$, $\kappa$, $\nu$,
$\rho_0$, $\|DF_{\lambda}\|_{C_\Gamma^2(B_r)}$,
$\|DK_0\|_{\rho_0,\underline{c},\Gamma}$, $\|N_0\|_{\rho_0}$,
$\|\frac{\partial F_{\lambda}(K_0)}{\partial
  \lambda}\|_{\rho_0,\underline{c},\Gamma}$,
$\|A_{0}^0\|_{\rho_0}$, $|\avg (A_{0}^0)|^{-1}$,
$|\avg (Q_{0})|^{-1}$ (where $A_{0}^0$,
$Q_{0}^0$ and
$N_0$ are as in Definition \ref{NDloc},
replacing $K$ with $K_0$)
and on $\|\Pi^{c,s,u}_{K_0(\th)}\|_{\rho_0,\underline{c},\Gamma}$ such that, if
for some $\delta$, 
 $0 < \delta <\min(1,\rho_0/12)$, we
have the following conditions satisfied  
\begin{equation*}
C\kappa^4 \delta^{-4\nu} 
\tep < 1
\end{equation*}
and
\begin{equation*}
C\kappa^2 \delta^{-2\nu} \tep <r
\end{equation*}

Then, we have in \eqref{transSolloc}
$$\lambda_\infty = 0,$$ i.e. the torus $K_\infty$ is actually an invariant 
torus for $F_{\lambda_0}$ and we have 
$$F_{\lambda_0} \circ K_\infty = K_\infty \circ T_\omega .$$
\end{thm}

\begin{remark} 
It is important to mention that we do not assume that the 
symplectic forms are the standard ones. This allows 
to consider the existence of external magnetic fields 
and magnetic interactions among the sites since the effect of a magnetic field is just a change of 
the symplectic form \cite{Thirring97}. Alternatively, if $(p,q)$ are the conjugated coordinates, 
one can change $p \to p - A$ where $A$ is the vector 
potential. Note that the introduction of a magnetic field 
destroys the reversibility under the usual involution
$S(p,q) = (-p,q)$. 
\end{remark}

\begin{remark} 
The whiskered tori that satisfy the
spectral hypothesis have invariant manifolds
that make them important in problems of stability.  However, the proof of the stable manifold is not completely 
straightforward since the space $\ell^\infty$ does
not have smooth cut-off functions. It is possible 
to show that these invariant manifolds have also some decay 
properties. This has been established in 
\cite{FLM2}. 
\end{remark}

We have also the following result which provides local uniqueness.

\begin{thm}\label{uniquenessLattice}
Let $\omega \in D(\kappa,\nu)$ for some $\kappa>0, \nu>l$ and
$K_1\in ND(\rho)$ and $K_2 \in ND(\rho)$ be two solutions of
equation \eqref{invarmap1} such that $K_1(D_{\rho}) \subset
B_r,\,K_2(D_{\rho}) \subset B_r$. There exists a constant $C>0$
depending on $l$, $\kappa$, $\nu$, $\rho$, $\rho^{-1}$,
$\|F \|_{C^2_\Gamma}$, $\|K_1\|_{\rho,\underline{c},\Gamma}$,
$\|N_1\|_{\rho,\Gamma}$, $\|A_1\|_{\rho,\Gamma}$, $|\avg(A_1)|^{-1}$ such
that if $\|K_1\circ T_\tau-K_2\|_{\rho,\underline{c},\Gamma}$
satisfies for some $\tau\in \mathbb{R}^l$
\begin{equation*}
C \kappa^2 \delta^{-2\nu} \|K_1\circ
T_\tau-K_2\|_{\rho_0,\underline{c},\Gamma} \leq 1 ,
\end{equation*}
where $\delta=\rho/4$, then there exists a phase $\tilde{\tau} \in
\mathbb{R}^l$ such that $K_1 \circ T_{\tilde\tau}=K_2$ in
$D_{\rho}$. Moreover, 
$$
|\tau-\tilde{\tau}| \le C\kappa^2\rho^{-2\nu}
\| K_1 \circ T_{\tau} - K_2\|_{\rho_0, {\underline{c}} \Gamma} .
$$
\end{thm}

\begin{remark}
It is important to remark that ALL  constants in the previous theorems are independent of $\underline{c}$. This fact is crucial for the next result, which provides an existence theorem for almost-periodic functions.
\end{remark}

The idea of the construction follows the one in the finite dimensional
case (see \cite{FontichLS09}). Notice here that the decay properties are on
the hyperbolic subspace, the center one being finite dimensional.
Some small differences between the scheme of the 
present paper and \cite{FontichLS09}  are detailed in 
Remark~\ref{changes}. 

We also have the analogous result to Theorem
\ref{existenceembedloc} for vector-fields, Theorem
\ref{existenceHamillattice}. We will postpone the statement of 
Theorem~\ref{existenceHamillattice} till 
Section \ref{secflows}  where we also present a proof.

As an application of Theorem~\ref{existenceHamillattice}, 
we will present a result on existence of 
solutions with  infinitely many frequencies
(also called \emph{almost periodic} solutions). 
We note that, as indicated before, we will establish the 
theorem  in two 
stages. In a first stage, we will 
 continue the breathers from the uncoupled system 
to the whole system. In the second stage, we will 
couple infinitely many of these breathers so that 
we obtain solutions with 
infinitely many frequencies. We note that the smallness conditions 
and the elimination of a positive measure set of frequencies only occurs in 
the first stage. In the second stage, we only need to eliminate 
a zero measure set of frequencies (in many different measures) 
and we do not need any smallness condition. The reason is 
that in the second stage, we adjust all the smallness conditions 
by placing the individual breathers far enough. Of course, if 
we wanted to let the breathers not to be so far appart, it could be 
done with other assumptions. 

The models we consider \eqref{coupledlattice} 
have been considered in the Physics and Mathematics 
literature. They are models of many microscopic processes. 
See
\cite{Braun, BraunK98,DauxoisRAW02,FloriaBG05,ChazottesF05,
GallavottiFPU, BilkiENW07} and references there among many others.

A large variety  of solutions for equations of this type have been
constructed: space-localized periodic in time solutions, known as
breathers (see \cite{MKA94,James01}), solitary waves
(\cite{Iooss00,IoossK00, FW94}, \cite{pego1,pego2,pego3,pego4}),
pulsating traveling waves (see \cite{JS05,Sire05}). The relevance
of these solutions in biological phenomena has also been discussed
(see \cite{DauxoisPW92,PS04,Peyrard04}).

There are already several other papers that have produced solutions
with infinitely many frequencies. The paper \cite{FrolichSW86}
produced such solutions by introducing some random terms and making
the excitation of each oscillator goes to zero, so that its effect on the others was small. Fr\"olich, Spencer and Wayne also assume that
the coupling is high order in terms of the amplitude. The paper
\cite{ChierchiaP95,Perfetti03} considered oscillators but made the
natural frequencies increase very fast so that there were no
resonances in each of them. The paper \cite{Poschel90} proved a very
abstract theorem that applies to perturbations of integrable systems
and managed to recover several results as applications of this
theorem. The paper \cite{GengY07} also considers coupled systems in
one dimension, but produces tori with finitely many frequencies.

The solutions we construct are based on a different
principle. We use that the solutions which are far apart interact very
weakly even if they are large. Therefore,
by placing solutions far apart, we will be able to
make them interact weakly and we can satisfy 
the smallness conditions assumed by the general theorem.
Notice that we are assuming that most of the sites
are close to a  hyperbolic orbit. Hence, the system 
will be very hyperbolic. This will allow us to deal 
with most of the normal directions using the methods of 
hyperbolic splittings and we will not need to consider 
the resonances that appear in the normally elliptic modes, which require more delicate estimates. 
We emphasize that we do not assume
 that the system is close to integrable.

\begin{thm}\label{thgl}
Consider a lattice  $\M = M^{\ZZ^N}$ with the symplectic form
 given by
$\Omega_\infty=\sum_{n \in \ZZ^N} dq_n \wedge dp_n$. Consider the
following Hamiltonian with respect to $\Omega_\infty$ given by
\begin{equation}\label{coupledlattice}
H(q,p)=\sum_{n \in \ZZ^N}\Big( \frac{1}{2}\,
|p_n|^2+W(q_n)\Big)+\ep \sum_{j\in \Z^N}\ \sum_{n \in \ZZ^N}
V_j(q_n-q_{n+j}).
\end{equation}

Let $\Gamma, \Gamma'$ be decay functions 
as in Proposition~\ref{exDecay}. 

Denote by $X$ the vector field associated to the 
Hamiltonian~\eqref{coupledlattice}. 

Assume:
\begin{enumerate}
\item[{\bf H1}] The system $\ddot{q}+W'(q)=0$ admits a hyperbolic fixed point,
which we will set without loss of generality at $q = 0$. 
\item[{\bf H2}] There exists a set $\Xi_0 \subset \real^l$ 
of positive Lebesgue measure such
that for all $\omega \in \Xi_0$, there exists a KAM torus 
invariant under the flow of $\ddot{q}+W'(q)=0$ and non-degenerate in the
sense of the standard KAM theory (twist condition).
\item[{\bf H3}] The potentials $V_j$ and $W$ are 
real analytic.  Moreover, we assume that there exists a
 constant $C_V$ such that 
\[
\| V_k\|_{C^2,\rho} \leq C_V \Gamma(k).
\]
and also that $\nabla V_k(0)=0$ for every $k$. 
\end{enumerate}

Fix $\rho' < \rho$. 
Then,
\begin{enumerate} 
\item{A)}For all $\ep^*$ sufficiently small, we can find 
a set $\Xi_1 (\ep^*) \subset \Xi_0$, 
such that if $\omega \in \Xi_1(\ep^*)$ and $|\ep| < \ep^*$, 
then the system \eqref{coupledlattice} has a localized breather
of frequency $\omega$. There exists 
$K: \torus^l \rightarrow \M$, $K  \in A_{\rho', \underline c, \Gamma}$ such that 
\[
X\circ K = \partial_\omega K 
\]
The embedding satisfies Definition~\ref{NDloc-hyp}
and we can choose the hyperbolicity  and non-degeneracy 
constants uniformly.

Furthermore we have 
$$\meas({ \Xi_0 \setminus \Xi_1(\ep^*)}) \to 0$$ as 
$\ep^* \to 0$. 

\item{B)}
Consider now $\Xi_\infty = \Xi_1(\ep^*)^\nat$ endowed with
the probability
measure $  \Big ( \frac{\meas(\cdot)}{\meas({\Xi_1(\ep^*)}} \Big )^\nat$. 
Then, there exists a set $\Xi_\infty^* \subset \Xi_\infty$,
$\meas( \Xi_\infty^*) = 1$, 
such that if $\underline{\omega} \in \Xi^*_\infty$, there exist a sequence of centers $\underline{c}$ 
and a $K$ analytic in some strip of $(\TT^l)^{\mathbb N}$ 
so that 
\[
J_\infty \nabla H \circ K = \partial_{ \underline{\omega}}K.
\]
\end{enumerate}
\end{thm}

Note that we have stated Theorem~\ref{thgl} only for decay functions of 
the form given in Proposition~\ref{exDecay}. It is clear that the proof only uses a few properties of 
the function (e.g. monotonicity in the modulus of the argument). We have 
refrained from reformulating the theorem in more abstract terms. 

We note that the only smallness conditions in $\ep$ enter 
just in the first stage of creating individual breathers
around each site and in the preservation of 
the hyperbolic structure and other non-degeneracy 
conditions. The second stage, on the other hand does 
not require any other smallness conditions. 

In the construction of infinite dimensional breathers out of
single breathers we 
just need to exclude a
few  sequences of frequencies 
which are very resonant (they have measure zero in the 
probability  measure indicated above). See Section~\ref{sec:coupling}. 
The smallness assumptions that we need to couple the sequences 
can be adjusted just by placing the different breathers far apart
and we do not need any further smallness conditions in $\ep$. 

Note also that the only hypothesis on the one site system is 
the existence of positive measure of KAM tori (and the existence of 
a hyperbolic fixed point). This is implied if 
the system is close to a non-degenerate integrable system. 
Nevertheless, there are other arguments to show existence of 
KAM tori in systems very far from integrable \cite{Duarte94, Duarte08}. 
Any of these systems could be taken as the basis for 
Theorem~\ref{thgl}.

\section{The Newton step}
The sketch of the proof of Theorem \ref{existencetranslatedloc}  is roughly  the same as for finite dimensional systems, with some minor changes 
detailed in Remark~\ref{changes}.  Of course, even if 
the strategy is similar to that in finite dimensions, 
all the details need to be different since the situation is 
very different and we need to pay attention to the decay 
properties. With a view to applications to almost periodic solutions of Theorem~\ref{thgl}, we also need to pay attention 
to the change in the non-degeneracy conditions and in the 
hyperbolicity properties and establish that many of 
the smallness assumptions are independent of the number 
and the geometry of the centers of oscillation. 

The proof of 
Theorem~\ref{existencetranslatedloc} 
is based on a Newton iteration of  Nash-Moser 
type. The estimates of the Newton step -- including uniqueness --
are summarized in 
Section~\ref{sec:inductive} (See, Lemma~\ref{main}). 
The fact that the inductive step  can be iterated is 
more or less standard in KAM theory and it is done in 
Section~\ref{convergence}. 

The proof of the estimates of the Newton step are 
obtained in different stages
\begin{enumerate}
\item
We show that the approximate invariant hyperbolic 
splitting can be transformed in an invariant splitting. 
See Section~\ref{sec:hyperbolic}.  
\item
The equations for the Newton step can be divided 
into equations along the  
hyperbolic spaces (studied in Section~\ref{sec:hyperbolicequation})
and the center space (studied in Section~\ref{sec:centerequation}).

As usual in the study of cohomology equations, the equations 
in the center direction are much more subtle. In particular, 
the Diophantine properties and the geometric properties are
only used in the equations in the center space.

\item
Once we have the estimates for the approximate solutions of 
the linearized equation, we show that, using the 
linearized equation, they improve the solutions of 
the translated equation. Furthermore, we estimate the 
changes in the hyperbolicity constants and the non-degeneracy
estimates. 

\item 
The passage from Theorem~\ref{existencetranslatedloc}
to Theorem~\ref{existenceembedloc} is  a geometric 
argument (a vanishing lemma) undertaken in 
Section~\ref{secvanishing}.
\end{enumerate}

\subsection{Estimates for the inductive step}
\label{sec:inductive} 
In this section we describe the inductive step of the proof
of Theorem~\ref{existencetranslatedloc}.

By Taylor's theorem we can write
$$
\mathcal{F}_{\omega}(\lambda+\Lambda,K+\Delta) =
\mathcal{F}_{\omega}(\lambda,K)  +
D_{\lambda,K}\mathcal{F}_{\omega}(\lambda,K)(\Lambda,\Delta)+
O(|(\Lambda,\Delta)|^2).
$$
Assuming that $(\lambda,K)$ is a pair that satisfies
$\F_\omega(\lambda,K)=0$ approximately with an error
$E(\th)=\F_\omega (\lambda,K)(\th)$ we look for
$(\Lambda,\Delta)$ such that
$\F_\omega(\lambda+\Lambda,\,K+\Delta)$ is as small as possible.
Then we are lead to consider the following Newton equation
\begin{equation}\label{linear}
D_{\lambda,K}\mathcal{F}_{\omega}(\lambda,K)(\Lambda,\Delta)=-E,
\end{equation}
where
\begin{equation*}
\begin{split}
D_{\lambda,K}&\mathcal{F}_{\omega}(\lambda,K)(\Lambda,\Delta)(\th)
=\frac{\partial F_\lambda(K(\th))}{\partial \lambda}\Lambda +
DF_{\lambda}(K(\th))\Delta(\th)-\Delta(\th+\omega).
\end{split}
\end{equation*}
To solve \eqref{linear} we project the equation on both the center
and the hyperbolic subspaces, taking advantage of the invariant
splitting. Then we try to solve the projected equations. The one
on the center subspace is reduced to two small divisors equations,
essentially one on the tangent of the torus and the other on its
conjugated directions. Taking advantage of 
the extra variable $\lambda$,  we can
solve these equations up to a quadratic error. Using the conditions on the co-cycles over $T_\omega$, we
solve the projection on the stable and unstable subspaces.

The next result gives an approximate solution of \eqref{linear}
with precise estimates. 
\begin{lemma}\label{main}
Under the hypotheses of Theorem \ref{existencetranslatedloc} the
equation
\begin{equation*}
D_{\lambda,K}\mathcal{F}_\omega(\lambda,K)(\Lambda,\Delta)=-E
\end{equation*}
has an approximate solution $(\Lambda,\Delta)$ in the following sense: let 
$$\tilde E=D_{\lambda, \kappa}\,\F_\omega(\lambda, K)\,(\Lambda, \Delta)+E\,.$$
For $0<\delta <\rho$ we have the following estimates
\begin{equation*}
\|\Delta\|_{\rho-\delta,\underline{c},\Gamma} \leq C \kappa ^2
\delta ^{-2\nu} \|E\|_{\rho,\underline{c},\Gamma},
\end{equation*}
\begin{equation*}
|\Lambda| \leq C \|E\|_{\rho,\underline{c},\Gamma},
\end{equation*}
\begin{equation*}
\|\tilde{E}\|_{\rho-\delta,\underline{c},\Gamma} \leq C \kappa ^2
\delta^{-(2\nu+1)} \|E\|_{\rho,\underline{c},\Gamma}
\|\mathcal{F}_\omega(\lambda,K)\|_{\rho,\underline{c}, \Gamma}.
\end{equation*}
Moreover, if $\Delta$ and $\tilde\Delta$ are solutions of
\eqref{linear} as above, i.e. solutions with quadratic error
bounded by $C\kappa^2\delta^{-(2\nu+1)} \|E\|_{\rho,\underline{c},\Gamma}\ \|\F_\omega\|_{\rho,\underline{c},\Gamma}$,
\begin{equation*}
\|\Delta-\tilde{\Delta}-DK(\th)\alpha\|_{\rho-\delta,\underline{c},\Gamma} \leq C \kappa
^2 \delta ^{-(2\nu+1)} \|E\|_{\rho,\underline{c},\Gamma}
\|\mathcal{F}_\omega(\lambda,K)\|_{\rho,\underline{c},\Gamma}.
\end{equation*}
In the previous estimates, the constant $C$ depends on $\rho, l,
\|DK\|_{\rho,\underline{c},\Gamma},$
$\|\Pi^{s,c,u}_{K(\th)}\|_{\rho,\Gamma},$ $ \|\frac{ \partial
  F_\lambda}{\partial \lambda}\|_{\rho,\underline{c},\Gamma}$, the hyperbolicity
constants and the decay function $\Gamma$ but it does not depend on $\underline c$. 
\end{lemma}

\subsection{Construction of invariant splittings out of approximately 
invariant ones}
\label{sec:hyperbolic} 

The main result of this section will be 
Proposition \ref{NDmoveloc},  which establishes that
given an approximately invariant splitting satisfying 
Definition~\ref{NDloc-hyp}, there is a truly invariant 
splitting nearby. Furthermore, we can estimate the distance 
between the true invariant splitting and the 
approximately invariant one. This, of course, 
implies the usual formulation of persistence of 
splittings under small perturbations. 

The way that this result fits into the Newton scheme 
is that this will allow us to split the equation into 
different components. Compared to other estimates 
in the Newton step, the construction of invariant 
splittings  requires much less sophisticated analysis
(it suffices to use contractions) and it does 
not require inductive assumptions
nor making choices (e.g. the domain loss).  The subtlety of 
the results comes because we have to choose 
appropriate spaces so that the estimates are 
uniform in the domains, the arrangement of the centers, 
etc. This uniformity of the results will be used 
when we consider the limit of infinitely many frequencies. 

The method of proof we use is very similar 
to the standard proof using graph transforms
\cite{HirschP68, HirschPS77}, 
which adapts very well to infinite dimensions
\cite{PlissS99}. Of course, there 
are several subtleties due to the infinite dimensional 
nature of the problem. In particular, we make 
essential use of the Banach algebra properties 
of the decay functions to make sure that we 
obtain estimates in the same spaces of functions 
(it is interesting to compare this with previous 
results in lattice dynamical systems). 
We emphasize that, in particular, the smallness 
conditions  are independent of the centers of the 
embedding. This will be crucial when we consider 
the limit of a large number of centers.

\begin{pro}\label{NDmoveloc}

Assume that the embedding $\tilde K$ has a $\delta$-invariant
hyperbolic splitting $\tilde \E^s, \tilde \E^c, \tilde \E^u$
 with respect to a map $F$ (See Definition~\ref{NDloc2}).  
Denote by $\tilde \Pi^{\sigma}$, $\sigma = s,c,u$
 the projections corresponding to 
this splitting.

There exists $\delta_0 > 0$ depending on $\|N\|_{\rho, \Gamma}, \| DF\circ \tilde
K\|_{\rho, \underline{c}, \Gamma}, \| DF^{-1}\circ \tilde
K\|_{\rho,\underline{c},\Gamma}$ and
 $\| \Pi^{s,c,u}\|_{\rho ,\underline{c}, \Gamma}$,
$\mu_{s,c,u}$ such that if $0 < \delta <\delta_0$ there
is an analytic splitting
\begin{equation}\label{splittingexact}
T_{\tilde K(\th)} \M = \mathcal{E}^s_{\tilde K(\th)}
\oplus \mathcal{E}^c_{\tilde K(\th)} \oplus
\mathcal{E}^u_{\tilde K(\th)}
\end{equation}
which is invariant under the co-cycle $DF \circ\tilde K$ over
$T_{\omega}$.

Let $\tilde \Pi^{s, c, u}_{\tilde K (\th)}$ be the projections
corresponding to the splitting \eqref{splittingexact}. We
furthermore have that there exist $C_h, 0 < \mu_1, \mu_2 < 1,\,
\mu_3 > 1$ such that $\mu_1 \mu_3 < 1,\, \mu_2\mu_3 < 1$ and the
characterizations \eqref{ndeg1loc}, \eqref{ndeg2loc},
\eqref{ndeg3loc} of the splitting hold.

Moreover, there exists $C > 0$, depending on the same quantities as
$\delta_0$ does, such that for $0 < \delta < \delta_0$
\begin{align*}
&\|\Pi^{s, c, u}_{\tilde K(\th)} - \tilde\Pi^{s,c,u}_{\tilde K(\th)}\|_{\rho , \Gamma}\leq
 C \delta,\\
&\vert\mu_{1, 2, 3} - \tilde \mu_{1, 2, 3}\vert < C \delta.
\end{align*}
\end{pro}

\begin{proof} The ideas in this proof follow the ones in \cite{FontichLS09}.
 They have been taken from
\cite{HirschPPS69}. We make sure that the estimates are uniform with
 respect to $\delta$ and $\underline c$.
 We divide the proof into several steps.

{\bf Step 1: Construction of the invariant spaces}. The existence of
the invariant splitting will be done through the Banach fixed point
principle applied to a graph transform operator.

We begin with the case of the stable bundle $\E^s_{\tilde{K}(\th)}$. 
We describe the stable space
$\E^s_{\tilde{K}(\th)}$ as the graph of a linear
map, i.e. $\E^s_{\tilde{K}(\th)}=\mbox{graph}\,(u\circ \tilde{K})$, where $u
\circ \tilde{K}$ maps $\mathcal{\tilde E}^s_{{\tilde K}(\th)}$ linearly
into $\mathcal{\tilde E}^c_{{\tilde K}(\th)} \oplus 
\mathcal{\tilde E}^u_{{\tilde K}(\th)}$.

Since the splitting \eqref{spscu} is approximately invariant we
can write the matrix $DF\big(\tilde K(\th)\big)$ with
respect to this decomposition as
$$DF\big(\tilde K(\th)\big)=
\begin{pmatrix} a_{11}(\th) & a_{12}(\th) \\ a_{21}(\th)& a_{22}(\th) \end{pmatrix}
$$
with $\|a_{12}\|_{\rho, \underline{c} ,\Gamma}< C\delta,\ \|a_{21}\|_{\rho,
\underline{c} ,\Gamma}< C\delta.
$
 We also write
$$DF\big(\tilde
K(\th+(N-1)\,\omega)\big)\times\dots\times
DF\big(\tilde K(\th)\big)=
\begin{pmatrix} a_{11}^N(\th) & a_{12}^N(\th) \\ a_{21}^N(\th)& a_{22}^N(\th)
\end{pmatrix}.
$$
Note that by \eqref{ndeg1Iterloc}, \eqref{ndeg2Iterloc} and
\eqref{ndeg3Iterloc} we have
\begin{equation}\label{cotaaN1}
\|a_{11}^N\|_{\rho, \underline{c} ,\Gamma}\le (1 + C \delta)\mu^N_1\,,\qquad \|a_{22}^{-N}\|_{\rho,\underline{c} ,
 \Gamma}\le(1 + C \delta)\mu^N_3\,,
\end{equation}
and
\begin{equation}\label{cotaaN2}
\|a_{12}^N\|_{\rho, \underline{c} ,\Gamma}\le C\delta\,,\qquad
\|a_{21}^{N}\|_{\rho, \underline{c} ,\Gamma}\le C\delta\,.
\end{equation}
The graph condition over the co-cycle is
$$DF\circ \tilde K(\th)\begin{pmatrix}\Id\\ u\circ\tilde K(\th)
\end{pmatrix}\in \mbox{graph} \Big(u\circ \tilde{K}\big(T_{\omega}(\th)\big)\Big).$$
This gives the functional equation for the map $u$
\begin{equation}\label{func}
u\circ \tilde K \big(T_\omega(\th)\big)(a_{11}+a_{12}\
u\circ\tilde K) (\th)=(a_{21}+a_{22}\ u\circ
\tilde{K})\,(\th)\,.
\end{equation}
Denoting $\tilde u=u\circ\tilde K\,$, \eqref{func} can be
rewritten as
\begin{equation}\label{ptfixe}
\tilde u=a_{22}^{-1}\big[\tilde u\circ T_\omega
(a_{11}+a_{12}\,\tilde u)-a_{21}\big].
\end{equation}
Let $\L_\eta$ be the ball of radius $\eta$ in the space of linear
operators from $\mathcal{\tilde  E}^s_{\tilde K(\th)}$ into 
$\mathcal{\tilde E}^c_{\tilde K(\th)}\oplus \mathcal{\tilde  E}^u_{\tilde K(\th)}$ with
the norm $\|\cdot\|_\Gamma$.

Let $\S_\eta$ be the space of analytic sections from $D_\rho$ to
$\L_\eta$, i.e. the space of $u:D_\rho\to\L_\eta$ such that
$u(\th):\mathcal{\tilde  E}^s_{\tilde K(\th)}\to \mathcal{\tilde  E}^c_{\tilde
K(\th)}\oplus \mathcal{\tilde  E}^u_{\tilde K(\th)}$ with the norm
$\|.\|_{\rho,\Gamma}$.

We take the operator $\Tau : \S_\eta\to \S_\eta$ defined as
the right-hand side of \eqref{ptfixe}. $\Tau$ is
approximated by $\Tau_0 : \S_\eta \to \S_\eta$ defined by
$$\Tau_0\,\tilde u=a^{-1}_{22}\ \tilde
u\circ T_\omega\ a_{11}.$$
We now consider $\Tau^N$ and
$\Tau_0^N$. An elementary computation gives
$$\Tau^N_0\,\tilde u= a^{-1}_{22}\dots a^{-1}_{22}\circ T^{N-1}_\omega\,\tilde u\circ T^N_ \omega \ a_{11}\circ T^{N-1}_\omega
\dots \,a_{11\,.}$$
Moreover, taking into account that
$\Tau$ is a degree two polynomial operator, we obtain by
simple algebraic manipulations
\begin{equation}\label{TN1}
\|\Tau^N-\Tau^N_0\|< C\delta
\end{equation}
and
\begin{equation}\label{TN2}
\mbox{Lip}(\Tau^N-\Tau^N_0)< C\delta\,.
\end{equation}
Using the Banach algebra properties of the decay norms, 
we have that
\begin{equation}\label{compaN1}
\|a^N_{11}(\th)-a_{11}\big(T^{N-1}_\omega(\th)\big)\dots
a_{11}(\th)\|_{\rho,\underline{c} ,\Gamma} < C\delta\,,
\end{equation}
\begin{equation}\label{compaN2}
\|a^{-N}_{22}(\th)-a_{22}^{-1}(\th)\dots a^{-1}_{22}
\big(T^{N-1}_\omega(\th)\big)\|_{\rho,\underline{c} ,\Gamma}< C\delta\,.
\end{equation}
By \eqref{cotaaN1}, \eqref{compaN1} and \eqref{compaN2}, if
$\delta$ is small, $\Tau^N_0$ sends $\S_\eta$ into $\S_\eta$
for all $\eta\in (0,1]$ and is a contraction in this domain. By
\eqref{TN1} and \eqref{TN2}, if $\delta$ is small $\Tau^N$
sends $\S_\eta$ into $\S_\eta$ for $\eta\in (C\delta,1]$ and it is
also a contraction.

Therefore $\Tau^N$ has a unique fixed point $u^*$ in $\S_1$
which belongs to $\S_{C\delta}$. It is clear that $\Tau  u^*$
is also a fixed point of $\Tau^N$, which belongs to
$\S_{C'\delta}\ \subset\S_1$ for some $C'\ge C$. By uniqueness
$\Tau u^*=u^*$.

A similar method can be applied for the center-unstable subspace. In this case the graph condition reads
$$DF_{\lambda}\circ \tilde{K}(\th)
\begin{pmatrix} v \circ \tilde{K}(\th) \\ \Id\end{pmatrix}\in
\mbox{graph} \Big(v\circ \tilde{K}\big(T_{\omega}(\th)\big)\Big),$$
and the resulting operator $\Tau:\mathcal{S}_\eta \rightarrow \mathcal{S}_\eta$ is
\begin{equation}\label{ptfixe2}
\Tau \tilde{v}=\big[(\tilde{a}_{11}\tilde{v}+\tilde{a}_{12})(\tilde{a}_{21}\tilde{v}+\tilde{a}_{22})^{-1}\big]\circ
T^{-1}_{\omega},
\end{equation}
where $\tilde v = v \circ \tilde K$.
Repeating the same procedure as above, we construct the center-unstable space and obtain similar bounds.

Now let us consider $G=F^{-1}$. We observe that
the stable space associated to the map $G$ is the unstable
space associated to $F$. Hence, applying the above
procedure to $G$, we construct the unstable space
$\E^u_{\tilde{K}(\th)}$ and center-stable space
$\E^{c,s}_{\tilde{K}(\th)}$ with similar
bounds. Finally we note that 
$\E^{c}_{\tilde{K}(\th)}=\E^{c,s}_{\tilde{K}(\th)} \bigcap \E^{c,u}_{\tilde{K}(\th)}$.

{\bf Step 2: Estimates on the projections}. 
We want to estimate the norm of the projection
$\Pi^s_{\tilde{K}(\th)}$ compared to the one of
$\tilde \Pi^s_{\tilde K(\th)}$.

Let $\xi\in T_{\tilde K(\th)}\M$. Using the decomposition
$\xi=(\xi^s,\xi^{cu}) \in \mathcal E^s_{\tilde K(\th)} \oplus
(\mathcal E ^c_{\tilde K(\th)} \oplus \mathcal E^u_{\tilde K(\th)})$ we have
the following representations
\begin{align*}
{\tilde \Pi}^s_{\tilde K(\th)} \xi = (\xi^s,0), &
\qquad \Pi^s_{ \tilde K(\th)}\xi =  (\tilde \xi^s,\tilde u(\th) \tilde \xi^s), \\
\tilde \Pi^{cu}_{\tilde K(\th)} \xi = (0,\xi^{cu} ), & \qquad 
\Pi^{cu}_{ \tilde K(\th)}\xi =  (\tilde v(\th) \tilde\xi^{cu}, \tilde \xi^{cu}) .
\end{align*}
Then
\begin{align*}
\xi^s &=\tilde \xi^s +
\tilde v(\th) \tilde \xi^{cu}, \\
\xi^{cu} &= \tilde u(\th) \tilde \xi^s +  \tilde \xi^{cu}
\end{align*}
or equivalently
$$
\left( \begin{array}{c}
\tilde \xi^s \\
\tilde \xi^{cu}
\end{array}\right)
=
\left( \begin{array}{cc}
\Id & \tilde v(\th) \\
 \tilde u(\th) & \Id
\end{array}\right) ^{-1}
\left( \begin{array}{c}
\xi^s \\
\xi^{cu}
\end{array}\right)
$$
since the matrix $B=\left( \begin{array}{cc}
\Id & \tilde v(\th) \\
 \tilde u(\th) & \Id
\end{array}\right)$
is invertible because is $O(\delta) $-close to the identity and
moreover, by the Neumann series theorem, we can write
$$B^{-1}=\begin{pmatrix}\Id+w_{11}(\th) &w_{12}(\th)& \\ w_{21}(\th) & \Id+w_{22}(\th)\end{pmatrix}$$
with $\|w_{ij}\|_{\rho,\Gamma}<C\delta$.
 Therefore using 
\begin{align*}
\left(\tilde\Pi^s_{\tilde K(\th)}-\Pi^s_{\tilde
K(\th)}\right)\begin{pmatrix}\xi^s\\ \xi^{cu}\end{pmatrix}=
\begin{pmatrix}\tilde\xi^s-\xi^s\\ \tilde u(\th)\
\tilde\xi^s\end{pmatrix} = 
\begin{pmatrix}  &\tilde v(\th) \tilde \xi^{c,u} \\
\tilde u(\th) \tilde \xi^s
\end{pmatrix} 
\end{align*}
this gives 
$$\big\|\tilde\Pi^s_{\tilde K(\th)}-\Pi^s_{\tilde K(\th)}\big\|_{\rho,\underline{c},\Gamma}
\le C\delta\,.$$
Analogously one has $\big\|\tilde\Pi^{cu}_{\tilde
K(\th)}-\Pi^{cu}_{\tilde K(\th)}\big\|_{\rho,\underline{c},\Gamma} <
C\delta\,.$

The estimates for the projections $\Pi^u$,  $\Pi^{sc}$ are
obtained in a similar way. {From} those we deduce readily the ones
for $\Pi^c$ by noting  that $\Pi^c = \Pi^{cu} - \Pi^u$.

{\bf Step 3: Existence of $\mu_1$, $\mu_2$, $\mu_3$
$C_h$ for the new splitting and estimates.} 

Since the distance between the spaces
$\mathcal E^{s,c,u}_{\tilde K(\th)}$ and $\mathcal{\tilde
E}^{s,c,u}_{\tilde K(\th)}$ is bounded by $C\delta$, the
restriction of $DF\circ \tilde K\circ
T^{N-1}_\omega\times\dots\times DF\circ \tilde K$ to them is
separated by a distance less than $C\delta$. Therefore there exist
$\mu_1,\,\mu_2,\,\mu_3$ with $\mu_1,\,\mu_2<1,\ \mu_3>1,\ \mu_1
\mu_3<1,\ \mu_2\mu_3<1$ such that $|\mu_{1,2,3}-\tilde
\mu_{1,2,3}|<C\delta$ and \eqref{ndeg1Iterloc} holds for
$\tilde\mu_1=\mu_1$ and $v\in\mathcal{\tilde E}^{s}_{\tilde
K(\th)}$. Similarly for \eqref{ndeg2Iterloc} and
\eqref{ndeg3Iterloc}.

Once we have these last properties we deduce that there exists
$C_h$ such that \eqref{ndeg1loc}, \eqref{ndeg2loc} and
\eqref{ndeg3loc} hold for all $n\ge 1$.
\end{proof}

As a consequence of Proposition \ref{NDmoveloc} we have the
following result, which shows that the hyperbolicity constants 
do not deteriorate much if we change very little the embeddings 
$K$. This will be used in the iterative process. We will use it to 
show that, during the iterative process, the hyperbolicity constants 
remain uniformly bounded. We will also deduce that when the 
embeddings converge, the splittings converge. 

\begin{pro}\label{tempmove}
Under the hypotheses of Proposition \ref{improvementloc}, assume
that $\| K-\tilde K\|_{\rho,\underline{c},\Gamma}$ is small
enough. Then there exists an analytic splitting
$$T_{\tilde K(\th)}\, \M = \mathcal E^s_{\tilde K(\th)}\oplus
\mathcal E^c_{\tilde K(\th)} \oplus \mathcal E^u_{\tilde
K(\th)}$$
 invariant under the co-cycle $DF_\lambda \circ \tilde K$ over
 $T_\omega$.

 Furthermore, there exists $C>0$ such that
 \begin{equation*}
 \begin{split}
 &\big\|\Pi^{s,c,u}_{K(\th)}-\Pi^{s,c,u}_{\tilde
 K(\th)}\big\|_{\rho, \Gamma}\le C \| K- \tilde K\|_{\rho,
 \underline{c},\Gamma},\\
 &\vert\mu_i -\tilde \mu_i\vert \le C\| K - \tilde K\|_{\rho, \underline{c}, \Gamma},\qquad i=1,2,3\\
& |C_h - \tilde C_h| \le C \delta
 \end{split}
 \end{equation*}
\end{pro}

\begin{proof}
The invariant splitting
$$T_{K(\th)}\, \M = \mathcal E^s_{K(\th)} \oplus \mathcal E^c_{K(\th)}\oplus \mathcal E^u_{K(\th)}$$
for $DF_\lambda \circ K$ is an approximate invariant splitting for
$DF_\lambda\circ \tilde K$. Here we identify $T_{\tilde K(\th)}\,\M$
with $T_{K(\th)}\,\M$ since $M$ is Euclidean. We then can take $\delta
= C\| K-\tilde K\|_{\rho,\underline{c},\Gamma}$.
\end{proof}

 \subsection{Solution of the linearized
equation on the center subspace}
\label{sec:centerequation}
In this section we solve approximately the projection of equation
\eqref{linear} on the invariant center subspace provided by Proposition \ref{tempmove} and establish estimates.
Projecting with $\Pi^c_{K(\th+\omega)}$ and using the notation
$$\Delta^c(\th)=\Pi^c_{K(\th)}\,\Delta (\th),\qquad E^c(\th)=\Pi^c_{K(\th+\omega)}\,E(\th)\,,$$
we obtain
\begin{equation}\label{eqCenter}
\Pi^c_{K(\th+\omega)}\frac{\partial F_\lambda(K(\th))}{\partial
\lambda}\Lambda+DF_{\lambda}(K
(\th))\Delta^c(\th)-\Delta^c(\th+\omega)=-E^c(\th).
\end{equation}
\subsubsection{Estimates on cohomology equations}

We recall the well-known small divisors lemma (see
\cite{Russmann76}, \cite{Russmann76a}, \cite{Russmann75},
\cite{Llave01c}).
\begin{pro}\label{sdrussmann}
Let $M$ be a finite dimensional Euclidean manifold and $\omega \in
D(\kappa,\nu)$. Assume the mapping $h:D_{\rho}\rightarrow M$ is
analytic on $D_{\rho}$ and has zero average. Then for any $0 <
\sigma <\rho$ the difference equation
\begin{equation*}
v(\th+\omega)-v(\th)=h(\th)
\end{equation*}
has a unique zero average solution $v:\torus^l \rightarrow M$,
real analytic on $D_{\rho-\sigma}$ for any $0 < \sigma <\rho$.
Moreover, we have the estimate
\begin{equation}\label{estimsd}
\|v\|_{\rho-\sigma} \leq C \kappa\sigma^{-\nu}\|h\|_{\rho},
\end{equation}
and where $C$ only depends on $\nu $ and the dimension of the torus $l$.
\end{pro}

We have the following corollary of R\"ussmann's result in our context. 
\begin{cor}
Let $\M=\ell^\infty(\Z^N)$ and $\omega \in D(\kappa,\nu)$ and
assume the mapping $h:\torus^l \rightarrow \M$ belongs to
$\A_{\rho,\underline c, \Gamma}$ and has zero average. Then for
any $0 < \sigma <\rho$ the difference equation
\begin{equation*}
v(\th+\omega)-v(\th)=h(\th)
\end{equation*}
has a unique zero average solution $v:\torus^l \rightarrow
\M$, belonging to $\A_{\rho-\sigma,\underline c, \Gamma}$
for any $0 < \sigma <\rho$. Moreover, we have the estimate
\begin{equation}\label{estimsdloc}
\|v\|_{\rho-\sigma,\underline{c},\Gamma} \leq C \kappa\sigma^{-\nu}\|h\|_{\rho,\underline{c},\Gamma},
\end{equation}
where $C$ only depends on $\nu $ and linearly on the dimension of the torus $l$.
\end{cor}
\begin{remark}
An important fact of the previous statement is that, since we
consider the supremum norm on $M$ and the equation is solved
component by component, the estimates are independent on the
dimension of $\mathcal M$.
\end{remark}

\begin{proof}
We write the equation in coordinates $ i \in \ZZ^N$ to get
\begin{equation*}
v_i(\th+\omega)-v_i(\th)=h_i(\th)
\end{equation*}
with $v_i$ and $h_i$ mapping $\T^l$ into $M$. We then apply the
previous finite dimensional result Proposition~\ref{sdrussmann}
to each of the components. 

We also observe that the partial derivatives of 
the functions also satisfy 
\[
\partial_{\th_k} v_i(\th + \omega) 
- \partial_{\th_k} v_i(\th )  = 
\partial_{\th_k} h_i(\th) 
\]
and that $\partial_{\th_k} h_i$ 
has zero average. Then,
$\partial_{\th_k} v_i(\th )$ will be the 
zero average solution obtained applying 
Proposition~\ref{sdrussmann}.
Multiplying the estimates afforded by
Proposition~\ref{sdrussmann} by 
$\Gamma^{-1}(i-c_j)$ and taking the  supremum in  $i$ and
 the infimum in $j$ we get the desired result.
\end{proof}

\subsubsection{Isotropic character of the torus}

One important issue is the approximate isotropic character of the
approximate torus $K(\T^l)$. In our context the two-form
$\Omega_\infty$ is formal but $K(\T^l)$ is finite dimensional,
therefore asking the forms to be isotropic amounts to ask that the
pull-back $K^*\Omega_\infty$ vanishes. By the decay properties of $K$ (see Lemma \ref{sympDecay}), this last object is a true form on $\T^l$ and we can write
$$K^*\Omega_\infty(\th)\,(\xi,\eta)=\langle \xi, L(\th)\eta\rangle,\qquad \xi,\eta\in\R^l\,. $$

The isotropic character of the torus is then equivalent to
$$L(\th)\equiv DK(\th)^\top J_\infty \big(K(\th)\big)\,DK(\th)=0$$
for all $\th\in\T^l$. Notice that $L$ is a $l \times l-$ matrix. 

We first consider the case when $K$ is a
solution of \eqref{translatedloc}.
\begin{lemma}\label{lagExact}
Let $(\M,\Omega_\infty=d\alpha_\infty)$ be the lattice manifold.
Assume that $F_{\lambda_0} \circ K=K\circ T_\omega$, $F_{\lambda_0}$ is 
 symplectic, $\omega$ is
rationally independent and $K\in\A_{\rho,\underline c,\Gamma}$. Then $L(\th)$ is identically zero.
\end{lemma}
\begin{proof}
Since $K\in\A_{\rho,\underline c,\Gamma}$, using that $F_{\lambda_0}$ is 
symplectic (see Appendix \ref{sec:symplectic}).
we have
$$K^*\Omega_\infty=K^* F_{\lambda_0}^*\Omega_\infty=(K\circ T_\omega)^*\Omega_\infty\,.$$
By the condition on $\omega,\ T_\omega$ is ergodic and therefore
$K^*\Omega_\infty$ is constant. Hence $L$ is also constant. Moreover, the
fact that $\M$ is formally exact symplectic shows that
$K^*\Omega_\infty = d (K^*\alpha_\infty)$, where now, $d$ is the differential 
on the torus (see Appendix \ref{sec:symplectic}) . In coordinates this means that 
$L(\th)$ has the form $DL_1(\th)^\top-DL_1(\th)$ for some
finite dimensional vector $L_1(\th)$.
Since the average of
derivatives is zero we get that $L$ is zero on $\T^l$.
\end{proof}

\subsubsection{Geometric considerations on the center bundle $\mathcal E^c_{K(\th)}$ in the exact case}

In this section, we show how to construct geometrically a very natural basis of the center subspace $\mathcal E^c_{K(\th)}$ when $K$ satisfies \eqref{translatedloc}. Recall first that we are assuming that $\mathcal E^c_{K(\th)}$ is finite dimensional with dimension $2l$. 
\medskip 
We start with the following lemma.  
\begin{lem}\label{lem:restriction}
The restriction $\Omega^c_{K(\th)}$ of $\Omega_\infty$ to the finite-dimensional space $\mathcal E^c_{K(\th)}$ is a symplectic form on $\mathcal E^c_{K(\th)}$. 
\end{lem}
\begin{proof}
To prove the claim, it is enough to show that the form $\Omega^c_{K(\th)}$ is non degenerate. Assume that $u,v \in T_{K(\th)}\M$ . Then we have 

$$\Omega_\infty(u,v)=\Omega_\infty( DF^n(K(\th)u, DF^n(K(\th)v),\,\,\,\,\,\, n \in \Z.$$
Since the torus is invariant, we have that 
\begin{equation*}
DF \circ K\circ T^{n-1}_{\omega}(\th)\times\dots \times
DF \circ K(\th)=DF^n \circ K (\th). 
\end{equation*} 

We deduce, sending $n \rightarrow \pm\infty$ and using the hyperbolic conditions (expansion/contraction properties), that $\Omega_\infty(u,v)=0$ in the following cases
\begin{itemize}
\item $u,v\in \mathcal E^s_{K(\th)} $,
\item $u,v\in \mathcal E^u_{K(\th)} $,
\item $u\in \mathcal E^s_{K(\th)}\cup\mathcal E^u_{K(\th)}  $ and $v\in \mathcal E^c_{K(\th)} $,
\item $u\in \mathcal E^c_{K(\th)} $ and $v\in \mathcal E^s_{K(\th)}\cup\mathcal E^u_{K(\th)}.  $
\end{itemize}

Assume that: let $\tilde c \in \E^c_{K(\th)}$
$$\Omega^c_{K(\th)}(c,\tilde c)=0,\,\,\forall \,\, c \in \E^c_{K(\th)}.$$

By the previous argument, we have that for every $u \in \mathcal E^u_{K(\th)}$ and $v \in \mathcal E^s_{K(\th)}$
$$\Omega^c_{K(\th)}(\tilde c,u)=\Omega_\infty(c,\tilde c)=\Omega_\infty(c+u+v,\tilde c)=0$$ 
Since $\Omega_\infty$ is non-degenerate, this leads to the desired result. 
\end{proof}

Now define $\hat L= DK^\perp J^c(K) DK.$ For every $\nu_1,\nu_2 \in T_\th \TT^l$, 
$$
\nu_1^T DK(\th) J^c(K(\th))DK(\th)\nu_2=\Omega^c_{K(\th)}(DK(\th)\nu_1,DK(\th)\nu_2)=\Omega_\infty(DK(\th)\nu_1,DK(\th)\nu_2).
$$
Hence 
$$
\nu_1^T DK(\th) J^c(K(\th))DK(\th)\nu_2=\nu_1^\perp DK(\th)^\perp J_\infty(K(\th)) DK(\th)\nu_2=\nu_1^\perp L \nu_2=0,
$$
hence 
$$\hat L=0$$

Since range $DK(\th) $ is the tangent space of the torus 
$K(\TT^l)$ and the dynamics on the torus is conjugated to 
a rotation, $DK(\th)\RR^l$ is contained in $\E^c_{K(\th)}$.
Moreover we have  that 
$J^c(K(\th))^{-1}DK(\th)\RR^l$ also is contained in  $\E^c_{K(\th)}$.
Instead of $J^c(K(\th))^{-1} DK(\th)$ 
we will consider the matrix $J^c(K(\th))^{-1}DK(\th) N(\th)$
where $N(\th)$ is the normalization $l \times l$-matrix 
$N(\th)=[DK(\th)^\top DK(\th)]^{-1}$ previously introduced. Both have the same range 
because $N(\th)$ is non-singular. 
The  role of $N$ is to provide some normalization for the 
symplectic conjugate.

Now we check that the range of $[\Pi^c_{K(\th)} DK(\th), J^c(K(\th))^{-1} DK(\th)N(\th)]$
is $2l$-dimensional. Indeed, let $\left \{e_j \right \} $ be the canonical basis of $\RR^l$ and assume that there is a linear combination that vanishes on $\E^c_{K(\th)}$
$$
f= \sum_{j=1}^l\alpha_j \Pi^c_{K(\th)} DK(\th) e_j + 
\sum_{j=1}^l \beta_j J^c(K(\th))^{-1} DK(\th)N(\th) e_j  = 0.
$$
Then, for $1\le k\le l$, using the isotropic character of $T_{K(\th)} K(\TT^l)$
\begin{align*}
0 & = \sum_{j=1}^l \beta_j e_k^\top DK(\th)^\top J^c(K(\th))J^c(K(\th))^{-1}   DK(\th)N(\th) e_j  \\
& = \sum_{j=1}^l \beta_j \langle e_k, e_j \rangle = \beta_k.
\end{align*}
This calculation shows that $f$ reduces to
$ \sum_{j=1}^l\alpha_j  DK(\th) e_j $. Moreover, for  $1\le k\le l$
\begin{align*}
0 & = \sum_{j=1}^l \alpha_j e_k^\top N(\th)^\top DK(\th)^\top J^c(K(\th))^{-\top} 
\Pi^c_{K(\th)} J^c (K(\th))
DK(\th) e_j  \\
& = -\sum_{j=1}^l\alpha_j \langle e_k, e_j \rangle = -\alpha_k .
\end{align*}
Hence $\alpha_j=\beta_j=0$ for all $j=1,...,l$. 
We conclude that 
$$
range [\Pi^c_{K(\th)} DK(\th),  J^c(K(\th))^{-1}  (K(\th))^{-1} DK(\th)N(\th)] = \E^c_{K(\th)}.
$$

The treatment of equation \eqref{linear} on the center subspace is greatly facilitated by observing that the preservation of the symplectic structure  imposes that the system is close to a diagonal structure.

In our context, the dimension of the center subspace is $2l$ and
$\mathcal{E}^c_{K(\th)} \sim \mathbb{R}^{2l}$. In \cite{FontichLS09},
the authors studied the case when $\M$ is finite dimensional (with symplectic structure $J$). They used the set of vectors $$\left \{ \frac{\partial K(\th)}{\partial \th_j},
  J^{-1}(K(\th))\frac{\partial K(\th)}{\partial \th_j} \right
\}_{j=1,\dots ,l}$$
to perform a transformation which allows to approximately solve 
up to quadratic error the projected equation on the center subspace.

We consider the map $\tilde M(\th)$ given in matrix notation by  
\begin{equation}\label{definicioMtilde}
\tilde M (\th)= \Big [ \Pi^c_{K(\th)} DK(\th),\,\,\,J^c(K(\th))^{-1} DK(\th)N(\th) \Big ].
\end{equation}
We will see that this map is a very convenient change of coordinates in the linearized equations which makes them easily solvable.
\begin{remark}
It is worth noting that all the quantities defined above (such as
$N$, for instance) make perfect sense even if we are manipulating
``infinite dimensional'' matrices. This is due to the fact that we
are considering maps in $\L_\Gamma(\ell^\infty(\Z^N))$. For instance, since
$K$ is assumed to be in $\mathcal{A}_{\rho,\underline{c},\Gamma}$,
one has for $i,j=1,\dots ,l$
\begin{equation*}
N(\th)^{-1}_{ij}=\sum_{k \in\ZZ^N}\frac{\partial K_k}{\partial
\th_i}\frac{\partial K_k}{\partial \th_j}.
\end{equation*}
Therefore, for $\th\in D_\rho$ this leads to
\begin{align*}
|N(\th)^{-1}_{ij}|&\leq
\|DK\|^2_{\rho,\underline{c},\Gamma}\,\min_{1\le p,q\le
R}\,\sum_{k \in \ZZ^N}\Gamma(k-c_p)\Gamma(k-c_q)\\
&\le \|DK\|^2_{\rho,\underline c,\Gamma}\,\min_{1\le p,q\le
R}\,\Gamma (c_p-c_q)\le\|DK\|^2_{\rho,\underline{c},\Gamma}\,.
\end{align*}
This gives that $\|N\|_{\rho,\Gamma} <\infty$.
\end{remark}
\begin{remark}
It is also worth noticing that we have the following estimates for
the generalized symplectic matrix $J_\infty$:
\begin{equation*}
\|J_\infty\|_{\Gamma}=\Gamma^{-1}(0)\,\|J\|, \,\,\|J^{-1}_\infty\|_{\Gamma}=\Gamma^{-1}(0)\,\|J^{-1}\|
\end{equation*}
since $(J_\infty)_{ij}=J\delta_{ij}$, where $\delta_{ij}$ is the
Kr\"onecker symbol.
\end{remark}

\subsubsection{Representation of $DF_\lambda$ on the center subspace} \label{sect:normal}

In this section we study a suitable representation of $DF_\lambda$
applied to the basis of the center subspace given by the columns
of $\tilde M(\th)$. We begin by considering the case when $K$
is a solution of \eqref{translatedloc}.
\begin{lemma}\label{repEx}
Let $K$ be a solution of equation \eqref{translatedloc}. Then
there exists a $2l\times 2l$ matrix $\S_\lambda(\th)$ such that
\begin{equation}\label{cMtilde}
DF_{\lambda}(K(\th))\tilde{M}(\th)=\tilde{M}(\th+\omega)\mathcal{S}_\lambda(\th),
\end{equation}
with
\begin{equation}\label{smatrix}
\mathcal{S}_\lambda(\th)=\begin{pmatrix} \Id_l & A_\lambda(\th)\\
0_l &
  \Id_l \end{pmatrix}.
\end{equation}
The matrix $A_\lambda(\th)$ is
\begin{equation*}
A_\lambda(\th)=P(\th+\omega)^\top
[(DF_{\lambda}
J^c(K(\th))^{-1} P(\th)-[J^c(K)^{-1} P](\th+\omega)],
\end{equation*}
where $P(\th)=DK(\th)N(\th)$.
\end{lemma}
\begin{proof}
Differentiating equation \eqref{translatedloc} with respect to
$\th$, we get
\begin{equation*}
DF_{\lambda}(K(\th))DK(\th)=DK(\th+\omega).
\end{equation*}
This shows that $\mathcal{S}_\lambda(\th)$ has the form
$$\begin{pmatrix} \Id_l & A_{\lambda} (\th)\\ 0_l & B_\lambda (\th)
\end{pmatrix}, $$ where $A_{\lambda}(\th), B_\lambda(\th)$ are $l \times l$ matrices.
We will now show that $B_\lambda=\Id_l$ via geometric properties. We should have
\begin{equation}
\label{formula-AB}
[DF_\lambda(K)J^c(K)^{-1}  DK \,N ](\th)= DK(\th+\omega)\,A_\lambda(\th)
+[ J^c(K)^{-1}DK \,N](\th+\omega)B_\lambda(\th).
\end{equation}
By the isotropic character of $K(\TT^l)$ we have 
$DK^\top J^c(K) DK =0$ and the definition of $N$, we have 
\begin{equation}
[DK ^{\top} J^c(K)](\th+\omega)
[DF(K) J^c(K)^{-1}DK \,N ](\th)=B_\lambda(\th).
\label{eq:formulaB}
\end{equation}

Also by the symplecticness of $F_\lambda$
$$
J^c(K(\th+\omega))DF_\lambda (K(\th)) = 
J^c(F_\lambda(K(\th))) DF_\lambda (K(\th)) = $$
$$ [DF_\lambda (K)^{-\top} J^c(K)] (\th),
$$
when restricted to $\E^c_{K(\th)}$.
Then equation \eqref{eq:formulaB} becomes
$$
B_\lambda(\th)=DK ^{\top} (\th+\omega)
[DF_\lambda (K)^{-\top} DK \,N ](\th)
= 
[DK ^{\top}  DK \,N ](\th)=\Id_l. 
$$

To obtain the expression of $A_\lambda(\th)$ we multiply  
\eqref{formula-AB} by $(DK\,N)(\th+\omega)^\top$ to get
\begin{equation}\label{defA}
A_\lambda(\th)=P(\th+\omega)^\top \Big[[DF_\lambda(K) J^c(K)^{-1}  P](\th)-[J^c(K)^{-1}P](\th+\omega)\Big]. 
\end{equation}
\end{proof}

\begin{figure} 
\begin{center} 
\includegraphics[height =  2.5 in]{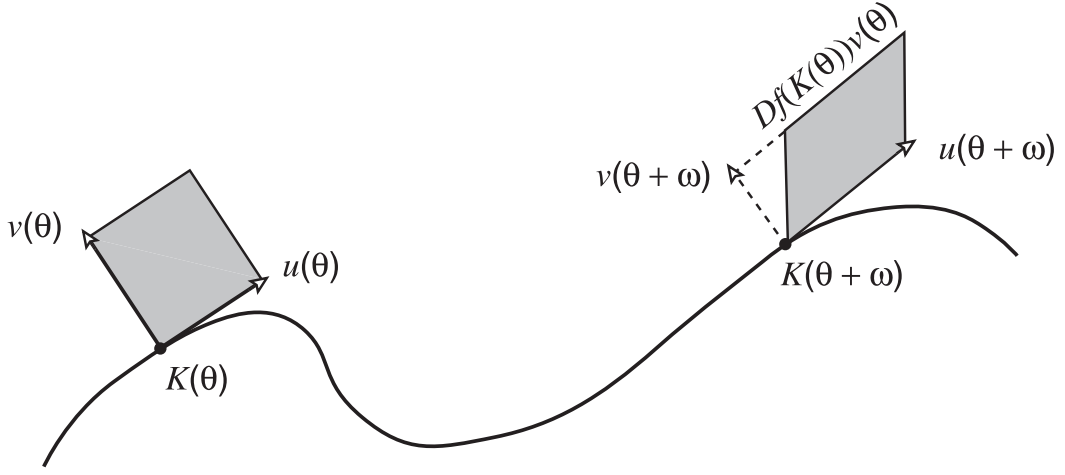} 
\end{center} 
\caption{Illustration of the geometric 
reason for reducibility of the equations in 
the center. Note that the base of a rectangle gets 
mapped into the base of the other (differentiating 
the invariance). The preservation of the symplectic 
form implies that the symplectic area of both rectangles is 
the same.} 
\end{figure} 

The matrix $\tilde{M}(\th)$ is not invertible since it is not
square. However we can derive a generalized inverse for
$\tilde{M}(\th)$. 
As a motivation for subsequent developments, we 
first present Lemma~\ref{isinvariant} which deals with  
the geometric cancellations in the case of an exactly invariant torus. The case of interest for 
a KAM algorithm --- when the torus is only approximately 
invariant --- will be studied in Lemma~\ref{repApp} as a perturbation of 
Lemma~\ref{isinvariant}. 

A straightforward calculation shows that 
\begin{equation}
\tilde{M}^\top J^c(K) \tilde{M}
=\begin{pmatrix} \hat L &   \Id_l \\ 
-\Id_l &(  N^{\top}DK^\top J^c(K)^{-\top} DK \,N  
\end{pmatrix}.
\label{eq:formJM}
\end{equation}

\begin{lemma} \label{isinvariant} 
Let $K$ be a solution of \eqref{translatedloc}. Then the matrix 
$\tilde{M}^\top J^c(K)\tilde{M} $ is invertible
and
\begin{equation*}
(\tilde{M}^\top  J^c(K)\tilde{M})^{-1}
=\begin{pmatrix} (- N^{\top}DK^\top J^c(K)^{-\top} DK \,N &
  -\Id_l \\ \Id_l & 0 
\end{pmatrix}.
\end{equation*}
\end{lemma}
\begin{proof} 
It follows immediately from \eqref{eq:formJM} and 
the isotropic character of the
invariant torus, i.e. $\hat L=0$.
\end{proof}

Now we consider the case we are interested in, that is when $K$ is
an approximate solution with an error
$E(\th)=\mathcal{F}_{\omega}(\lambda,K)(\th)$, assumed to be
small. We will need the invertibility of $\tilde{M}(\th)^\top J^c(K(\th)) \tilde{M}(\th)$ in this case.

More precisely, we introduce 
\begin{equation}\label{e}
e(\th)=DF_\lambda(K(\th))\tilde{M}(\th)-\tilde{M}(\th+\omega)\mathcal{S}_\lambda(\th),
\end{equation}
where $\mathcal{S}_\lambda$ is given by \eqref{smatrix}.
If we denote $e(\th)=(e_1(\th), e_2(\th))$, a simple algebraic
computation yields
\begin{equation*}
\begin{split}
& e_1(\th)= DE^c(\th) \\
& e_2(\th)=[(DF_\lambda J^c(K)^{-1}DK\,N](\th)-
DK(\th+\omega)A_\lambda (\th)\\
& -[J^c(K)^{-1}DK\,N](\th+\omega)= O(E,DE)
\end{split}
\end{equation*}
by the choice of $A_\lambda$. 

We first prove the approximate isotropic character of the torus

\begin{lemma}\label{lagApp}
Under the previous conditions, let $K\in\A_{\rho,\underline
c,\Gamma}$ be a function which solves \eqref{translatedloc}
approximately and let $E=F_{\lambda} \circ K-K\circ T_\omega$ be the corresponding 
error. Then for $0<\delta<\rho/2$,
\begin{equation}\label{lag} \|L\|_{\rho-2\delta,\Gamma}\le C\kappa\,\delta^{-(\nu+1)}\
\|E\|_{\rho,\underline c,\Gamma},
\end{equation}
where $C$ depends on $l,\nu,R,\rho,\ \|DK\|_{\rho,\underline
c,\Gamma},\|F_{\lambda}\|_{C^1_\Gamma (B_r)},\|J^c(K)\|_{C^1 (B_r)}$.
\end{lemma}
\begin{proof}
We define the two-form on the torus $\T^l$
$$\Omega_e=K^*\Omega_ \infty-(K\circ T_\omega)^*\Omega_\infty .$$
The corresponding matrix is $L-L\circ T_\omega$. Using that
$F_\lambda$ is symplectic we have that for any $(\xi,\eta) \in \R^{2l}$
\begin{align*}
\Omega_e(\th)(\xi,\eta)&=\big((F_\lambda \circ K)^*
\Omega_\infty-(K\circ T_\omega)^*\Omega_\infty\big) (\th)
(\xi,\eta)\\
&=\dpy\sum_{i\in\Z^N}\Big[\Omega\Big((F_\lambda)_i\big(K(\th)\big)\Big)
\big(D((F_\lambda)_i\circ K)(\th)\xi,\ D((F_\lambda)_i\circ
K)(\th)\eta\big)\\
&\quad -\Omega\big(K\circ T_\omega (\th)\big)\big(D(K_i\circ
T_\omega)(\th)\xi,\ D(K_i\circ T_\omega)(\th)\eta\big)\Big].
\end{align*}
Since $D((F_\lambda)_i\circ K)-D(K_i\circ T_\omega)=DE_i$ and
$\|DE\|_{\rho-\delta,\underline c, \Gamma}\le \frac{1}{\delta}\,
\|E\|_{\rho,\underline c, \Gamma}$, using the decay properties of
$F\circ K$ and $K\circ T_\omega$ to sum the series, we obtain that
\begin{equation}\label{differenceL}
L-L\circ T_\omega=g
\end{equation}
with $\|g\|_{\rho-\delta}\le C\delta^{-1}\,\|E\|_{\rho,\underline
c, \Gamma}$ for some $C$ as in the statement. Now we use
Proposition \ref{sdrussmann} to finish the proof. 
\end{proof}

The next step is to ensure the invertibility of the $2l \times
2l$-matrix $\tilde{M}^\top  J^c(K)\tilde{M}$. According to expression
\eqref{eq:formJM}, we can write  
\begin{equation*}
\tilde{M}(\th)^\top  J^c(K)\tilde{M}(\th)=V(\th)+R(\th),
\end{equation*}
where
\begin{equation*}
V =\begin{pmatrix} 0 &
  \Id_l\\ -\Id_l  & N^\top DK^\top   J^c(K)^{-\top}DK\, N \end{pmatrix}
\end{equation*}
and 
\begin{equation*}
R =\begin{pmatrix} \hat L &   0\\0 & 0\end{pmatrix}.
\end{equation*}
We have the following lemma, providing the desired invertibility
result  under 
a smallness assumption on $E$, namely \eqref{smallnessdeltaE1} 
in the next lemma.

\begin{lemma}\label{repApp}
There exists a constant $C>0$ such that if 
\begin{equation} \label{smallnessdeltaE1}
C \kappa \delta^{-(\nu+1)} \|E\|_{\rho,\underline c,\Gamma} \leq 1/2 
\end{equation}
for some $0<\delta< \rho/2$ then the matrix $\tilde{M}^\top(\th) 
 J^c(K) \tilde{M}(\th)$ is
invertible for $\th \in D_{\rho-2\delta}$ and there exists a matrix $\tilde{V}(\th)$ such that 
\begin{equation*}
(\tilde{M}(\th)^\top J^c(K(\th)) \tilde{M}(\th))^{-1}=V(\th)^{-1}+\tilde{V}(\th)
\end{equation*}
with 
\begin{equation*}
\tilde{V}(\th)=\Big(\sum_{k=1}^{\infty}(-1)^k (V(\th)^{-1}R(\th))^k\Big)V(\th)^{-1},
\end{equation*}
where the series is absolutely convergent. Furthermore, we have the estimate 
\begin{equation}\label{invEst}
\|\tilde{V}\|_{\rho-2\delta,\Gamma} \leq C' \kappa \delta^{-(\nu+1)} \|E\|_{\rho,\underline c, \Gamma}, 
\end{equation} 
where the constant $C'>0$ depends on $l$, $\nu$, 
$\|F_/¬ambda\|_{C_\Gamma^1(B_r)}$, $\|J^c(K)\|_{C_\Gamma^1(B_r)}$,
$\|DK\|_{\rho,\underline c,\Gamma}$,
$\|N\|_{\rho,\Gamma}$. 
\end{lemma}
\begin{proof}  The matrix
$V(\th)$ is invertible with 
\begin{equation*}
V^{-1}=\begin{pmatrix} N^\top DK^\top   J^c(K)^{-\top}DK\, N  & -\Id_l \\ 
\Id_l & 0
\end{pmatrix}. 
\end{equation*}

We can write
\begin{equation*}
\tilde{M}(\th)^\top  J^c(K(\th)) \tilde{M}(\th) =V(\th) (\Id_{2l}+V(\th)^{-1}R(\th)). 
\end{equation*}
To apply the Neumann series
(and consequently justify the existence of
the inverse of $\Id_{2l}+V^{-1}R$ as well as the estimates
for its size), we have to estimate the
term $V^{-1}R$. According to Lemma \ref{lagApp}, we have the
estimate for $L$
\begin{equation*}
\|L\|_{\rho-2\delta,\Gamma} \leq C \kappa \delta^{-(\nu+1)} \|E\|_{\rho,\underline c, \Gamma}
\end{equation*}
for all $\delta \in (0,\rho/2)$. 
This leads to the estimate 
\begin{equation*}
\|V^{-1}R\|_{\rho-2\delta,\Gamma} \leq C \kappa \delta^{-(\nu+1)} \|E\|_{\rho,\underline c, \Gamma} 
\end{equation*}
for $0<\delta <\rho/2$, 
where $C>0$ depends on $l$,  $\nu$, $\|DK\|_{\rho,\underline c, \Gamma}$,
$\|N\|_{\rho,\Gamma}$ and $\|J^c(K)^c\|_{\rho, \Gamma}$. 
Because of assumption \eqref{smallnessdeltaE1}, we have
that  the right-hand side of the last equation is less than $1/2$.

Then the matrix $\Id_{2l}+V(\th)^{-1}R(\th)$ is invertible with 
\begin{equation*}
\|(\Id_{2l}+V^{-1}R)^{-1}\|_{\rho-2\delta,\Gamma} \leq \frac{1}{1-\|V^{-1}R\|_{\rho-2\delta,\Gamma}} \leq 2.
\end{equation*}
Now the estimates follow immediately. 
\end{proof} 

\subsubsection{Identification of the center subspace}
In this section, we identify the center space as being very 
close (up to terms that can be bounded by the error) 
to the range of the matrix $\tilde{M}$. 
This will allow us to use the range of $\tilde{M}$ in 
place of $\E_{K(\th)}^c$ without changing the quadratic 
character of the method.

\begin{pro}\label{prop:distance}
Denote by $\Gamma_{K(\th)}$ the 
range of $\tilde{M}(\th)$ and by 
$\Pi^\Gamma_{K(\th)}$ the projection onto
$\Gamma_{K(\th)}$ according to the splitting 
$\E^s_{K(\th)} \oplus \Gamma_{K(\th)} \oplus \E^u_{K(\th)}$.

Then there exists a constant $C>0$ such that if  
$$
\delta^{-1} \| E\|_{\rho, \underline c, \Gamma} \leq C
$$
we have the estimate
\begin{equation} \label{eq:distancebound} 
\| \Pi_{K(\th)}^c - \Pi_{K(\th)}^\Gamma \|_{\rho - 2 \delta,\underline c, \Gamma}
\le  C\delta^{-1} \| E \|_{\rho,\underline c, \Gamma} 
\end{equation}
for every $\delta \in (0,\rho/2)$ and where $C$, as usual, depends on the non-degeneracy constants of 
the problem. 
\end{pro} 

\begin{proof} 
{From} \eqref{e} and Cauchy estimates (see Lemma \ref{Cauchyestimates} in Appendix \ref{decayfunctions}), we 
have:
\[
\dist_{\rho-\delta,\underline c, \Gamma} ( (DF_\lambda \circ K) \Gamma_{K(\th)}, \Gamma_{K(\th)} \circ T_\omega)  
\le C \delta^{-1} \| E \|_{\rho,\underline c, \Gamma},
\]
where $dist$ stands for the distance between two spaces at the Grassmannian level. 
Using again equation \eqref{e} and iterating it, we obtain for $n \geq 1$

$$
DF_\lambda(K(\th + n \omega )) \times \cdots \times 
DF_\lambda(K(\th))\tilde{M}(\th) = 
$$
$$
\tilde{M}(\th + n \omega) \S_\lambda(\th + (n-1) \omega) \times \cdots \times \S_\lambda(\th) 
+ R_n,
$$
where 
$$
\| R_n \|_{\rho - \delta,\underline c, \Gamma} \le C_n \delta^{-1} \| E\|_{\rho,\underline c, \Gamma}
$$
and $C_n$ depends on $n$.

Since $\S_\lambda(\th)$ is upper triangular with Id$_l$ on the diagonal, 
we have:
\[
\S_\lambda(\th + (n-1) \omega)\times  \cdots \times \S_\lambda(\th)  = 
\begin{pmatrix} 
& \Id_l  & A_\lambda(\th + (n-1) \omega) + \cdots +A_\lambda (\th) \\
& 0 & \Id_l 
\end{pmatrix} .
\]

Therefore, by induction, we have for every $n \in  \nat$
\[
\| DF_\lambda(K(\th + n \omega )) \cdots 
DF_\lambda(K(\th))\tilde{M}(\th)  \|_{\rho - \delta,\underline c, \Gamma} 
\le 
C n + C_n \delta^{-1} \| E\|_{\rho,\underline c, \Gamma}.  
\]
Identical calculations give that 
\[
\| DF_\lambda^{-1}(K(\th - n \omega )) \cdots 
DF_\lambda^{-1} (K(\th))\tilde{M}(\th+\omega)  \|_{\rho - \delta,\underline c, \Gamma} 
\le 
C n + C_n \delta^{-1} \| E\|_{\rho,\underline c, \Gamma}.
\]

Note that, given any $\mu_3 > 1 $ (as in Definition \ref{NDloc-hyp}),  there exists an integer $n_{\mu_3} \geq 0$ such that for all $n \geq n_{\mu_3}$, 
we have $C n <   \mu_3 ^n $. Consequently, choosing such $n_{\mu_3}$ there exists a constant $C$ such that if the error satisfies
$$\delta^{-1} \| E\|_{\rho,\underline c, \Gamma} \leq C,$$
 we have 
$C n + C_n \delta^{-1} \| E\|_{\rho,\underline c, \Gamma} <   \mu_3^n $. 
In other words, the above estimates hold for all sufficiently large 
$n$, provided that we impose a suitable smallness condition on
$\delta^{-1} \| E\|_{\rho,\underline c, \Gamma}$.

As a consequence, $\Gamma_{K(\th)}$ is an approximately
invariant bundle, and we also have bounds on the rate of 
growth of the co-cycle both in positive and negative times. 
Using  Proposition~\ref{NDmoveloc}, this shows that indeed 
one can find a true invariant subspace $\tilde \E_{K(\th)}$ close to $\Gamma_{K(\th)}$. Since 
this invariant subspace should be of the same dimension of 
the center space $\E^c_{K(\th)}$, we deduce that 
$$\tilde \E_{K(\th)}=\E^c_{K(\th)}.$$ 
\end{proof}

\subsubsection{Estimates on the center subspace}
We recall the projection into the center subspace of the linearized
equation
\begin{equation}\label{linTr}
\Pi^c_{K(\th+\omega)}\frac{\partial F_\lambda}{\partial
\lambda}(K(\th)) \Lambda+DF_\lambda (K(\th)) \Delta^c(\th)
-\Delta^c(\th+\omega)=-E^c(\th).
\end{equation}
To shorten the notation till the end of the section we will
write\break $\frac{\partial F_\lambda}{\partial \lambda}$
$(K(\th))\Lambda$ instead of $\Pi^c_{K(\th+\omega)}\frac{\partial
F_\lambda}{\partial\lambda} (K(\th))\Lambda$.

We introduce the new function $W(\th)$ through
\begin{equation}\label{changeVar}
\Delta^c(\th)=\tilde{M}(\th)W(\th)+ \hat e(\th)  W(\th), 
\end{equation}
where 
\begin{equation}\label{hate} 
\hat e = \Pi_{K(\th+\omega)}^c - \Pi_{K(\th+\omega)}^\Gamma
\end{equation}  which 
was estimated in Proposition~\ref{prop:distance}. 

Substituting \eqref{changeVar} into equation \eqref{linTr} we get 
\begin{align}\label{change1-1}
& DF(K(\th))  \tilde{M}(\th)W(\th)  -
\tilde{M}(\th+\omega)W(\th+\omega) \\
& =-E^c(\th)-\frac{\partial F_\lambda}{\partial
\lambda}(K(\th)) \Lambda
+\hat e(\th+\omega) W(\th+\omega) - DF(K(\th)) \hat e(\th) W(\th) \nonumber
\end{align}

We anticipate that the term $\hat e W$ will be quadratic in the error. Similarly, writing 
$$
\frac{\partial F_\lambda}{\partial
\lambda}(K(\th))=\Pi^\Gamma_{K(\th+\omega)}\frac{\partial F_\lambda}{\partial
\lambda}(K(\th))+\hat e \frac{\partial F_\lambda}{\partial
\lambda}(K(\th)).
$$ 
we also anticipate that the term $\hat e \frac{\partial F_\lambda}{\partial
\lambda}(K(\th)) \Lambda$ will be quadratic 
in the error. 
As a consequence, we 
will ignore these two terms and the equation for $W$ is
\begin{align}\label{change2}
DF_{\lambda}(K(\th))
\tilde{M}(\th)W(\th)-\tilde{M}(\th+\omega)W(\th+\omega)=-E^c(\th)-\frac{\partial
F_\lambda}{\partial \lambda}(K(\th)) \Lambda.
\end{align}

Using $e(\th)$ above and multiplying equation \eqref{change2} by 
$\tilde{M}(\th+\omega)^\top  J^c(K(\th+\omega))$,  using Lemma \ref{repApp} (giving the invertibility of $\tilde{M}^\top J^c(K) \tilde{M}$) and
equation \eqref{e}, we end up with 
\begin{eqnarray}\label{eqSD1}
\left[ \begin{pmatrix} \Id_l & A_\lambda (\th)\\ 0_l & \Id_l
\end{pmatrix}+B(\th)\right]W(\th)-W(\th+\omega)=p_1(\th)+p_2(\th)\\
-[\tilde{M}^\top J^c(K)
\tilde{M}](\th+\omega)^{-1}[\tilde{M}^\top J^c(K)](\th+\omega) \frac{\partial
F_\lambda}{\partial \lambda}(K(\th))\Lambda,\nonumber 
\end{eqnarray}
where 
\begin{equation}\label{b}
B(\th)=[\tilde{M}^\top J^c(K)
\tilde{M}](\th+\omega)^{-1}[\tilde{M}^\top J^c(K)](\th+\omega)
e(\th),
\end{equation}

\begin{equation}\label{p1}
p_1(\th)=-V(\th+\omega)^{-1}[\tilde{M}^\top J^c(K)](\th+\omega)
E^c(\th) 
\end{equation}
and
\begin{equation}\label{p2}
p_2(\th)=-\tilde{V}(\th+\omega)[\tilde{M}^\top J^c(K)](\th+\omega) E^c(\th).
\end{equation}
The next result provides the estimates of the previously introduced
quantities.

\begin{lemma}\label{repres}
Assume $\omega \in D(\kappa,\nu)$ and $\delta$ and $\|E\|_{\rho,\underline{c},\Gamma}$ satisfy \eqref{smallnessdeltaE1}.
Then using the 
linear change of variables \eqref{changeVar}, 
equation \eqref{change2} becomes
\begin{equation}\label{eqSD}
\begin{split}
&\left[\S_\lambda(\th)+B(\th) \right] W(\th)-W(\th+\omega) = p_1(\th)+p_2(\th) \\
& -[\tilde{M}^\top J^c(K)
\tilde{M}](\th+\omega)^{-1}[\tilde{M}^\top J^c(K)](\th+\omega) \frac{\partial
F_\lambda}{\partial \lambda}(K(\th))\Lambda,\nonumber
\end{split}
\end{equation}
where $B$, $p_1$ and $p_2$ are given by equations \eqref{b},
\eqref{p1} and \eqref{p2} respectively.

Moreover the following estimates hold: for $p_1$ we have
\begin{equation}\label{estimp1}
\|p_1\|_{\rho,\underline{c},\Gamma} \leq C \|E\|_{\rho,\underline{c},\Gamma},
\end{equation}
where $C$ only depends on $\|J^c(K)\|_{\rho,\Gamma}$, $\|N\|_{\rho}$,
$\|DK\|_{\rho,\underline{c},\Gamma}$. For $p_2$ and
$B$, we have
\begin{equation}\label{estimp2}
\|p_2\|_{\rho-2\delta,\underline{c},\Gamma} \leq C \kappa \delta^{-(\nu+1)}
\|E\|^2_{\rho,\underline{c},\Gamma}
 \end{equation}
 and
\begin{equation}\label{estimb}
\|B\|_{\rho-2\delta,\Gamma} \leq C   \delta^{- 1 }\|E\|_{\rho,\underline
c,\Gamma},
 \end{equation}
where $C$ depends $l$, $\nu$, $\rho$, $R$, $\|N\|_{\rho}$,
$\|DK\|_{\rho,\underline{c},\Gamma}$, $\|F_\lambda\|_{C_\Gamma^1(B_r)}$,
$\|J^c(K)\|_{C^1(B_r)}$.
\end{lemma}

\begin{proof}
We have basically to estimate $\tilde M(\th+\omega)^\top J^c(K(\th+\omega)) E^c(\th)$ and
$\tilde M(\th+\omega)^\top J^c(K(\th+\omega)) e(\th)$ and then use Lemma \ref{repApp}. First we bound
\begin{align*}
|E^c_i(\th)|&\le
\sum_{j\in\Z^N}\,|(\Pi^c_{K(\th+\omega)})_{ij}|\
|E_j(\th)|\\
&\le \sum_{j\in\Z^N}\,\|\Pi^c_{K(\th+\omega)}\|_{\rho,\Gamma}\
\Gamma(i-j)\ \|E\|_{\rho,\underline c,\Gamma}\ \max_k\
\Gamma(j-c_k)\\[-3mm]
&\le \|\Pi^c_{K(\th+\omega)}\|_{\rho,\Gamma}\
\|E\|_{\rho,\underline c,\Gamma}\ \sum^R_{k=1}\Gamma(i-c_k)\,.
\end{align*}
For $1\le i\le l$ we have, taking into account that $J^c$
is uncoupled, 

\begin{equation}\label{mtildei}
|\big(\tilde M(\th+\omega)^\top\,
J^c(K) E^c(\th)\big)_i | \leq C \sum_{j\in\Z^N}\,
|\partial_{\th_i}\,K_j(\th+\omega)| \,\,|E^c_j(\th)|. 
\end{equation}
We estimate from above by
\begin{equation} 
\begin{split} 
\sum_{j\in\Z^N}\, &|D_{\th_i}\,K_j(\th+\omega)|\
|E^c_j(\th)|\\
&\le \sum_j\,\|DK\|_{\rho,\underline c,\Gamma}\ \max_m
\,\Gamma(j-c_m)\,\|\Pi^c_{K(\th+\omega)}\|_{\rho,\Gamma}\,\|E\|_{\rho,\underline
c,\Gamma}\,\sum^R_{k=1}\Gamma(j-c_k)\\
&\le R\, \|DK\|_{\rho,\underline c,\Gamma}\
\|\Pi^c_{K(\th+\omega)}\|_{\rho, \Gamma}\,
\|E\|_{\rho,\underline c,\Gamma}\,.
\end{split} 
\end{equation} 

For $l+1\le i\le 2l$, one gets
\begin{eqnarray}\label{mtildeidos}
|\big(\tilde M(\th+\omega)^\top J^c(K(\th))\,
E^c(\th)\big)_i | \leq \\
C |\Big(N(\th+\omega)^\top
DK(\th+\omega)^\top\,\tilde
J^c\big(K(\th+\omega)^\top\big)\,E^c(\th)\Big)_i |. \nonumber
\end{eqnarray}

We get a similar bound for \eqref{mtildeidos} taking into account
that $N$ is a bounded finite dimensional matrix. 
Now the bounds \eqref{estimp1} and \eqref{estimp2} follow
immediately from Lemma \ref{repApp}. For the estimate on $B$ we  use Cauchy estimates
for $e_1(\th)$. From
\begin{equation*}
B(\th)=(V(\th+\omega)^{-1}+\tilde{V}(\th+\omega))\tilde{M}(\th+\omega)^\top
e(\th)
\end{equation*}
we have
\begin{equation*}
\begin{split}
\|B\|_{\rho-2\delta}&\leq \|V(\th+\omega)^{-1}\|_{\rho-2\delta}
\|\tilde{M}(\th+\omega)^\top e(\th)\|_{\rho-2\delta}\\
&\quad +\ |\tilde{V}(\th+\omega)\tilde{M}(\th+\omega)^\top
e(\th)\|_{\rho-2\delta}
\end{split}
\end{equation*}
and using estimate \eqref{invEst} we end up with
\begin{equation*}
\|B\|_{\rho-2\delta,\Gamma} \leq C\delta^{-1}
\|E\|_{\rho,\underline{c},\Gamma}+\kappa
\delta^{-(\nu+1)}\|E\|_{\rho,\underline{c},\Gamma}\ \delta^{-1}
\|E\|_{\rho,\underline{c},\Gamma}.
\end{equation*}
This gives the desired result.
\end{proof}

\subsubsection{Approximate solvability of the linearized equation on the center subspace}
\label{sec:temp}

In this section we find a solution of equation \eqref{eqSD1} up to
quadratic error. The convergence of the Newton scheme is of course not
affected (see \cite{Zehnder75a}).

For that we introduce the following operator
\begin{equation*}
\L W(\th)=\begin{pmatrix} \Id_l & A_\lambda(\th)\\ 0_l & \Id_l \end{pmatrix}W(\th)-W(\th+\omega).
\end{equation*}
Then equation \eqref{eqSD1} can be written as
\begin{align}\label{sd4}
\L W(\th)&+B(\th)W(\th)=p_1(\th)+p_2(\th)\\
& -[\tilde{M}^\top  J^c(K)
\tilde{M}](\th+\omega)^{-1}[\tilde{M}^\top J^c(K)](\th+\omega) \frac{\partial
F_\lambda}{\partial \lambda}(K(\th))\Lambda .\nonumber
\end{align}
We will reduce equation \eqref{sd4} to two small divisors equations.
Generically their right-hand sides will not have zero average, but
we will use the freedom in choosing $\Lambda$ and in fixing the
average of the solution to solve one after the other. By Lemma
\ref{repApp} we can write
\begin{equation*}
[(\tilde{M}^\top
J^c(K)\tilde{M})^{-1}\tilde{M}^\top J^c(K)](\theta+\omega) \frac{\partial F_\lambda
(K(\th))}{\partial \lambda}\Lambda=H(\th)\Lambda+q(\th)\Lambda,
\end{equation*}
where the $2l \times l$ matrix $H$ is
\begin{equation*}
H(\th)=V(\th+\omega)^{-1}\tilde{M}(\th+\omega)^\top J^c(K(\theta+\omega))  \frac{\partial F_\lambda (K(\th))}{\partial \lambda}
\end{equation*}
and $q$ satisfies for all $\delta \in (0,\rho/2)$
\begin{equation*}
\|q\|_{\rho-2\delta,\Gamma}\leq C \kappa \delta^{-(\nu+1)}
\Big\|\frac{\partial F_\lambda (K(\th))}{\partial
\lambda}\Big\|_{\rho,\underline{c},\Gamma}  \,
\|E\|_{\rho,\underline{c},\Gamma},
\end{equation*}
where the constant $C$ depends on $l$, $\nu$, $\rho$, $R$,
$\|N\|_{\rho}$, $\|DK\|_{\rho,\underline{c},\Gamma}$, $\|F
\|_{C_\Gamma^1(B_r)}$, $\|J^c(K)\|_{C^1(B_r)}$.

We will take as an  approximate solution the solution $v$ of
\begin{equation}\label{sd2Approx}
\L v(\th)=p_1(\th)-H(\th)\Lambda
\end{equation}
obtained from \eqref{sd4} by removing the terms containing $B,
p_2$ and $q$.

\begin{pro}\label{approximateSol}
Assume $\omega \in D(\kappa,\nu)$ and $(\lambda,K)$ is a
non-degenerate pair (i.e. $(\lambda,K) \in ND_{loc}(\rho,\Gamma)$). If the error
$\|E\|_{\rho,\underline{c},\Gamma}$ satisfies \eqref{smallnessdeltaE1}, there exist a mapping $v\in \mathcal A_{\rho-2\delta,\underline{c},\Gamma}$ for any $0 < \delta <\rho/2$
and a vector $\Lambda \in \mathbb{R}^{l}$
 solving equation \eqref{sd2Approx}.

Moreover there exists a constant $C>0$ depending on $\nu, \rho, l,
R, \|K\|_{\rho,\underline{c},\Gamma},$ $|\avg(Q_\lambda)|^{-1}$,
$|\avg(A_\lambda)|^{-1}$, $\|N\|_{\rho}$ and
$\|J^c(K)\|_{\rho}$ such that
\begin{equation}\label{cotav}
\|v\|_{\rho-2\delta,\underline{c},\Gamma} <C \kappa^2 \delta^{-2\nu}
\|E\|_{\rho,\underline{c},\Gamma}
\end{equation}
and
\begin{equation*}
|\Lambda| <C \|E\|_{\rho,\underline{c},\Gamma}.
\end{equation*}
\end{pro}
\begin{proof}  We denote $T(\th)$ the right-hand side of equation
\eqref{sd2Approx}, i.e. we have to solve
\begin{equation}\label{eqTemp}
\L v(\th)=T(\th),
\end{equation}
with
\begin{equation*}
T=p_1-H\Lambda.
\end{equation*}
We now decompose equation \eqref{eqTemp} into two equations. Writing
$v=(v_1,v_2)^\top $, $T(\th)=(T_1(\th), T_2(\th))^\top $
equation \eqref{eqTemp} is equivalent to
\begin{eqnarray}
v_1(\th)+A_\lambda(\th)v_2(\th)=v_1(\th+\omega)+T_1(\th),\label{eqforv1}\\
v_2(\th)=v_2(\th+\omega)+T_2(\th).\label{eqforv2}
\end{eqnarray}
A simple computation shows that 
\begin{equation*}
T_2(\th) = - [DK^\top J^c(K)] \circ T_\omega \circ (E_2^c +  \frac{\partial F_\lambda (K(\th))}{\partial \lambda}\Lambda)
\end{equation*}
We begin by solving equation \eqref{eqforv2}. To apply Proposition
\ref{sdrussmann} we choose $\Lambda\in\R^l$ such that
$$\avg (T_2)=0.$$

This condition is equivalent to
\begin{equation*}
\avg \Big(DK^\top(\omega+\th) J^c(K(\omega+\th))
(E^c(\th)+ \frac{\partial F_\lambda (K(\th))}{\partial \lambda} \Lambda )\Big)=0.
\end{equation*}
This leads to 
\begin{align*}
\avg\Big((DK^\top(\omega+\th) & J^c(K(\omega+\th)) \frac{\partial F_\lambda (K(\th))}{\partial \lambda}\Big)\Lambda\\ = &
    -\avg \Big (DK^\top(\omega+\th) J^c(K(\omega+\th)) E^c(\th)\Big).    
    \end{align*}
Note that the matrix which applies to $\Lambda$ is the average of
$Q_\lambda$ which, by hypothesis, is invertible. This gives $|\Lambda|<C\,\|E\|_{\rho,\underline c,\Gamma}$.

{From} the expression of $T$ and the value of $\Lambda$ obtained
above, we have that there exists a constant $C$ such that
\begin{equation*}
\|T_i\|_{\rho,\underline{c},\Gamma} \leq C
\|E\|_{\rho,\underline{c},\Gamma},
\end{equation*}
 for $i=1,2$.

 Then Proposition \ref{sdrussmann} provides us with an analytic
 solution $v_2$ on $D_{\rho-\delta}$ with arbitrary average and
\begin{equation}\label{estimv2}
\|v_2\|_{\rho-\delta,\underline{c},\Gamma} \leq C \kappa \delta^{-\nu}
\|T_2\|_{\rho,\underline{c},\Gamma}+|\avg (v_2)|.
\end{equation}

Now we come to equation \eqref{eqforv1}. To apply Proposition
\ref{sdrussmann} we choose $\avg (v_2)$ such that $\avg
(T_1-A_\lambda v_2)=0$. This condition is equivalent to
\begin{equation*}
\avg (v_2)=\avg(A_\lambda)^{-1}(\avg (T_1)-\avg (A_\lambda
v^{\perp}_2)),
\end{equation*}
where $v_2=v_2^{\perp}+\avg (v_2)$. This is possible since by the
twist condition $\avg(A_\lambda) $ is invertible.

We have that $$|\avg (v_2)|\le
C\kappa\delta^{-\nu}\,\|E\|_{\rho,\underline c,\Gamma}.$$

Then we take $v_1$ as the unique analytic solution of
\eqref{eqforv1} with zero average. Furthermore, we have the
estimate
\begin{equation*}
\|v_1\|_{\rho-2\delta,\underline{c},\Gamma} \leq C \kappa \delta^{-\nu} \|T_1-A_\lambda
v_2\|_{\rho-\delta,\underline{c},\Gamma}.
\end{equation*}

Collecting the previous bounds we get the result.
\end{proof}

We now come back to the solutions of \eqref{eqCenter}. The above
procedure allows us to prove the following proposition, providing an approximate solution of the projection of $D_{\lambda,K}\mathcal{F}_{\omega}(\lambda,K)(\Lambda,\Delta)=-E$ on the center subspace.

\begin{pro}\label{solCenterloc}
Let $(\Lambda,W)$ be as in Proposition \ref{approximateSol}
and assume the hypotheses of that proposition hold. Define $\Delta^c(\th)=\tilde{M}(\th)W(\th)$.
Then, equation \eqref{eqCenter} is approximately solvable and we have the following estimates
\begin{equation}\label{cotadelta}
\|\Delta^c\|_{\rho-2\delta,\underline{c},\Gamma} \leq C \kappa^2
\delta^{-2\nu} \|E\|_{\rho,\underline{c},\Gamma},
\end{equation}
\begin{equation*}
\quad|\Lambda| \leq C \|E\|_{\rho,\underline{c},\Gamma},
\end{equation*}
where the constant $C$ depends on $\nu, \rho, l, R,
|\avg(Q_\lambda)|^{-1}$, $|\avg(A_\lambda)|^{-1}$, $\|N\|_{\rho}$,
$\|\frac{\partial F_\lambda(K)}{\partial
\lambda}\|_{\rho,\underline{c},\Gamma}$ and
$\|J^c(K)\|_{\rho}$ and
\begin{equation}\label{estimApprox}
\|D_{\lambda,K}\mathcal{F}_{\omega}(\lambda,K)(\Lambda,\Delta^c)+E^c\|_{\rho-2\delta,\underline{c},\Gamma}
\leq C \kappa^2 \delta^{-(2\nu+1)}
(\|E\|^2_{\rho,\underline{c},\Gamma} +
\|E\|_{\rho,\underline{c},\Gamma} |\Lambda|),
\end{equation}
where the constant $C$ depends on $l$, $\nu$, $\rho$, $R$,
$\|F\|_{C_\Gamma^1(B_r)}$, $\|DK\|_{\rho,\underline{c},\Gamma}$,
$\|N\|_{\rho}$, $|\avg (A_\lambda)|^{-1}$, $|\avg
(Q_\lambda)|^{-1}$ and $\|\frac{\partial F_\lambda(K)}{\partial
\lambda}\|_{\rho,\underline{c},\Gamma}$.
\end{pro}
\begin{proof}  For the first estimate we take $\th \in
D_{\rho-2\delta}$ and write $W = (W_1, W_2)$. Then
\begin{equation*}
\Delta^c(\th) = DK(\th) W_1(\th) + J^c(K(\th))^{-1}
DK(\th) N(\th) W_2(\th).
\end{equation*}
We have
\begin{equation*}
\begin{split}
\big\vert(DK(\th) W_1(\th))_i\big\vert &= \Big\vert
\sum^l_{j=1} D_{\th_j} K_i(\th) W_{1,j}(\th) \Big\vert\\&
\leq \| DK\|_{\rho,\underline{c},\Gamma}\, \max_{k}\,
\Gamma(i-c_k) \| W_1\|_{\rho-2\delta}.
\end{split}
\end{equation*}
Also, since $ J_\infty$ is uncoupled and $N$ is finite dimensional,
\begin{equation*}
\begin{split}
&\Big\vert\big(J^c(K(\th))^{-1} DK(\th) N(\th)
W_2(\th)\big)_i\Big\vert \leq \| J(K)\|_\rho
\Big\vert\big(DK(\th) N(\th) W_2(\th)\big)_i\Big\vert\\
&\qquad \leq \| J(K)\|_\rho \| N\|_\rho \|
DK\|_{\rho,\underline{c},\Gamma}\, \max\, \Gamma(i-c_k) \|
W_2\|_{\rho-2\delta}.
\end{split}
\end{equation*}
Now using \eqref{cotav}, we obtain \eqref{cotadelta}.

For \eqref{estimApprox}, using the previous notations and Lemma
\ref{repres},
\begin{equation}
\begin{split}
D_{\lambda,k}\, &\F_\omega (\lambda,K)(\Lambda, \Delta_c) (\th) +
E^c(\th)\label{formuladcalf}\\
&= \frac{\partial F_\lambda}{\partial \lambda} (K(\th)) \Lambda +
\tilde M(\th+ \omega) \big[\S_\lambda(\th) W(\th) - W(\th +
\omega)\big]\\&\quad + e(\th) W(\th) + E^c(\th) =
\frac{\partial F_\lambda}{\partial \lambda} (K(\th))
\Lambda\\&\quad + \tilde M (\th + \omega) \big[p_1(\th)
-Q(\th) \Lambda\big] + e(\th) W(\th)+ E^c(\th).
\end{split}
\end{equation}
Note that
\begin{align*}
&\tilde M (\th +\omega)p_1(\th) = \\
&\tilde M (\th +\omega)
\big[-([\tilde M^\top  J^c(K) \tilde M)^{-1}] (\th+\omega)[\tilde M^\top
J^c(K)](\th + \omega) E^c(\th) - p_2(\th)\big]
\end{align*}
and also that $\tilde M (\tilde M^\top J^c(K)  \tilde M)^{-1}\tilde M^\top J^c(K) $
is symmetric and
$$\tilde M  \tilde M^\top J^c(K)  \tilde M)^{-1}\tilde
M^\top J^c(K) - \Id $$ maps  the vectors of the center subspace to zero
because it is generated by the columns of $\tilde M$.

Also we have 
 \begin{align*}
& \tilde M(\th+\omega)Q_\lambda (\th) \Lambda = \\
&\tilde M (\theta+\omega) \Big [ [\tilde M^\top J^c(K) \tilde
M)^{-1}\tilde M^\top J^c(K)] (\th+\omega) \frac{\partial
F_\lambda}{\partial \lambda} (K(\th))\Lambda -
q(\th)\Lambda\Big]. \end{align*}

We recall that here the derivative of $F_\lambda$ with respect to $\lambda$
actually means the projection of it into the center subspace.

Therefore \eqref{formuladcalf} becomes
$$\tilde M(\th+\omega) [q(\th) \Lambda - p_2(\th)] + e(\th) W(\th)$$
and \eqref{estimApprox} follows.
\end{proof}

 \subsection{Solution of the equation in  the hyperbolic subspaces}
\label{sec:hyperbolicequation} 

In this section, we study the projection of the Newton 
equation \eqref{linear} on the 
the hyperbolic spaces. 

According to the splitting \eqref{splittingReseau},
 there exist projections on the linear spaces
 $\mathcal{E}^s_{{K(\th)}}$ and $\mathcal{E}^u_{{K(\th)}}$. The
analytic regularity of the splitting implies the analytic dependence
of these projections in $\th$. We denote
$\Pi^s_{K(\th)}$ (resp. $\Pi^u_{K(\th)}$) the
projections on the stable (resp. unstable) invariant subspace.

We project equation \eqref{linear} on the stable and unstable subspaces to obtain
\begin{equation}\label{eqStable}
\Pi^s_{K(\th+\omega)}\Big(\frac{\partial
F_\lambda(K(\th))}{\partial\lambda}\Lambda +DF_{\lambda}(K
(\th))\Delta(\th)-\Delta(\th+\omega)\Big)=
-\Pi^s_{K(\th+\omega)}E(\th),
\end{equation}
\begin{equation}\label{eqUnstable}
\Pi^u_{K(\th+\omega)}\Big( \frac{\partial
F_\lambda(K(\th))}{\partial\lambda}\Lambda +DF_{\lambda} (K
(\th))\Delta(\th)-\Delta(\th+\omega)\Big)=
-\Pi^u_{K(\th+\omega)}E(\th).
\end{equation}
The invariance of the splitting reads
\begin{equation*}
\Pi^s_{K(\th+\omega)}DF_{\lambda}(K (\th))
\Delta(\th)=DF_{\lambda}(K(\th))\Pi^s_{K(\th)}\Delta(\th)
\end{equation*}
for the stable part and
\begin{equation*}
\Pi^u_{K(\th+\omega)}DF_{\lambda}(K(\th))\Delta(\th)=DF_{\lambda}
(K (\th))\Pi^u_{K(\th)}\Delta(\th)
\end{equation*}
for the unstable one.

We define
$$\Delta^{s,u}(\th)=\Pi^{s,u}_{K(\th)}\Delta(\th).$$
Using the notation $\th' = \th +\omega$, equations
\eqref{eqStable}-\eqref{eqUnstable} become
\begin{equation}\label{eqStable2}
DF_{\lambda}(K)\circ T_{-\omega}(\th')\Delta^s(T_{-\omega}(\th'))-\Delta^s(\th')
=-\tilde{E}^s(\th',\Lambda) ,
\end{equation}
where
\begin{equation*}
\tilde{E}^s(\th',\Lambda)=\Pi^s_{K(\th')}\Big( \frac{\partial
 F_\lambda(K(T_{-\omega}(\th')))}{\partial\lambda}\Lambda\Big)+\Pi^s_{K(\th')}E \circ
T_{-\omega}(\th')
\end{equation*}
and
\begin{equation}\label{eqUnstable2}
DF_{\lambda}(K)\circ T_{-\omega}(\th')\Delta^u(T_{-\omega}(\th'))-\Delta^u(\th')
=-\tilde{E}^u(\th',\Lambda),
\end{equation}
where
\begin{equation*}
\tilde{E}^u(\th',\Lambda)=\Pi^u_{K(\th')}\Big(\frac{\partial
F_\lambda(K(T_{-\omega}(\th')))}{\partial\lambda}\Lambda\Big)+\Pi^u_{K(\th')}E\circ
T_{-\omega}(\th').
\end{equation*}

Contrary to the projections on the center subspace, the projections
on the hyperbolic subspaces can be solved exactly. The key point for
the estimates is the Banach algebra property of the decay functions.
\begin{pro}\label{hyperbloc}
Fix $\rho>0$. Then equation \eqref{eqStable2}
(resp. \eqref{eqUnstable2}) admits a unique analytic solution
$\Delta^s:D_{\rho} \rightarrow \mathcal{E}^s_{K(\th)}$
(resp. $\Delta^u:D_{\rho} \rightarrow
\mathcal{E}^u_{K(\th)}$). Furthermore there exists a constant $C$
depending only on the hyperbolicity constant $\mu_1$ (resp. $\mu_2$),
the norm of the projector $\|\Pi^s_{K(\th)}\|_{\rho,\Gamma}$
(resp. $\|\Pi^u_{K(\th)}\|_{\rho,\Gamma}$) and $\|\frac{\partial
  F_{\lambda}(K)}{\partial \lambda}\|_{\rho,\underline{c},\Gamma}$
such that
\begin{equation}\label{estimHyperbloc}
\|\Delta^{s,u}\|_{\rho,\underline{c},\Gamma} \leq C
(\|E\|_{\rho,\underline{c},\Gamma}+|\Lambda|).
\end{equation}
\end{pro}
\begin{proof}
We only give the proof for the stable case, the unstable
case being very similar. Using equation \eqref{eqStable2}, we claim
\begin{equation}\label{seriesloc}
\Delta^s(\th')=\displaystyle{\sum_{k=0}^{\infty}}(DF_{\lambda}(K)\circ
T_{-\omega}(\th')\times\dots \times DF_{\lambda}(K)\circ T_{-k\omega}(\th'))
\tilde{E}^s(T_{-k\omega}(\th'),\Lambda).
\end{equation}
We introduce the map
\begin{equation*}
\mathcal{F}^{co}(k, \th')=DF_{\lambda}(K)\circ
T_{-\omega}(\th')\times\dots \times DF_{\lambda}(K)\circ
T_{-k\omega}(\th').
\end{equation*}
{From} the definition of $\tilde E^s$ we have $\|\tilde E^s
(T_{-k\omega}(\th'), \Lambda)\|_{\rho,\underline{c},\Gamma}
\leq C\big(\| E\|_{\rho,\underline{c}, \Gamma} + \vert
\Lambda\vert\big)$. Using\eqref{ndeg1loc} we obtain that the $k$
term in \eqref{seriesloc} is bounded by
\begin{equation*}
\|\mathcal{F}^{co}(k, \th') \tilde E^s(T_{-k \omega}(\th'),
\Lambda)\|_{\rho, \underline{c}, \Gamma} \leq C_h \mu^k_1
\|\tilde E^s (T_{-k\omega}(\cdot),
\Lambda)\|_{\rho,\underline{c}, \Gamma}
\end{equation*}
and therefore
\begin{equation*}
\|\Delta^s\|_{\rho,\underline{c},\Gamma}
 \leq C_h \|\tilde{E}^s\|_{\rho,\underline{c},\Gamma}
 \displaystyle{\sum_{k=0}^{\infty}}\,
 \mu_1^k \leq C \big(\| E\|_{\rho,\underline{c},\Gamma} + \vert
 \Lambda\vert\big),
\end{equation*}
since $\mu_1 <1$.
\end{proof}

\begin{remark}
\label{changes} 
It is perhaps interesting to compare the method of proof of
this paper with that of \cite{FontichLS09}. Both papers use
the invariance equations and formulate a quasi-Newton method
that can be solved using techniques from hyperbolic lore 
and some geometric identities. 
One of the strengths of the set-up based on decay functions 
is that we can obtain 
estimates independent on the number and positions of the centers of activity
by methods that resemble the finite dimensional methods. 

Both \cite{FontichLS09} and the present paper 
use a counterterm to adjust some of the constants
and then prove a vanishing lemma. In \cite{FontichLS09}, the counterterm is
obtained adding $J^{-1}\circ K_0 DK_0 \lambda$. In this 
paper, we consider a family of symplectic maps. This 
allows us to use the vanishing lemma only once at the end of 
the proof, whereas in \cite{FontichLS09}, the vanishing lemma 
had to be used at each iterative step.

We also deal in a different way with the invertibility of 
the linear change of variables $M$. In the present paper, we 
obtain the invertibility on the range using some geometric 
identities, whereas in \cite{FontichLS09}, we used some easier 
argument based on finite dimensional arguments. 

One geometric aspect that required several changes (due in part 
to the changes in the counterterm and to the infinite dimensional 
character) is the estimates on the difference between 
the center space and the range of the change of variables $M$. 
\end{remark}

\section{Iteration of the modified Newton method, convergence and proof of Theorem~\ref{existencetranslatedloc}}\label{convergence}

This section is devoted to the iteration of the Newton method.
 We derive first the usual KAM estimates for convergence. We assume
      that we are under the assumptions of
       Theorem \ref{existencetranslatedloc}.
        We note that with the decay norms the estimates are
         very similar to the ones we obtained in the finite dimensional case (see \cite{FontichLS09}).
\subsection{Iteration of the method}

Let $(\lambda_0, K_0)$ be an approximate solution of
\eqref{translatedloc} (i.e. a solution of the linearized equation
with error $E_{0}$). Following the Newton scheme we define the
following sequence of approximate solutions
\begin{align*}
K_m&=K_{m-1}+\Delta K_{m-1},\qquad m\ge 1,\\
\lambda_m&=\lambda_{m-1}+\Lambda_{m-1},\qquad m\ge 1,
\end{align*}
where $(\Lambda_{m-1},\Delta K_{m-1})$ is a solution of
\begin{equation*}
D_{\lambda,K}\mathcal{F}_{\omega}(\lambda_{m-1},K_{m-1})(\Lambda_{m-1},\Delta K_{m-1})=-E_{m-1}
\end{equation*}
with
$E_{m-1}(\th)=\mathcal{F}_{\omega}(\lambda_{m-1},K_{m-1})(\th)$.
The next lemma states a classical result in KAM theory: the
approximation to the solution at step $m$ has an error which is
bounded in a smaller complex domain by the square of the norm of the
error at step $m-1$.

\begin{pro}\label{improvementloc}
Assume $(\lambda_{m-1},K_{m-1}) \in ND_{loc}(\rho_{m-1},\Gamma)$ is an approximate
solution of equation \eqref{translatedloc}  and that the following holds
\begin{equation*}
r_{m-1}=\|K_{m-1}-K_0\|_{\rho_{m-1},\underline{c},\Gamma} < r\,.
\end{equation*}
If $E_{m-1}$ is small enough such that lemma \ref{solCenterloc}
applies then there exists a function $\Delta K_{m-1} \in
\mathcal{A}_{\rho_{m-1}-3\delta_{m-1},\underline{c},\Gamma}$ for any
$0< \delta_{m-1} < \rho_{m-1} /3$ and a vector $\Lambda_m \in
\mathbb{R}^l$ such that
\begin{equation}\label{improve1loc}
\|\Delta K_{m-1}\|_{\rho_{m-1}-2\delta_{m-1},\underline{c},\Gamma}
 \leq (C^1_{m-1}+C^2_{m-1} \kappa^{2} \delta_{m-1}^{-2\nu}) \|E_{m-1}\|_{\rho_{m-1},\underline{c},\Gamma},
\end{equation}
\begin{equation}\label{improveLamloc}
|\Lambda_m| \leq C \|E_{m-1}\|_{\rho_{m-1},\underline{c},\Gamma}
\end{equation}
\begin{equation}\label{improve2loc}
\|D\Delta K_{m-1}\|_{\rho_{m-1}-3\delta_{m-1},\underline{c},\Gamma}
 \leq (C^1_{m-1}\delta_{m-1}^{-1}+C^2_{m-1} \kappa^{2} \delta_{m-1}^{-(2\nu+1)})
  \|E_{m-1}\|_{\rho_{m-1},\underline{c},\Gamma},
\end{equation}
where $C^1_{m-1}, C^2_{m-1}$ depend only on $\nu$, $l$,
$|F_{\lambda}|_{C^1_\Gamma(B_r)}$,
$\|DK_{m-1}\|_{\rho_{m-1},\underline{c},\Gamma}$,
$||\Pi^c_{K_{m-1}(\th)}||_{\rho_{m-1},\Gamma}$,
$||\Pi^s_{K_{m-1}(\th)}||_{\rho_{m-1},\Gamma}$,
$||\Pi^u_{K_{m-1}(\th)}||_{\rho_{m-1},\Gamma}$,
$|\avg(Q_{\lambda_{m-1}})|^{-1}$ and
$|\avg(A_{\lambda_{m-1}})\vert^{-1}$. Moreover, if
$K_m=K_{m-1}+\Delta K_{m-1}$ and
\begin{equation*}
r_{m-1}+C^1_{m-1}+\Big (C^2_{m-1} \kappa^{2} \delta_{m-1}^{-2\nu} \Big ) \|E_{m-1}\|_{\rho_{m-1},\underline c, \Gamma}<r
\end{equation*}
then we can redefine $C^1_{m-1}$ and $C^2_{m-1}$ and all previous
quantities such that the error
$E_m(\th)=\mathcal{F}_{\omega}(\lambda_m,K_m)(\th)$ satisfies
\begin{equation}\label{estimateloc}
\|E_m\|_{\rho_m,\underline{c},\Gamma} \leq C_{m-1}\kappa^4
 \delta_{m-1}^{-4\nu}
 \|E_{m-1}\|^2_{\rho_{m-1},\underline{c},\Gamma},
\end{equation}
where we take $\rho_m = \rho_{m-1}-3\delta_ {m-1}$.
\end{pro}
\begin{remark}
The estimate \eqref{estimateloc} showing that the norm of the error
at step $m$ is essentially bounded by the square of the norm of the
error at step $m-1$ was already in \cite{Zehnder75a,Zehnder75b}.
\end{remark}
\begin{proof}  Taking into account that $\Delta K_{m-1} (\th)$ is the sum
 of its three projections on the stable, center and unstable
 subspaces
  estimates \eqref{improve1loc} and \eqref{improveLamloc}
follow from Proposition \ref{solCenterloc} and Proposition
\ref{hyperbloc}. Estimate \eqref{improve2loc} follows from estimate
 \eqref{improve1loc} and Cauchy's inequalities.
  Define the remainder of the Taylor expansion
\begin{align*}
\mathcal{R}(\lambda ',\lambda,K',K) =
&\, \mathcal{F}_{\omega}(\lambda,K)-\mathcal{F}_{\omega}(\lambda',K')\\
&-D_{\lambda,K}\mathcal{F}_{\omega}(\lambda',K')(\lambda-\lambda',K-K').
\end{align*}
Then putting $\lambda = \lambda_m$ and $\lambda'=\lambda_{m-1}$ $K=K_{m-1}$ and $K'=K_{m-1}$, we
have
\begin{align*}
E_m(\th)=E_{m-1}(\th)&+D_{\lambda,K}\mathcal{F}_{\omega}
(\lambda_{m-1},K_{m-1}(\th))(\Lambda_{m-1},\Delta
K_{m-1}(\th))\\&+
\mathcal{R}(\lambda_{m-1},\lambda_m,K_{m-1},K_m)(\th).
\end{align*}
According to estimate \eqref{estimApprox} and since the equations on
 the hyperbolic subspace are {\sl exactly} solved, we have
\begin{eqnarray*}
\|E_{m-1}+D_{\lambda,K}\mathcal{F}_{\omega}(\lambda_{m-1},K_{m-1})(\Lambda_{m-1},\Delta
K_{m-1})\|_{\rho_m,\underline{c},\Gamma} \\
\leq c_{m-1} \kappa^2
\delta_{m-1}^{-(2\nu+1)}\|E_{m-1}\|^2_{\rho_{m-1},\underline{c},\Gamma}.
\end{eqnarray*}
Estimate \eqref{estimateloc} then follows from Taylor's remainder.
\end{proof}

In the following, we derive the changes in the non-degeneracy
conditions during the iterative step.

For the twist condition, we have the following lemma, which is 
proved easily noting that we are 
just perturbing finite dimensional matrices. 

\begin{lemma}
Assume that the hypothesis of Proposition \ref{improvementloc} hold.
 If $\|E_{m-1}\|_{\rho_{m-1},\underline c, \Gamma}$ is small enough, then
\begin{itemize}
\item If $DK_{m-1}^\top DK_{m-1}$ is invertible with inverse $N_{m-1}$
  \\then $DK_{m}^\top DK_{m}$ is invertible with inverse $N_m$ and we have
\begin{equation*}
\|N_m\|_{\rho_{m}} \leq \|N_{m-1}\|_{\rho_{m-1}}+C_{m-1} \kappa^2
\delta_{m-1}^{-(2\nu+1)}
 \|E_{m-1}\|_{\rho_{m-1},\underline{c},\Gamma}.
\end{equation*}
\item If $\avg(A_{\lambda_{m-1}})$ is non singular then, $\avg(A_{\lambda_{m}})$
is non-singular  and we have the estimate
 \begin{equation*}
|\avg(A_{\lambda_m})|^{-1}\leq
|\avg(A_{\lambda_{m-1}})|^{-1}+C'_{m-1} \kappa^2
\delta_{m-1}^{-(2\nu+1)}
\|E_{m-1}\|_{\rho_{m-1},\underline{c},\Gamma}.
\end{equation*}
\item If $\avg(Q_{\lambda_{m-1}})$ is non singular then, 
 $\avg(Q_{\lambda_{m}})$ is non-singular  and we have the estimate
 \begin{equation*}
|\avg(Q_{\lambda_m})|^{-1}\leq
|\avg(Q_{\lambda_{m-1}})|^{-1}+C''_{m-1}
 \kappa^2 \delta_{m-1}^{-(2\nu+1)} \|E_{m-1}\|_{\rho_{m-1},\underline{c},\Gamma}.
\end{equation*}
\end{itemize}
\end{lemma}

\subsection{Iteration of the Newton step and convergence}

Once we have the estimates for a step of the iterative method,
following the standard scheme in KAM theory one can prove the
convergence of the method. One takes $0<\delta_0< \min(1,
\rho_0/12),\ \delta_m = \delta_0/2$ and $\rho_m = \rho_{m-1}-
3\delta_{m-1}.$

{From} \eqref{estimateloc} we have that
\begin{equation}
\| E_m\|_{\rho_m,\underline{c}, \Gamma} \le C_{m-1} \kappa^4
\delta^{-4\nu}_0 2 ^{4\nu(m-1)}\|
E_{m-1}\|^2_{\rho_{m-1},\underline{c},\Gamma}\,.\label{formulaem}
\end{equation}
Moreover the constants $C_m$ are bounded uniformly in $m$ by a
constant $C$. Using \eqref{formulaem} iteratively one obtains
$$\| E_m\|_{\rho_m, \underline{c},\Gamma} \le
\Big(C \kappa^4 \delta^{-4\nu}_0 2^{4\nu} \| E_0\|_{\rho_0,
\underline{c},\Gamma}\Big)^{2^m}.$$
 Then, if
$\| E_0\|_{\rho_0,\underline{c},\Gamma}$ is small enough the
iteration converges to a pair $(\lambda_\infty, K_\infty) \in
ND_{loc}(\rho_\infty, \Gamma),$\ with $K_\infty \in
\A_{\rho_\infty, \underline{c}, \Gamma}$ such that
$F_{\lambda_\infty} \circ K_\infty=K_\infty \circ T_\omega$.

\section{Vanishing lemma and existence of invariant tori.
Proof of Theorem~\ref{existenceembedloc}} 
\label{secvanishing}
In this context of infinite dimensional lattices, one could ask for
the existence of a vanishing lemma, which would ensure that the
translated tori are actually invariant ones. The issue is that we do
not have a true symplectic form on the whole manifold $M^{\ZZ^N}$ but just
 a formal one. The proof we give is inspired by  the one
  in \cite{FontichLS09} but we have to take into account that formal forms make sense only via their pull-backs to  $\torus^l$.

The following lemma, called the vanishing lemma,  is more or less 
 equivalent to showing that some averages cancel, 
but is somewhat easier to implement.

\begin{lemma}\label{vanishingloc}
Assume  $F_{\lambda}$ is analytic, smooth in $\lambda \in
\mathbb{R}^l$ and maps $\M$ into itself. Assume $\omega \in
D(\kappa,\nu)$ and let $(\lambda, K) \in ND_{loc}(\rho, \Gamma),$
where $K \in \A_{\rho,\underline{c},\Gamma}$ is a solution of
\begin{equation*}
F_{\lambda} \circ K=K \circ T_{\omega}.
\end{equation*}
Assume furthermore:
\begin{itemize}
\item
$F_0$ is exact symplectic, $F_\lambda$ is symplectic for 
$\lambda \neq 0$ and $F_\lambda$ is constructed as in Appendix \ref{maps}.
\item $F_\lambda$ extends analytically to a neighborhood of
  $K(\torus^l)$.
\item We have $|\lambda| \leq \lambda^*$, where $\lambda^*$ depends only on derivatives
 of $F$ and the symplectic structure $J$.
\end{itemize}
Then
\[
\lambda = 0.
\]
\end{lemma}
\begin{proof}
We will write equation \eqref{translatedloc} as
\begin{equation}\label{embed2loc}
F_0 \circ K =  R_\lambda \circ K\circ T_\omega,
\end{equation}
where $R_\lambda=F_0 \circ F^{-1}_\lambda$.  We denote
\begin{equation}\label{thetahat}
\hth_i = (\th_1, \ldots, \th_{i-1}, \th_{i+1},\ldots,\th_l)\in \T^{l-1}
\end{equation}
and similarly
 $\hat \omega_i = (\omega_1, \ldots, \omega_{i-1},
\omega_{i+1},\ldots, \omega_l)\in \RR^{l-1}$.
We also denote $\sigma_{i, \hth_i}:\torus \to \torus^l$ the path
given by
\begin{equation}\label{sigmai}
\sigma_{i, \hth_i}(\eta) = (\th_1, \ldots, \th_{i-1}, \eta,
\th_{i+1},\ldots, \th_l).
\end{equation}

We will compute the integral
$\int_{\torus^{l-1}}\int_{\sigma_{i, \hth_i}}
K^*F_0^*\alpha_{\infty}$ in two different ways. Note that the
quantity $\int_{\sigma_{i, \hth_i}}  K^*F_0^* \alpha_{\infty}$ is
well-defined since $K$ has decay on $\M$.

Using that $F_0$ is exact symplectic we have:
\begin{equation}\label{calculation1loc}
\begin{split}
\int_{\sigma_{i, \hth_i+\hat \omega_i}} K^*F_0^* \alpha_{\infty} &=
  \int_{\sigma_{i, \hth_i+\hat \omega_i}}  (K^*\alpha_{\infty}  + dW_K) \\
&= \int_{\sigma_{i, \hth_i+\hat \omega_i}}  K^*\alpha_{\infty}
 =\int_{K \circ \sigma_{i, \hth_i+\hat \omega_i}} \alpha_{\infty} .
\end{split}
\end{equation}

Similarly, we have
\begin{equation}\label{calculation2loc}
\begin{split}
\int_{\sigma_{i, \hth_i}} (R_\lambda \circ K \circ T_\omega)^*\alpha_{\infty} &=
  \int_{\sigma_{i, \hth_i}}  T^*_\omega (R_\lambda \circ K)^* \alpha_\infty \\
&= \int_{\sigma_{i, \hth_i+\hat \omega_i}}  (R_\lambda \circ K)^*
\alpha_\infty\,.
\end{split}
\end{equation}

Since the averages over $\torus^{l-1}$ are the same, one gets
\begin{equation}\label{calc3}
\begin{split}
0&=\int_{\torus^{l-1}}\int_{\sigma_{i, \hth_i}}[K^* \alpha_\infty-(R_\lambda \circ K)^*\alpha_\infty]\\
&= -\int_{\torus^{l-1}}\int_{\sigma_{i, \hth_i}}K^*(R^*_\lambda
\alpha_\infty-\alpha_\infty).
\end{split}
\end{equation}

Now, we estimate $K^* \Big (R^*_\lambda \alpha_\infty-\alpha_\infty \Big )$. We know that
\begin{equation*}
R_\lambda=F_0 \circ F_\lambda^{-1   }.
\end{equation*}
Therefore, we have
\begin{equation*}
K^* \Big ( R^*_\lambda \alpha_\infty-\alpha_\infty \Big )=K^* \Big ( (F_\lambda^{-1})^* (F_0^*\alpha_\infty-F_\lambda^* \alpha_\infty) \Big ).
\end{equation*}
This gives, using the exact symplecticness of $F_0$ 
\begin{equation*}
\begin{split}
0=\int_{\torus^{l-1}}\int_{\sigma_{i, \hth_i}}K^* (F_\lambda^{-1})^*(\alpha_\infty-F_\lambda^* \alpha_\infty).
\end{split}
\end{equation*}
By the construction of the map $F_\lambda$, we have
\begin{equation*}
\begin{split}
(F_\lambda^{-1})^*(\alpha_\infty-F_\lambda^* \alpha_\infty)
=\frac{1}{2} (F_\lambda^{-1})^*\Big \{ \sum_{ j \in \J}\sum_{k=1}^l
\lambda^j_k F_0^* \delta^j_k +d\beta \Big \},
\end{split}
\end{equation*}
where $\beta$ is a smooth function on the torus and
 $(\delta^j_k)_{k=1,\dots ,l}$ is a basis of $H^1(\T^l)$. This gives
\begin{equation*}
\begin{split}
(F_\lambda^{-1})^*(\alpha_\infty-F_\lambda^* \alpha_\infty)
=\frac{1}{2} \sum_{j \in \J} \sum_{k=1}^l \lambda^j_k R_\lambda^*
\delta^j_k +d\bar{\beta}.
\end{split}
\end{equation*}
Therefore, one gets
\begin{equation*}
\begin{split}
\frac{1}{2} \sum_{j \in \J} \sum_{k=1}^l \lambda^j_k
\int_{\torus^{l-1}}\int_{\sigma_{i, \hth_i}}K^* R_\lambda^*
\delta^j_k=0.
\end{split}
\end{equation*}
By the smoothness of $F_\lambda$ with respect to $\lambda$ we can
write
\begin{equation*}
R_\lambda=\mbox{Id}+O(|\lambda|).
\end{equation*}
This gives
\begin{equation*}
\begin{split}
0&=\frac{1}{2} \sum_{j \in \J} \sum_{k=1}^l \lambda^j_k \int_{\torus^{l-1}}
\int_{\sigma_{i, \hth_i}}K^* R_\lambda^* \delta^j_k\\
&= \frac{1}{2} \sum_{j \in \J} \sum_{k=1}^l \lambda^j_k
\int_{\torus^{l-1}}\int_{\sigma_{i,
\hth_i}}K^*\delta^j_k+(|\lambda|^2).
\end{split}
\end{equation*}
Since $K$ is an embedding, $\{\delta^k_j\}_{1\le j\le l}$ is a basis
of $H^1(K(\T^l))$. Consequently, the map $v\mapsto
\int_{\torus^{l-1}}\int_{\sigma_{i, \hth_i}}K^*\delta^j_k v$ is
invertible and the smallness assumption on $\lambda$ (which is
satisfied in particular by the KAM theorem) ensures, by the Implicit
Function Theorem, that $\lambda=0$.
\end{proof}

\begin{figure} 
\begin{center} 
\includegraphics[height = 2.5 in]{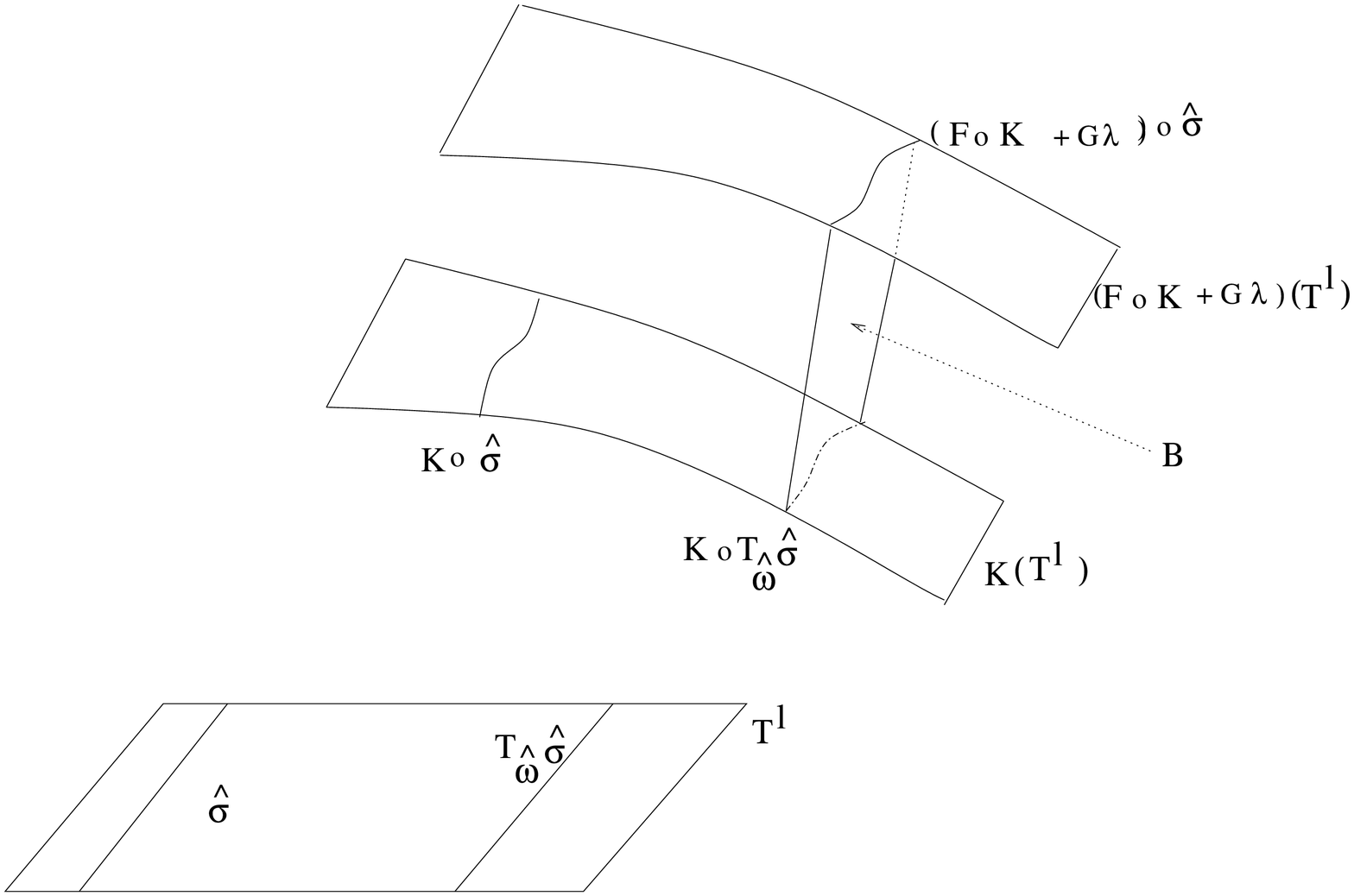} 
\end{center}
\caption{Illustration of the proof of the vanishing 
lemma, Lemma~\ref{vanishingloc}. } 
\end{figure} 

Once we have the vanishing lemma, Theorem \ref{existenceembedloc}
follows directly from Theorem \ref{existencetranslatedloc} by using
the construction in Appendix \ref{maps}. Indeed, starting with the
exact symplectic map $F$ we construct the family $F_\lambda$ to
which we apply Theorem \ref{existencetranslatedloc}, with an
approximate solution $(\lambda_0=0, K_0)$, to obtain
$(\lambda_\infty, K_\infty)$ such that
$$F_{\lambda_\infty}\circ K_\infty = K_\infty\circ T_\omega .$$
The vanishing lemma implies that $\lambda_\infty = 0$.

\section{Uniqueness results}
In this section, we prove Theorem \ref{uniquenessLattice}. We closely follow
the proof in \cite{FontichLS09}. It is based on showing that the operator
$D\mathcal{F}_\omega(K)$ has an approximate left inverse (as in
\cite{Zehnder75a,Zehnder75b}). Notice first that the composition on the right by every
translation of a solution of \eqref{invarmap1} is also a
solution. Therefore, one cannot expect a strict uniqueness
result. Moreover, the second statement in Lemma \ref{main} and the calculation on the hyperbolic directions
show that,
roughly speaking, two solutions of the linearized equation differ by their average.
 Moreover this difference is in the direction of the tangent space of the torus.
The idea behind the local uniqueness result is to
prove that one can transfer the difference of the averages between two
solutions to a difference of phase between the two solutions.

Now
we assume that the embeddings $K_1 $ and $K_2 $
satisfy the hypotheses in Theorem \ref{uniquenessLattice}, in particular $K_1$ and $K_2$ are solutions
of \eqref{invarmap1}, or \eqref{translatedloc} with $\lambda = 0$.
If $\tau\ne 0$ we write $K_1$ for $K_1\circ T_\tau$
which is also a solution.
Therefore
$\mathcal{F}_{\omega}(0,K_1)=\mathcal{F}_{\omega}(0,K_2)=0$.
By Taylor's theorem we can write
\begin{equation}\label{unique}
\begin{split}
0=\mathcal{F}_{\omega}(0,K_1)-\mathcal{F}_{\omega}(0,K_2)=&D_{\lambda,K}\mathcal{F}_{\omega}(0,K_2)(0,K_1-K_2)\\
&+\mathcal{R}(0,0,K_1,K_2).
\end{split}
\end{equation}
Moreover, there exists $C>0$ such that
\begin{equation*}
\|\mathcal{R}(0,0,K_1,K_2)\|_{\rho,\underline{c},\Gamma} \leq C \|K_1-K_2\|_{\rho,\underline{c},\Gamma}^2
\end{equation*}
since $F \in C^2_\Gamma$.
Hence we end up with the following linearized
equation
\begin{equation*}
D_{\lambda,K}\mathcal{F}_{\omega}(0,K_2)(0,K_1-K_2)=-\mathcal{R}(0,0,K_1,K_2).
\end{equation*}
We denote $\Delta=K_1-K_2$.

Projecting this equation on the center subspace, writing
$\Delta^c(\th) = \Pi^c_{K_2(\th)}\Delta (\th)  $ and making the
change of function $\Delta^c(\th)= \tilde M(\th) W(\th) $, where
$\tilde M$ is defined in \eqref{definicioMtilde} with $K=K_2$, we
obtain
\begin{align}\label{change-uniq}
DF(K_2(\th))\tilde{M}(\th)W(\th)&-\tilde{M}(\th+\omega)W(\th+\omega) \nonumber \\
& = - \Pi^c_{K_2(\th+\omega )} \mathcal{R}(0,0,K_1,K_2).
\end{align}
Applying the property $ DF(K_2(\th))\tilde{M}(\th) = \tilde{M}(\th+\omega) \mathcal{S}_0(\th)$
for solutions of \eqref{invarmap1}, multiplying both sides by  $\tilde{M}(\th+\omega)^\top J^c(K)$
and using that $\tilde{M}^\top J^c(K) \tilde{M} $ is invertible we get
$$
\mathcal{S}_0(\th) W(\th) - W(\th+\omega) =$$
$$
- [(\tilde{M}^\top J^c(K)\tilde{M})^{-1} \tilde{M}^\top J^c(K)](\th+\omega) \Pi^c_{K_2(\th+\omega )} \mathcal{R}(0,0,K_1,K_2).
$$
We get bounds for $W$, from the fact that it solves the previous equation,
using the methods in Section \ref{sec:temp}.
We write $W=(W_1, W_2) $. Since $\mathcal{S}_0$ is triangular we begin by looking for $W_2$. We search for it 
in the form $W_2= W_2^\bot + \avg(W_2) $. We have
$\|W_2^\bot\|_{\rho-\delta,\underline{c},\Gamma} = C \kappa \delta ^{-\nu} \|K_1-K_2\|_{\rho,\underline{c},\Gamma}^2 $.
For $W_1$ we have
\begin{align}\label{condW1}
W_1(\th) - W_1(\th+\omega)
= & [N_2DK_2^\top](\th) (\Pi^c_{K_2(\th+\omega )} \mathcal{R}(0,0,K_1,K_2))_1(\th) \nonumber \\
 & - A_0(\th)W^\bot_2(\th) - A_0(\th) \avg(W_2),
\end{align}
where $N_2=\Big ( DK_2^\top DK_2\Big )^{-1}$. 
The condition that 
the right-hand side of \eqref{condW1} has  zero average gives
$|\avg(W_2)  | \le  C\kappa \delta ^{-\nu} (\|K_1-K_2\|_{\rho,\underline{c},\Gamma})^2$. Then
$$
\|W_1 -\avg(W_1) \|_{\rho-2\delta} \le C\kappa^2 \delta ^{-2\nu}
\|K_1-K_2\|_{\rho,\underline{c},\Gamma}^2
$$
but $\avg(W_1) $ is free. Then
\begin{equation*}
\big\|\Delta^c-DK_2\avg(W_1)
\big\|_{\rho-2\delta,\underline{c},\Gamma} \leq C \kappa ^2 \delta
^{-(2\nu+1)} \|K_1-K_2\|_{\rho,\underline{c},\Gamma}^2.
\end{equation*}
The next step  is done in the same way as in \cite{LGJV05}. We quote Lemma 14 of that reference
using our notation.
\begin{lemma} \label{lem:tau}
There exists a constant $C$ such that if $C
\|K_1-K_2\|_{\rho,\underline{c},\Gamma} \leq 1$ then there exits a
phase $\tau_1 \in \left \{ \tau  \in
  \mathbb{R}^l \mid \; |\tau|< \|K_1-K_2\|_{\rho,\underline{c},\Gamma} \right \}$ such that
 \begin{equation*}
\avg\big({N_2}DK_2^\top \Pi^c_{K_2(\th)}(K_1\circ
T_{\tau_1}-K_2)(\th)\big)=0.
\end{equation*}
\end{lemma}
The proof is based on the application of the Banach fixed point theorem in $\mathbb{R}^l$.

As a
consequence of Lemma \ref{lem:tau}, if $\tau_1$ is as in the statement, then $K_1 \circ
T_{\tau_1}$ is a solution of \eqref{invarmap1} such that if
\begin{equation*}
W=(\tilde{M}(\th+\omega)^\top J^c(K_2)\tilde{M}(\th+\omega))^{-1}\tilde{M}(\th+\omega)^\top \Pi^c_{K_2(\th)}(K_1\circ
T_{\tau_1}-K_2),
\end{equation*}
for all $\delta \in (0,\rho/2)$ 
we have the estimate
\begin{equation*}
\|W\|_{\rho-2\delta} <C \kappa^2 \delta^{-2\nu}
\|\mathcal{R}\|_{\rho,\underline{c},\Gamma} \leq C \kappa^2
\delta^{-2\nu} \|K_1-K_2\|^2_{\rho,\underline{c},\Gamma}.
\end{equation*}
This leads to on the center subspace
\begin{equation*}
\|\Pi^c_{K(\th)}(K_1\circ
T_{\tau_1}-K_2)\|_{\rho-2\delta,\underline{c},\Gamma} \leq C
\kappa^2 \delta^{-2\nu} \|K_1-K_2\|^2_{\rho,\underline{c},\Gamma}.
\end{equation*}
Furthermore, we can show that
$\Delta^h=\Pi^h_{K(\th)} (K_1\circ T_{\tau_1}-K_2)$ satisfies the estimate
\begin{equation*}
\|\Delta^h\|_{\rho-2\delta,\underline{c},\Gamma} <C \|\mathcal{R}\|_{\rho,\underline{c},\Gamma}.
\end{equation*}
All in all, we have proven the estimate for $K_1 \circ T_{\tau_1}
-K_2$ (up to a change in the original constants)
\begin{equation*}
\|K_1\circ T_{\tau_1}-K_2\|_{\rho-2\delta,\underline{c},\Gamma} \leq C \kappa^2
\delta^{-2\nu} \|K_1-K_2\|^2_{\rho,\underline{c},\Gamma}.
\end{equation*}
We are now in position to perform the same scheme used in Section \ref{convergence}. We can take a sequence $\left \{ \tau_m \right
\}_{m \geq 1}$ such that $|\tau_1| \le \|K_1  -K_2 \|_{\rho,\underline{c},\Gamma}$ and
\begin{equation*}
|\tau_m -\tau_{m-1}| \leq \|K_1 \circ T_{\tau_{m-1}} -K_2 \|_{\rho_{m-1},\underline{c},\Gamma} , \qquad m\ge 2,
\end{equation*}
and
\begin{equation*}
\|K_1 \circ T_{\tau_{m}} -K_2 \|_{\rho_{m},\underline{c},\Gamma} \leq C \kappa^2
\delta_m^{-2\nu} \|K_1 \circ T_{\tau_{m-1}} -K_2 \|^2_{\rho_{m-1},\underline{c},\Gamma},
\end{equation*}
where $\delta_1=\rho/4$, $\delta_{m+1}=\delta_m/2$ for $m\ge 1$ and
$\rho_0=\rho $, $\rho_m=\rho_0-\sum_{k=1}^m \delta_k$ for $m\ge 1$.
By an induction argument we end up with
\begin{equation*}
\|K_1 \circ T_{\tau_{m}} -K_2 \|_{\rho_{m},\underline{c},\Gamma} \leq
(C\kappa^2\delta^{-2\nu}_1 2^{6\nu} \|K_1-K_2\|_{\rho_0,\underline{c},\Gamma})^{2^m}
2^{-2\nu(3+m)}.
\end{equation*}
Therefore, under the smallness assumptions on $\|K_1-K_2\|_{\rho_0,\underline{c},\Gamma}$,
the sequence $\left \{\tau_m \right \}_{m \geq 1}$ converges and one gets
\begin{equation*}
\|K_1 \circ T_{\tau_\infty} -K_2\|_{\rho /2,\underline{c},\Gamma}=0.
\end{equation*}
Since both $K_1 \circ T_{\tau_\infty} $ and $ K_2$ are analytic in $D_\rho$ and coincide in
$D_{\rho/2}$ we obtain
the result.

\section{Results for flows on lattices}\label{secflows}

In this section, we study vector-fields defined on lattices. The
models we consider -- which have appeared naturally in solid state
physics and in biophysics, see \cite{ChazottesF05} for a review --
consist of a sequence of  copies  of an individual system arranged
in a lattice coupled with their neighbors.

The  basic result  we will prove is
Theorem~\ref{existenceHamillattice}.
The proof is based on applying the result for maps to the time-one
map of the flow. This requires to study the decay properties of
flows generated by Hamiltonian systems with decay, which may have
some independent interest. Again, it is important to emphasize that
the results we prove have an {\sl a posteriori} format, showing that
close to approximate solutions  -- which satisfy some hyperbolicity
and twist conditions -- there is a true solution. We will not need
that the system is close to integrable.

\subsection{Integrating decay vector fields} 
\label{sec:vectorfields}
In this section, we study properties of
 vector-fields with decay. We will show in Proposition \ref{decayODE}
that the flows they generate are families of diffeomorphisms with
decay. As we will see, this is a consequence
of the composition properties of spaces of functions with decay,
which in turn is a property of the Banach algebra properties under
multiplication. This, together with some more delicate study of the
non-degeneracy conditions, will allow us to apply the existence
Theorem~\ref{existenceHamillattice} in the next section
 to the time-one map of the flow for a model problem given by a
 vector-field $X$.

We consider the equation on $\M =\ell^\infty(\ZZ^N)$
\begin{equation}\label{hamilLattice}
\partial_{\omega} K(\th)=X \circ K(\th),
\end{equation}
where $X$ is a vector-field on $\M$ and $K$ maps $\T^l$ into $\M$.

We note that we can deal with systems represented by formal
Hamiltonians which are given by formal sums which do not need to
converge and hence do not define a function, but however their
partial derivatives and therefore the differential equations they determine are
well-defined. Therefore, the invariance equations make sense. This
is one of the reasons why the present method, based on the study of
the invariance equation has advantages over the more classical
methods \cite{Zehnder75b} based on transformations of the
Hamiltonian function.

We prove that decay vector fields generate 
flows $\{ S_t\}_{t \in \real}$ such that all $S_t$ 
are decay diffeomorphisms.  

\begin{pro}\label{decayODE}
Let $X$ be a $C^r$ vector-field, $r\ge 1,$ on an open set $\B
\subset \M$ (recall that we consider $\M$ endowed 
with the $\ell^\infty$ topology) 
and consider the differential equation
\begin{equation}
x' = X(x).\label{odelat}
\end{equation}
Let $\B_1\subset\B$ be an open set 
such that $d(\B_1, \B^c)=\eta > 0$.

Then there exist $T>0$ such that
for all the initial conditions $ x_0 \in \B_1$ there is a unique solution $x_t$ of 
the Cauchy problem corresponding to \eqref{odelat} defined for 
$|t| < T$. 
We denote by $S_t(x_0) = x_t$. Note that, by the 
uniqueness result, we have $S_{t+s} = S_t \circ S_{s}$ when 
all the maps  are defined and the composition makes sense. Moreover

\begin{enumerate}
\item  For all $t\in(-T, T),\ S_t :\B_1\to \B$ is a diffeomorphism
onto its image.
\item If $X\in C^r_\Gamma(\B)$ then $S_t \in C^r_\Gamma(\B_1)$ for
all $t\in(-T, T)$. Moreover, there exist $C, \mu > 0$ such that
$$\| D S_t(x)\|_\Gamma \leq C e^{\mu t},\qquad x\in\B_1,\ \
t\in(-T, T).$$
\end{enumerate}

Note also that, when $\B = \M$ and $DX$ is bounded, we have $T = \infty$.
\end{pro}

\noindent {\bf Remark.} When $M$ is the complexified manifold the
derivatives can  be considered as complex derivatives. Therefore
in such  a case $S_t$ is analytic according to Definition~\ref{analyt}.

Moreover if $0$ is an equilibrium point of $X, S_t(0)=0$ and hence
if $\psi \in \A_{\rho,\underline{c},\Gamma}$, by Lemma \ref{compo2},
$S_t \circ \psi \in \A_{\rho,\underline{c},\Gamma}$.

\begin{proof}
The first claim (1) follows from the standard proof of existence, uniqueness and
regularity of solutions of ordinary differential equations in Banach
spaces \cite{Hale80}.
Let $m=\| X\|_{C^1}$ and $T<\eta/m$. Then we have
that $S_t(x)$ is $C^r$ with respect to $(t,x)\in (- T, T)\times
\B_1$. The uniqueness implies the flow property
$$S_{t+s}(x)= S_t(S_s(x))$$
and this property implies that $S^{-1}_t = S_{-t}$.

We note that the standard theory of existence of solutions, also 
gives that, when $\B = \M$, we have $T = \infty$ (recall 
that our definition of $C^r$ implies that the derivatives are 
bounded, so that $X$ is globally Lipschitz). 

(2) By the general theory we also have that $DS_t$ satisfies the
first order variational equation
$$(DS_t(x))' = DX(S_t(x)) DS_t(x),\qquad DS_0(x)=\Id,$$
and by the theory of linear systems we know that $DS_t(x)$ is the
limit of the sequence given by
\begin{equation}\label{phio}
\Phi^0_t(x) =\Id,
\end{equation}
\begin{equation}\label{phik}
\Phi^k_t(x) =\Id  + \int^t_0 DX(S_s(x))\Phi^{k-1}_s(x)\, ds,\qquad
k\ge 1
\end{equation}
for $t\in(-T, T), x\in \B_1$. To get that $DS_t(x)\in \L_\Gamma$ we
make estimates of the $\Gamma$-norm of $\Phi^k_t$.

Using \eqref{phio} and \eqref{phik} we obtain by induction that
\begin{equation*}
\begin{split}
(1,k)&\ \Phi^k_t(x)\in \L_\Gamma,\quad\forall x\in \B_1,\ \
\forall
t\in(-T, T),\\
(2,k)&\ \sup_{x\in\B_1} \big\| \Phi^{k+1}_t (x) -
\Phi^k_t(x)\big\|_\Gamma \le \frac{1}{(k+1)!}\big(\|
X\|_{C^1_\Gamma} \vert t\vert\big)^{k+1},  \forall t \in (-T, T),
\end{split}
\end{equation*}
for $k\ge 0$.

In particular, we note that if $\Phi_t^k(x)$ is a linear operator 
represented by its matrix, then, so is $\Phi_t^{k+1}(x)$. 

Writing
$$\Phi^k_t(x) = \Id + \sum^k_{j=1}\big(\Phi^j_t(x) - \Phi^{j-1}_t(x)\big),$$
by $(1,k), (2,k)$ we easily obtain that $\Phi^k_t(x)$ converges in
$\L_\Gamma$. Hence $DS_t(x)\in \L_\Gamma$ and
$$\| DS_t(x)\|_\Gamma \le \Gamma^{-1}(0) + \sum^\infty_{j=1}
 \frac{1}{j!}\big(\| X\|_{C^1_\Gamma}\vert t\vert\big)^j = \Gamma^{-1}(0) - 1 +
  \exp\big(\| X\|_{C^1_\Gamma}\vert t\vert\big).$$

The higher order derivatives of $S_t$ satisfy higher order variational
equations which are  also linear equations. By an analogous
argument we get that 
$D^kS_t(x) \in \L_\Gamma(\M, \L^{k-1}(\M,\M))$.
We present  the details for the case $k=2$
and leave to the reader the adaptation of the typography for 
larger $k$. We have:
\begin{equation*}
\begin{split}
(D^2S_t(x))^\prime &= DX(S_t(x)) D^2 S_t(x)\\
&\quad + D^2X(S_t(x))(DS_t(x),\
DS _t(x)), \qquad D^2S_0(x)=0.
\end{split}
\end{equation*}

Let
$$G_t(x) = \int^t_0 D^2X(S_s(x)) (DS_s(x),\, DS_s(x))\, ds.$$
By Lemma \ref{lemAB}, $G_t (x)\in \L^2_\Gamma$. We can write
$$D^2 S_t(x) = \int^t_0 DX(S_s(x)) D^2S_s(x)\, ds + G_t(x) .$$

The sequence given by
\begin{equation*}
\begin{split}
\Psi^0_t (x) &= 0\\
\Psi^k_t (x) &= \int^t_0 DX(S_s(x))  \Psi^{k-1}_s (x)\, ds + G_t(x)
\end{split}
\end{equation*}
converges to $D^2S_t(x)$. Similarly to the case $k=1$ one proves
by induction that $\Psi^k_t(x) \in \L^2_\Gamma$. Since
$\L^2_\Gamma$ is complete we obtain that $D^2S_t(x)\in
\L^2_\Gamma$.
\end{proof}

We also remark that, using the standard argument of 
adding extra equations \cite{Hale80}, we can obtain the 
smooth dependence on parameters.

\subsection{Invariant tori for flows} 
 The following result  is our main KAM theorem for
vector-fields on lattices. For the sake of simplicity, 
we have not formulated the most general result possible
but have rather stated the result for models that appear in the 
Physics literature.

In the following, for the sake of simplicity, we consider equations of the type 

\begin{equation}\label{glVF}
\dot u = J_\infty \nabla H(u)
\end{equation}
where 
\begin{itemize}
\item The operator $J_\infty$ is given by 
$$J_\infty (z)= {\rm diag} \big(\dots, J,\dots\big),$$

where $J$ is the standard symplectic form. 

\item The $\nabla$ operator is the standard operator induced by the $\ell^2(\Z^N)$ metric on the lattice. 

\end{itemize} 

The previous equation \eqref{glVF} arise in the context of statistical physics as described in the introduction. If one desires to consider more general vector-fields $X$, we refer the reader to the paper \cite{FontichLS09} where algorithms and proofs are provided in this more general context.

We first describe the non-degeneracy conditions. The linearized equation 
\begin{equation*}
\frac{d\Delta}{dt}=J_\infty D\nabla H(K(\th+\omega t))\Delta
\end{equation*}
plays a crucial role. We denote $A(\th )\equiv J_\infty D\nabla H(K(\th))$ and we remark that since the vector field $J_\infty \nabla H$ has decay, by Proposition \ref{decayODE}, the vector field $J_\infty \nabla H$ generates an evolution operator, denoted $U_{\th}(t)$ with decay. We have
\begin{equation*}
\frac{d}{dt}U_{\th}(t)=A(\th+ \omega t)U_{\th}(t),
\end{equation*} 
and $U_{\th}(0)=\Id$.
We now have the following definitions. 
\newtheorem{condition}[thm]{Condition}

\begin{condition}\label{ND1VFNH}(Spectral non-degeneracy condition)
Given an embedding $K: D_\rho\supset \TT^l \rightarrow \M$ 
we say that $K$ is
hyperbolic
 non-degenerate if 
there is an analytic splitting 
\begin{equation*}
T_{K(\th)} \mathcal{M}=\mathcal{E}_{K(\th)}^s\oplus
\mathcal{E}_{K(\th)}^c  \oplus \mathcal{E}_{K(\th)}^u 
\end{equation*} 
invariant under the linearized equation \eqref{varNH} in the sense that 
\begin{equation*}
U_{\th}(t)\mathcal{E}^{s,c,u}_{K(\th)}=\mathcal{E}^{s,c,u}_{K(\th+\omega t)}. 
\end{equation*}
Moreover the center subspace $\mathcal{E}^c_{K(\th)}$ has dimension $2l$.
We denote $\Pi_{K(\th)}^s$, $\Pi_{K(\th)}^c$ and
$\Pi_{K(\th)}^u$ the projections associated to this splitting and we denote
\begin{eqnarray*}
U^{s,c,u}_{\th}(t)=U_{\th}(t)|_{\mathcal{E}^{s,c,u}_{K(\th)}}.
\end{eqnarray*}

Furthermore, we assume that 
there exist $\beta_1 ,\,\beta_2 ,\, \beta_3>0$ and $C_h>0$ independent of $\th$ satisfying $\beta_3<\beta_1  $, $\beta_3<\beta_2 $ and such that the splitting is characterized by the following rate conditions:
\begin{eqnarray}
\|U^s_{\th}(t)U^s_{\th}(\tau)^{-1}\|_{\rho,\underline{c} \Gamma} &\leq &C_h e^{-\beta_1 (t-\tau)}, \qquad t \geq \tau,\nonumber \\
\|U^u_{\th}(t)U^u_{\th}(\tau)^{-1} \|_{\rho,\underline{c},\Gamma} &\leq& C_h e^{\beta_2 (t-\tau)},  \qquad  t\leq \tau ,\label{cotes-scu}\\
\|U^c_{\th}(t)U^c_{\th}(\tau)^{-1}\|_{\rho,\underline{c},\Gamma} &\leq &C_h e^{\beta_3 |t-\tau|},  \qquad t,\tau \in \mathbb{R}\nonumber . 
\end{eqnarray}
\end{condition}
\begin{condition}\label{ND2VFNH} Let 
$N(\th)=[DK(\th)^\top DK(\th)]^{-1}$
  and $P(\th)=DK(\th)N(\th)$. The
  average on $\torus^l$ of the matrix 
\begin{equation*}
S(\th)=N(\th)DK(\th)^\top [A(\th)J_\infty-J_\infty A(\th)]DK(\th)N(\th).
\end{equation*}
is non-singular. Here $A(\th)=J_\infty D\nabla H(K(\th))$. 
\end{condition}

We now state our theorem 

\begin{thm}\label{existenceHamillattice}
Let $H$ be a formal Hamiltonian function on $T ^*\mathcal \mathcal M$ such that the associated vector-field $X=J_\infty \nabla H$ is a $C^2_\Gamma$, analytic vector-field in
$\M$.  For some decay function $\Gamma$, let $\omega\in D_h (\kappa,
\nu)$ for some $\kappa > 0,\ \nu\ge l-1,\ \rho_0>0$ and $\underline
c=(c_1,\dots ,\,c_R)\in (\Z^N)^R$. Denote $S_t$ the flow associated
to $X$.

Consider the equation
\begin{equation}\label{flots}
\sum_{i=1}^l \omega_i \frac{\partial K}{\partial \th_i} (\th)=(X
\circ K)(\th).
\end{equation}
Assume
\begin{enumerate}
\item $X$ extends analytically to a complex neighborhood $\,\U$ of
$K_0 (D_{\rho_0})$:
\begin{equation*}
B_r=\big \{ z \in \M | \exists\, \th \in \T^l\,\,|\,\, 
|\Im\, \th| < \rho_0 ,  \ |z-K_0(\th)| <r \big
\},
\end{equation*}
for some $r>0$. 
 \item
There exists $K_0\in\A_{\rho_0,\underline c,\Gamma}$ such that
$K_0 \in ND_{loc}(\rho_0,\Gamma)$ (the embedding $K_0$ is
  non-degenerate) in the sense that it satisfies non-degeneracy conditions \ref{ND1VFNH} and \ref{ND2VFNH}.

 \item There exists a constant $C>0$ depending on
$l$, $\kappa$, $\nu$, $\rho_0$, $\|H\|_{C^3(B_r)_\Gamma}$,
$\|DK_0\|_{\rho_0,\underline{c},\Gamma}$, $\|N_0\|_{\rho_0}$,
$\|S_{0}\|_{\rho_0}$, $|$\avg$(S_{0})|^{-1}$ (where $S_{0}$ and
$N_0$ are as in definitions \ref{ND1VFNH}-\ref{ND2VFNH} replacing $K$ by $K_0$) and
$\|\Pi^{c,s,u}_{K_0(\th)}\|_{\rho_0,\Gamma}$ such that
$E_0=J_\infty \nabla H (K_0)-\partial_\omega K_0 $ satisfies the
following estimates
\begin{equation*}
C\kappa^4 \delta^{-4\nu} \|E_0\|_{\rho_0,\underline{c},\Gamma} <1
\end{equation*}
and
\begin{equation*}
C\kappa^2 \delta^{-2\nu} \|E_0\|_{\rho_0,\underline{c},\Gamma} <r,
\end{equation*}
where $0 < \delta <\min(1,\rho_0/12)$ is fixed.
\end{enumerate}
Then there exists an analytic embedding
$K\in\A_{\rho-6\delta,\underline c,\Gamma}$ such that $K\in
ND_{loc}(\rho-6\delta,\Gamma)$ and satisfies equation \eqref{flots}
for all $t\in \R$.
\end{thm}

\begin{proof}

We will only sketch the proof and refer the reader to \cite{FontichLS09} where the complete proofs are provided in the case of finite dimensions. In the present framework, the Banach algebra properties of our spaces make the proofs in the infinite dimensional case very similar to the the ones in the finite dimensional context. 

We consider the following linearized equation: 

\begin{equation}\label{varNH}
\frac{d\Delta}{dt}-A(\th+ \omega t)\Delta=-E(\th+\omega t). 
\end{equation}

We first project equation \eqref{varNH} on the center subspace and on the hyperbolic subspaces. On the center subspace, one has 
\begin{equation}\label{centerVFNH}
\partial_{\omega} \Delta^c(\th)-A(\th)\Delta^c(\th)=-E^c(\th).
\end{equation}
Using the following proposition (\cite{Russmann76}, \cite{Russmann76a},
\cite{Russmann75}, \cite{Llave01c}), one can prove the following reducibility property in Lemma \ref{represVFNH}. 
\begin{pro}\label{russVF}
Assume that $\omega\in D_h(\kappa,\nu)$ with 
$\kappa>0$ and $\nu \ge l-1$, i.e.
\begin{equation*}
|\omega \, \cdot \,k|^{-1} \leq \kappa |k|^{\nu},\qquad
\mbox{for all $k\in\integer^l \setminus \{ 0\}$}.
\end{equation*}
Let $h:D_\rho \supset\TT^l \rightarrow \M$ be a real analytic function with zero
average. Then, for any  $0<\delta <\rho$ there exists a unique analytic solution 
$v:D_{\rho-\delta} \supset\TT^l \rightarrow \M$ 
of the linear equation 
\begin{equation*}
\sum_{j=1}^l \omega_j \frac{\partial v}{ \partial \th_j}=h
\end{equation*} 
having zero
average.
Moreover, if $h\in \mathcal{A}_{\rho,\underline c, \Gamma}$ then $v$ satisfies the
following estimate 
\begin{equation*}
\|v\|_{\rho-\delta,\underline c, \Gamma} \leq C \kappa\delta^{-\nu} \|h\|_{\rho,\underline c, \Gamma}, \qquad 0<\delta <\rho.
\end{equation*}
The constant $C$ depends on $\nu$ and the dimension of the torus $l$ but is independent of $\underline c$.  
\end{pro}

\begin{lemma}\label{represVFNH}
Assume $\omega\in D_h(\kappa,\nu)$ with $\kappa>0$ and $\nu\ge l-1$
and $\|E\|_{\rho,\underline c, \Gamma}$ is small enough. Then there exist a matrix
$B(\th)$ and vectors $p_1$ and $p_2$ such that equation 
\begin{equation}\label{tempcVFNH}
[\partial_{\omega}\tilde{M}(\th)-A(\th)\tilde{M}(\th)]\xi(\th)+\tilde{M}(\th)\partial_{\omega}
\xi(\th)=-E^c(\th),
\end{equation}

 can be written as
\begin{align}\label{eqSDVFNH}
\Big[ \begin{pmatrix} 0_l & S(\th)\\ 0_l & 0_l
\end{pmatrix}   &+  B(\th) \Big]  \xi(\th)+\partial_{\omega}\xi(\th)\nonumber = 
p_1(\th)+p_2(\th). 
\end{align}
Moreover, the following estimates hold:
\begin{equation}\label{estimp1VFNH}
\|p_1\|_{\rho,\underline c,\Gamma} \leq C  \|E\|_{\rho,\underline c,\Gamma},
\end{equation}
where $C$ just depends on $\|J_\infty(K)\|_{\rho,\Gamma}$, $\|N\|_{\rho}$, $\|DK\|_{\rho,\underline c,\Gamma}$ and
$\|\Pi^c_{K(\th)}\|_{\rho,\Gamma}$. For   $p_2$ and $B$  we have 
\begin{equation}\label{estimp2VFNH}
\|p_2\|_{\rho-2\delta,\underline c,\Gamma} \leq C \kappa \delta^{-(\nu+1)} \|E\|^2_{\rho,\underline c,\Gamma}
 \end{equation}
 and
\begin{equation}\label{estimbVFNH}
\|B\|_{\rho-2\delta} \leq C \kappa \delta^{-(\nu+1)} \|E\|_{\rho,\underline c, \Gamma}
 \end{equation}
for
$\delta \in (0,\rho/2)$, 
where $C$ depends on $l$, $\nu$, $\|N\|_{\rho}$, 
$\|DK\|_{\rho,\underline c, \Gamma}$, $|H|_{C^3(B_r)_\Gamma}$, $|J|_{C^1(B_r)}$ and
$\|\Pi^c_{K(\th)}\|_{\rho,\Gamma}$. 
\end{lemma}

The solution of the reduced equations works in the same way as in
the case of maps. We sketch the procedure and we emphasize on the differences.   

We write $\xi=(\xi_1,\xi_2)$. Consider the equation 
\begin{equation}\label{sdApproxNH}
\begin{pmatrix} 0_l & S(\th)\\ 0_l & 0_l
\end{pmatrix}\xi(\th)+\partial_{\omega} \xi (\th)=p_1(\th),
\end{equation}
where $p_1=(p_{11},p_{12})$.
Using this decomposition of $\E^c_{K(\th)}$ we can write equation \eqref{sdApproxNH}
in the form
\begin{eqnarray*}
S(\th)\xi_2(\th)+\partial_\omega \xi_1(\th)=p_{11}(\th),\\
\partial_\omega \xi_2(\th)=p_{12}(\th).
\end{eqnarray*}
where 
$$p_{12}(\th)=DK(\th)^\top J_\infty DE(\th). $$

In order to be able to solve this small divisor equations, one has to ensure that the average on $\T^l$ of $DK(\th)^\top J_\infty DE(\th)$ is zero. This is the main difference with the finite dimensional case and we perform now the computation. We use the fact that $J_\infty$ has a special structure. Indeed, we have 
$$(J_\infty)_{ij}=J \delta_{ij} $$
and then 

$$(DK^\top J_\infty DE)_{ij}=\sum_{k \in \Z^N} (DK^\top)_{ik}(J_\infty DE)_{kj}. $$     
But, we have 
$$(J_\infty DE)_{kj}= J (DE)_{kj}$$
and then 
$$(DK^\top J_\infty DE)_{ij}=\sum_{k \in \Z^N}(DK^\top)_{ik} J (DE)_{kj}.$$

Thereofore, the average on $\T^l$ of $DK^\top J_\infty DE^c$ amounts to compute the average on $\T^l$ of $(DK^\top)_{ik} J (DE)_{kj}$. 

\begin{remark}
Here we have use the fact 
$$E^c=\tilde M E +\hat e E$$
where $\hat e= \pi^c_{K(\th+\omega)}-\pi^\Gamma_{K(\th+\omega)}$ and the term $\hat e E$ being quadratic in the error, one can omit it. 
\end{remark} 

By the computations in \cite{LGJV05}, one proves that then the average of $(DK^\top)_{ik} J (DE)_{kj}$ is zero. Hence this gives the desired result.

We now project the linearized equation \eqref{varNH} on the stable and unstable subspaces
by using the projections $\Pi_{K(\th)}^s$ and $\Pi_{K(\th)}^u$ respectively.  
We denote 
$\Delta^s(\th)=\Pi_{K(\th)}^s \Delta(\th)$, $\Delta^u(\th)=\Pi_{K(\th)}^u \Delta(\th)$.

Using the previous notation, we obtain
\begin{equation}\label{VFstNH}
\partial_{\omega} \Delta^s(\th) -A(\th) \Delta^s(\th) 
= -\Pi_{K(\th)}^sE(\th)  
\end{equation}  
for the stable part and
\begin{equation}\label{VFinstNH}
\partial_{\omega} \Delta^u(\th) -A(\th) \Delta^u(\th) 
= -\Pi_{K(\th)}^uE(\th)   
\end{equation}  
for the unstable one. 

The following result provides the solution of the
previous equations.  
\begin{pro}
Given  $\rho>0$, equations \eqref{VFstNH} and 
\eqref{VFinstNH}  admit      unique analytic solutions
$\Delta^s:D_{\rho} \rightarrow \mathcal{E}^s$ and
$\Delta^u:D_{\rho} \rightarrow
\mathcal{E}^u $ respectively,
such that $\Delta^{s,u}(\th) \in  \mathcal{E}^{s,u}_{K(\th)}$. 
Furthermore there exist constants $C^{s,u}$
such that 
\begin{equation}\label{estimHyperbVFNH}
\|\Delta^{s,u}\|_{\rho,\underline c, \Gamma} \leq C^{s,u} \|E\|_{\rho,\underline c, \Gamma},
\end{equation} 
where  $C^{s,u}$ depend on $\beta_1$,  
$\|\Pi^s_{K(\th)}\|_{\rho,\Gamma}$ (resp.  $\beta_2$, $\|\Pi^u_{K(\th)}\|_{\rho, \Gamma}$)
and $C_h$ but is independent of $\underline c$. 
\end{pro}  

The  proof of Theorem \ref{existenceHamillattice} processes then as in the finite dimensional case. 

\end{proof}

\section{Proof of Theorem~\ref{thgl}}
\label{sec:thmthgl}

The goal of this section is to prove 
Theorem~\ref{thgl}.  We proceed in three stages: 
\begin{enumerate}
\item In the first stage, we construct 
quasiperiodic breathers around one site
indexed by a frequency $\omega \in \Xi(\ep^*)$. 
This will be a straightforward application of 
Theorem~\ref{existenceembedloc}. See Section~\ref{sec:1breathers}.
 We will use as initial 
approximation the solutions in which one site is 
oscillating quasi-periodically and the others are at the fixed point. 
This is an exact solution when $\ep = 0$ and will be an 
approximate solution when $\ep$ is sufficiently small. 
Note that, since the system is translation invariant, the center site
can be chosen to be any point on the lattice. 

\item In a second stage, carried out  in Section~\ref{sec:coupling}
we show that, given two
solutions which are centered around two groups of sites, if 
we displace far enough these solutions  and add them, we obtain an 
approximate solution (for a slightly slower decay function).
Then we can conclude 
to the existence of a true solution close to them. The estimates of solutions displaced 
will be the content of the \emph{coupling lemma} (Lemma~\ref{lem:coupling}), which is the centerpiece of 
the argument. 
This second stage of coupling different
solutions  requires several new techniques. In particular, 
a detailed discussion of Diophantine vectors in infinite dimensions. 
It will also be crucial that many of the estimates that we have 
obtained before are uniform in the number and the geometry of 
the sites. 

\item Finally, in a third stage, we will show that there is a 
limit to this process of clustering breathers. We obtain a well defined limit if 
the centers are placed far enough apart. 

\end{enumerate}
We will need the following definition. 
\begin{defi}
Given $m \in \ZZ^N$, let  $\tau^m :\M\to \M$
be defined by
\begin{equation*}
\big(\tau^m (x)\big)_i= x_{i+m}, \qquad\,\,\,i \in \ZZ^N\,.
\end{equation*}
In particular if $F:\M \to \M$\and
 $k: \torus^{p} \to \M$
\begin{align*}
(\tau^m F)_i (x) &= F_{i+m}(x),\\
(\tau^m k)_i(\th)&= k_{i+m}(\th).
\end{align*}
\end{defi}

Let $S_t$ be the flow of the system associated to the Hamiltonian in
the statement of Theorem \ref{thgl}, and let $\tilde S_t = \tau^m
S_t\ \tau^{-m}$, with $m\in \ZZ^N$. Both $S_t$ and $\tilde
S_t$ satisfy the same initial value problem, hence they coincide
wherever they are defined. As a consequence we have,
using $F = S_1$, 
\[
F= \tau^m\circ F\circ \tau^{-m}\,.
\]

{From} this we deduce that if $K_\omega : \torus^{rl}\to
\M$ with $\omega \in \mathbb{R}^{rl}$ is a solution of
$F\circ K_\omega = K_\omega \circ T_\omega$ then for all $m\in
\ZZ^N $ we have that $\tau^m K_\omega$ is also a solution.

\subsection{Existence of quasi-periodic breathers centered around one site (Part A of 
Theorem~\ref{thgl}) } 
\label{sec:1breathers}

Since the problem is invariant under translations, we 
will choose, without loss of generality, to center the 
breather at the origin. We then consider the Hamiltonian 

\begin{equation*}
H_\ep (q,p)=\sum_{n \in \ZZ^N}\Big( \frac{1}{2}\,
p^2_n+W(q_n)\Big)+\ep \sum_{j\in \Z^N}\ \sum_{n \in \ZZ^N}
V_j(q_n-q_{n+j}).
\end{equation*}

We note that, by Proposition \ref{decayODE}  for 
$\ep$ small enough, we can obtain a time-$1$ map,
which we will denote by $F_\ep$. This map will be exact 
symplectic by Proposition~\ref{prop:symplectic} in Appendix \ref{sec:symplectic}. 

We also note that, for $\ep = 0$, $F_0$ is an uncoupled map
$$(F_0(x) )_i  = f_0(x_i),\,\,\,\, i\in \Z^N$$
 with $f_0$ the time-$1$ map
of the the flow on $M$ corresponding to the Hamiltonian 
$\frac{1}{2}p^2 + W(q)$. 

Assumption {\bf H2} of Theorem~\ref{thgl}, implies 
that, for $\omega \in \Xi_0$ we can find an embedding 
$k_\omega: \torus^l \rightarrow M$ such that 
\[
f_0 \circ k_\omega = k_\omega \circ T_\omega .
\]
We can then consider the embedding $K:\torus^l \rightarrow \M$ 
defined by:
\[
\big(K_\omega(\th)\big)_i = 
\begin{cases} 
k_\omega(\th) &\qquad {i = 0} \\
0 &\qquad {i \ne 0} .
\end{cases}
\]

Note that $F_0 \circ K_\omega = K_\omega \circ T_\omega$. 
What we want to do is to check that $K_\omega$ satisfies the 
hypothesis of Theorem~\ref{existenceembedloc}. 

We start by embedding $F_\ep$ into a family $F_{\ep,\lambda}$, 
$\lambda \in \real^l$, 
constructed by setting 
\[
(F_{\ep,\lambda}(z) \big)_i = 
\begin{cases} 
(F_{\ep}(z) \big)_i + (0,\lambda)  &\qquad i = 0 \\
(F_{\ep}(z) \big)_i  &\qquad i \ne 0 \\
\end{cases} 
\]

We can think of $F_{\ep,\lambda}$ as the composition of 
the map $F_\ep$ and a translation in the direction of 
the action in the $i = 0$ component. Both maps are symplectic
but the translation is not exact symplectic. 

To verify the quality of the embedding, 
we note that
\[
\sum_{n \in \ZZ^N} 
\frac{\partial (K_\omega)_n}{\partial \th_i} 
\frac{\partial (K_\omega)_n}{\partial \th_j} 
= 
\frac{\partial (k_\omega)}{\partial \th_i} 
\frac{\partial (k_\omega)}{\partial \th_j} 
\]
and, by assumption the later is non-degenerate uniformly 
in $\omega$.

We note that for the uncoupled map,  the twist condition and the 
parameter nondegeneracy conditions in Definition
\ref{NDloc} reduce to the conditions for the time-one map. 
Similarly, the hyperbolicity conditions (Definition~\ref{NDloc-hyp}) 
are satisfied whenever $\ep = 0$. The stability results 
developed before show that these conditions remain true
(with uniform values) for $|\ep| \ll 1$.

If we choose $\Xi_0(\ep^*)$ 
with uniform Diophantine constants, and chose $\ep*$
accordingly, we obtain from Theorem~\ref{existenceembedloc} 
the existence of the KAM tori. The uniformity of the 
hyperbolicity and the non-degeneracy constants is 
a consequence of the perturbation results for the non-degeneracy conditions.

\subsection{Number-theoretic properties of infinite sequences of
  frequencies }
\label{sec:diophantineinfinite}
This section is devoted to some results on infinite sequences of
frequencies. We want to introduce the concept of 
Diophantine sequence (see Definition~\ref{diophInfty}) 
and show that these sequences are very abundant in the sense that they 
have full probability with respect to several probability measures. 

Let us  consider $\Xi_0 \subset [-L, L]^l$ with $l \in
\nat$ and $L>0$. Assume that $\Xi_0$ has positive Lebesgue
measure. Later we will take as $\Xi_0$ to be a subset of 
$\D(\kappa_0, \nu_0) $ such that there are KAM tori 
in the uncoupled system with this Diophantine properties 
(See Assumption {\bf H2} in Theorem~\ref{thgl}.)

To discuss infinite products of measures,  we consider the normalized 
probability measure
\begin{equation*}
\meas_*(\cdot) =\frac{\meas(\cdot)}{\meas(\Xi_0)}.
\end{equation*}
where $\meas(\cdot)$ is any measure absolutely continuous with respect to the Lebesgue measure on $\RR^l$. 
By a theorem of Kolmogorov \cite{Durrett96}, the
product set $\Xi_0^{\nat}$ with the product
$\sigma$-algebra can be endowed with the product probability measure
$\meas{}_*^{\nat}$.

Note that, there are different sets $\Xi_0$ which satisfy the 
assumption {\bf H2} of Theorem~\ref{thgl}. Each of these choices 
will lead to mutually singular measures in the infinite product. 
Nevertheless, we do not include $\Xi_0$ in the notation for 
the infinite measure. The result will, of course, be valid 
for all choices.

Now we introduce a notion of Diophantine sequences in 
$(\R^l)^\NN$ which is well adapted for our needs. 
Basically, we just require that for every
 $r \geq 1$ the first $r$ components are Diophantine,
even if the exponent and the constant change with $r$.  The Diophantine
properties of the sequence are just the sequence of 
Diophantine properties of the  truncations. 
This will be natural for us since at every stage of the 
argument we will be working with just a finite number of 
frequencies.

We introduce the following notation: consider sequences
$$\underline{\omega} = (\omega_1, \omega_2,\dots) \in \Xi^{\mathbb N}_0\quad {\rm and}\quad
\underline{k} = (k_1, k_2,\dots ) \in (\Z^l)^{\mathbb
N}\,\backslash\,\{0\}$$ 
and denote $\underline{\omega}^{( r)} =
(\omega_1,\dots,\omega_r)$
 and $\underline {k}^{( r)} = (k_1,\dots, k_r)$ the truncated
 sequences of length $r$. Hence
 $$\underline {\omega}^{( r)} \cdot \underline{k}^{( r)} =
  \sum^r_{i=1} \omega_i \cdot k_i\quad{\rm and}\quad
  \big\vert k^{( r)}\big\vert = \sum^r_ {i=1} \vert k_i\vert\,,$$
  where $k_i = (k_{i,1},\dots, k_{i,l}) \in \Z^l$ and $\vert k_i\vert 
  = \vert k_{i,1}\vert+ \dots +\vert k_{i,l}\vert .$
 Also, given $\omega_1 \in \R^{r l}$ and $\omega_2 \in \R^l$ we
will write $\omega_{12}(\omega_1, \omega_2) \in \R^{(r+ 1) l}$ the
concatenation of the vectors $\omega_1, \omega_2$.
\begin{defi}\label{diophInfty}
We define
\begin{equation*}
\mathcal{D}=\bigcup_{(\underline{\kappa},\underline{\nu}) \in
  (\mathbb{R}^+)^{\nat} \times (\mathbb{R}^+)^{\nat}}
  \mathcal{D}(\underline{\kappa},\underline{\nu}) ,
\end{equation*}
where $\underline \kappa =(\kappa_1,...,\kappa_r,...)$, $\underline \nu =(\nu_1,...,\nu_r,...)$ and 
\begin{equation*}
\mathcal{D}(\underline{\kappa},\underline{\nu})=\left \{
\begin{array}{cc}
\underline{\omega} \in
  \Xi_0^{\nat}\,|\, \forall
  r \geq 1, \quad 
\big| \sum_{i = 1}^r \omega_i \cdot k_i -m\big|^{-1} \leq \kappa _r
  |k^{( r)}|^{ \nu_r}, \\ \forall \underline{k} \in
  (\ZZ^N)^\nat\ {\rm s.t.}\ k^{( r)} \not= 0, \; \forall m\in \Z
\end{array} \right \}.
\end{equation*}
\end{defi}
\begin{remark}
Note that, since we are considering infinite dimensions, the notion
of $|k|$ we are using in  Definition \ref{diophInfty} could matter.
We note however that changing the norms only changes the sequence
$\underline{\kappa}$, the sequence $\underline{\nu}$ remaining the
same.
\end{remark}

The next result ensures that there are many sequences of Diophantine
vectors.
\begin{lemma}
Let $\underline{\nu}$ be a given sequence such that $\nu_r > r l $. Then,
\[
\meas^{\nat}_* \Big(\Xi_0^{\nat}\setminus \bigcup_{\underline{\kappa}
\in (\RR^+)^\nat
}\mathcal{D}(\underline{\kappa},\underline{\nu})\Big)=0.
\]
\end{lemma}

\begin{proof}
We follow the standard argument for the finite dimensional case (see
\cite{Llave01c} for a pedagogical exposition).

Notice first that  $L$ 
-- the size of the box in $\RR^l$ containing our 
set $\Xi_0$ -- 
is fixed. We start by considering $r$ fixed. 

 For $k \in \ZZ^{rl}$,
$\kappa_r \in \RR^+$, $\nu_r \in \RR^+$, we define
\[
B_{k,m, \kappa_r, \nu_r} = \big\{ \omega \in ([-L, L]^l)^\nat \,|\,
|\, \omega^{( r)} \cdot k -m| < \kappa_r^{-1} | k|^{-\nu_r} \big\}\,.
\]

We note that
\begin{equation}\label{estimDioph}
\Xi^\NN_0 \setminus \mathcal{D}(\underline{\kappa},\underline{\nu}) =
\bigcup_{r \geq 1}\ \bigcup_{k \in \ZZ^{rl} \setminus \{0\},m} B_{k,m, \kappa_r,
\nu_r}\,.
\end{equation}

Geometrically, the sets $B_{k,m, \kappa_r, \nu_r} $ are slabs of width $2 \kappa_r^{-1}
|k|^{-\nu_r - 1}$. As a consequence, we obtain that
\[
\meas^\NN( B_{k,m, \kappa_r, \nu_r} \cap \Xi_0^\nat )\le
\meas^r( B_{k,m, \kappa_r, \nu_r} \cap \Xi_0^r )
\leq C_r \kappa_r^{-1} |k|^{-\nu_r - 1}\,.
\]
Moreover, given $k$, the number of sets $B_{k,m, \kappa_r, \nu_r}$
intersecting $\Xi^r_0$ is bounded by a 
constant depending on the dimension 
times $|k|$.
Hence, we have that for $\nu_r  > r l$,
\begin{align*}
\meas^\NN(\cup_{k \in \ZZ^{rl} \setminus \{0\},m}
 B_{k, m,\kappa_r, \nu_r})
 &\leq \sum_{k \in \ZZ^{rl} \setminus \{0\}}
\meas^\NN( B_{k, m,\kappa_r, \nu_r})C_{rl}|k| \\
&\leq \meas(\Xi_0)^{-r} 2rl\ \kappa^{-1}_r \sum_{s=1}^\infty 
C'_{rl} \frac{s^{rl-1 }}{s^{\nu_r }} \leq  C''_{rl,\nu_r}\
\kappa^{-1}_r,
\end{align*}
where $C_{r l}$, $C'_{r l}$ and $C''_{r l,\nu_r}$ are explicit constants. The right-hand side of the previous expression can be estimated from
above by $\sum_{r\geq 1} C_{rl,\nu_r} \kappa_r^{-1} $.

By choosing a suitable sequence $\underline\kappa$, the sum can be
made as small as desired.
\end{proof}

\subsection{Constructing more complicated breathers
out of simpler ones. The coupling lemma} 
\label{sec:coupling} 
The main goal of this section is to prove 
Lemma~\ref{lem:coupling} that shows 
that if we have two solutions of 
the invariant equation and put them in places separated sufficiently
far apart, when we add them, we obtain a very 
approximate solution of the invariance equations. 

In Lemma~\ref{lem:nondeg} we will show that
these solutions obtained superimposing the 
two non-degenerate (in the sense of 
Definition~\ref{NDloc}) solutions centered around very far apart 
centers also satisfy the same non-degeneracy assumptions
with only slight worse constants.

We will also show in Lemma~\ref{lem:nondeg-hyperbolic} 
that if the approximate solutions are (up to a  
bounded error $\eta$) superpositions of centered breathers, 
then, they satisfy the hyperbolicity conditions
of Theorem~\ref{existencetranslatedloc} with uniform bounds. 
The crucial point of Lemma~\ref{lem:nondeg-hyperbolic} is that the 
estimates on the  $\eta$ allowed and the non-degeneracy constants 
are independent of $\underline{c}$, the finite set of sites  that we 
are considering. This will be a relatively easy consequence of 
all the uniformity properties that we have developed so far.

\begin{figure} 
\begin{center} 
\includegraphics[ height = 1.5 in]{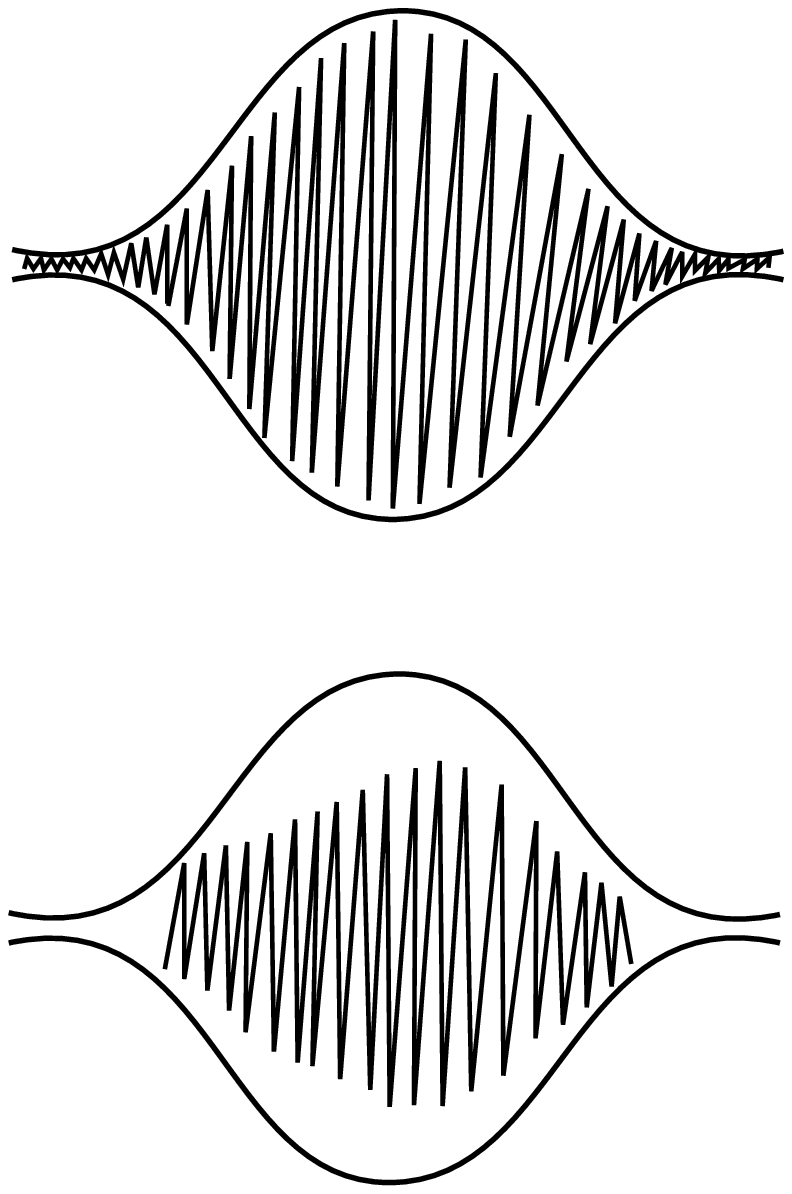} 
\end{center} 
\begin{center} 
\includegraphics[ height = 1.5 in]{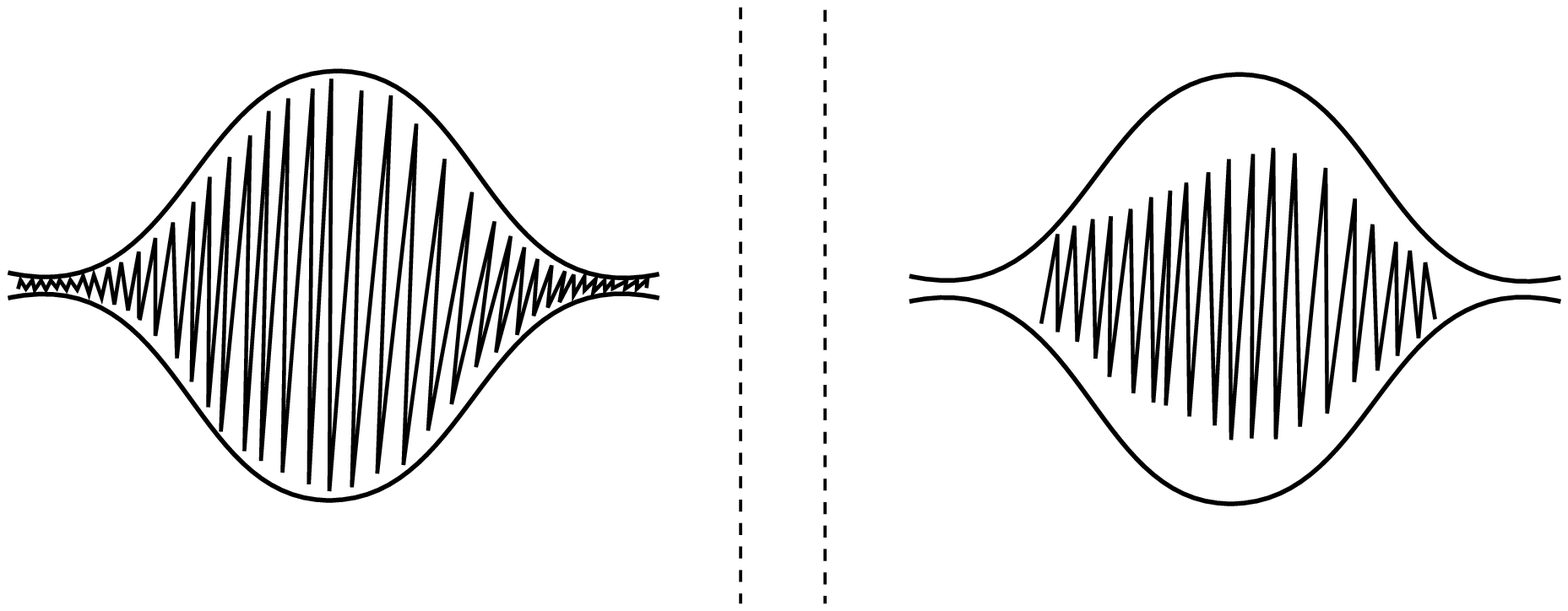} 
\end{center} 
\begin{center} 
\includegraphics[ height = 1.5 in]{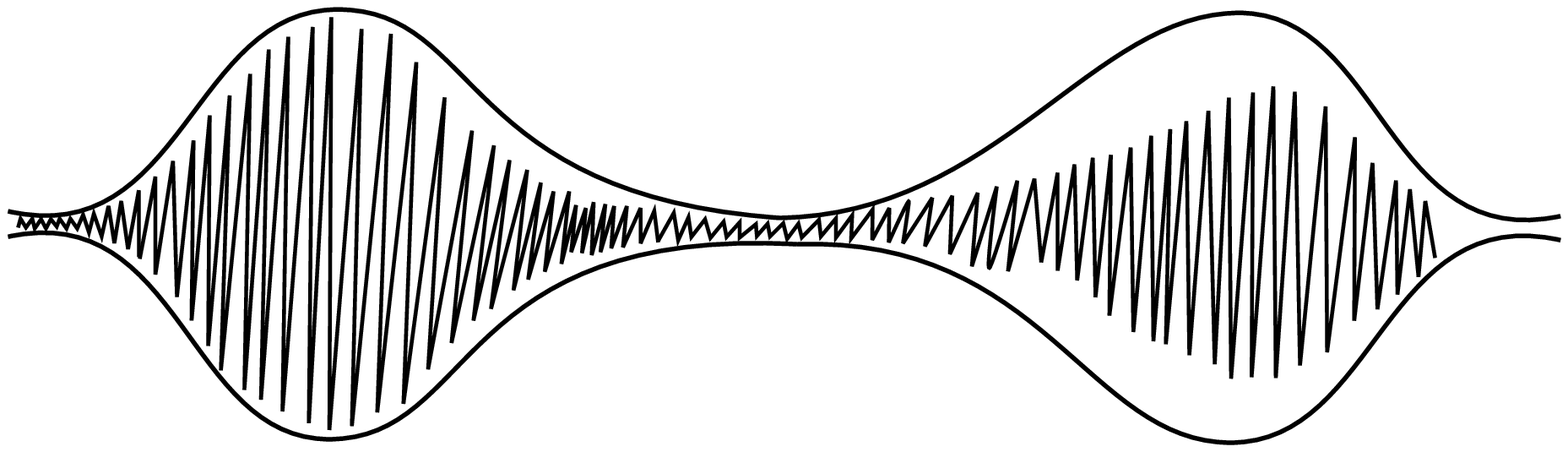} 
\end{center} 
\caption{Given two  breathers, placing them far apart, we obtain 
an approximate solution. Using the a-posteriori
Theorem~\ref{existenceembedloc}, we obtain that there is
a true solution close to it. See Lemma~\ref{lem:coupling} }
\end{figure}

\subsubsection{Some elementary calculations with the 
decay functions in Proposition~\ref{exDecay}}

Theorem~\ref{thgl} is formulated with 
the special scale of
decay functions $\Gamma_\beta$ defined by
$$\Gamma_\beta (i)=\begin{cases}
a\ \vert i \vert^{-\alpha}\ e^{-\beta\vert i\vert} &\qquad {\rm
if}\qquad i\not= 0,\\
a&\qquad{\rm if} \qquad i = 0\,,\end{cases}
$$
with $\alpha=\alpha_0 >N$ fixed  and $0 < \beta \leq \beta_0$.

In the proof of Proposition \ref{exDecay} in \cite{JL00} it is shown
that the value of $a$ can be chosen as any value less than some
$a_0$ independent of $\beta$. Actually we have 
$$a_0(\alpha) < \big(2^{\alpha+1} K_{N,\alpha} + 2\big)^{-1},\qquad
 {\rm with}\quad K_{N,\alpha}= \sum_{j\in \ZZ^N\backslash\{0\}} \vert j\vert^{-\alpha}.$$
Throughout this section, we set $\Gamma=\Gamma_{2\beta_0}$
In the definition of both $\Gamma$ and $\Gamma_\beta$  we take the
 value of $a = \min \big(a_0(\alpha), a_0(2\alpha)\big).$

 With this choice we have the following properties:
 \begin{enumerate}
 \item if $\tilde\beta < \beta$ then $\Gamma_\beta(i) \le
 \Gamma_{\tilde\beta}(i)$ for all $i\in \ZZ^N$.
  \item if $\tilde\beta < \beta$ then  
  $\displaystyle{\lim_{\vert m\vert\to\infty}} \frac{\Gamma_\beta
 (m)}{\Gamma_{\tilde\beta} (m)} =0.$
 \item for any $\beta, \tilde\beta \leq \beta_0$
 \end{enumerate}
\begin{equation}
\Gamma(i) \leq \frac 1 a\ \Gamma_\beta (i)\ \Gamma_{\tilde\beta}(i),
\qquad i \in \ZZ^N.\label{propgammabeta}
 \end{equation}
 We will encounter the quantity $\displaystyle{\sum^r_{k=1}} \Gamma_\beta (i-c_k)$.
 To be able to estimate it in a convenient way, independently on
 $r$, we will work with sequences of sites $\underline{c} = (c_1, c_2,\dots) \in (\ZZ^N)^\NN$
 satisfying the property
\begin{defi} 
\label{propertyP} 
We say that a sequence of sites $\underline{c}$ is spatially non-resonant when for all  $i\in \ZZ^N$ there exist at most two different sites 
$c_p, c_q$ in the sequence
such that $\vert c_{p}-i\vert = \vert c_q - i\vert$.
\end{defi}

\begin{remark} 
If we arrange the sites $c_k$ in a
coordinate plane of $\ZZ^N$, for instance $\ZZ^2 \times
\{0\}^{N-2}$, and for all $k$ we have $\vert (c_{k+1}- c_k)_2\vert <
\vert(c_{k+1}-c_k)_1\vert$, then $\underline{c}$
is spatially non resonant according to Definition~\ref{propertyP}. 
\end{remark}

\begin{lemma}
Let $\underline{c}$ be a spatially non-resonant sequence. 

Let $\beta \in (0,\beta_0)$. Then for every $i\in \ZZ^N$ and $r\geq 2$ we have
\begin{equation}
\sum^r_{k=1} \Gamma_\beta (i-c_k) < \frac{2}{1-e^{-\beta}}\
\max_{k}\ \Gamma_\beta (i-c_k).\label{sumagamma}\end{equation}
\end{lemma}

\begin{proof}
Let $i\in \ZZ^N$ be fixed. Let $k_0$ be such that $\vert i -
c_{k_0}\vert = \min_k \vert i -c_k\vert$. By the spatially non resonant property in the sum
\eqref{sumagamma} for any value $\Gamma_\beta(i-c_k)$ there are at
most two terms taking the same value. Then we can group the terms in
pairs. Moreover if $\vert i -c_p\vert > \vert i -c_q\vert$ then
$$\Gamma_\beta(i - c_p)< \Gamma_\beta(i - c_q)\ e^{-\beta (\vert i - c_p\vert - \vert i - c_q\vert)}.$$
Therefore
$$\sum^r_{k=1} \Gamma_\beta (i - c_k) < 2 \Gamma_\beta (i - c_{k_0}) + 2 \sum^\infty_{m=1}
\Gamma_\beta(i-c_{k_0}) e^{-\beta m}$$ and \eqref{sumagamma}
follows.
\end{proof}

\subsubsection{Statement and proof of the \emph{Coupling Lemma}}

\begin{lemma}({\bf Coupling lemma})
\label{lem:coupling}
Let
$K_{\underline{\omega}_1} \in
 \mathcal{A}_{\rho,\underline{c}_1,\Gamma_\beta} 
\cap ND_{\rm loc} (\rho, \Gamma_\beta)$, 
$K_{\omega_2}
\in \mathcal{A}_{\rho,{c}_2,\Gamma}  \cap ND_{\rm loc} (\rho, \Gamma)$, $\beta <\beta_0$, 
be the
  parameterizations of two invariant tori for $F$, localized around $\underline{c}_1$ and
${c}_2$ respectively, vibrating with frequencies $\underline{\omega}_1 \in
\R^{rl}$ and $\omega_2\in \R^l$ respectively.

Then, if $\vert m\vert$ is large enough, $K_{\omega_{12}}:
\T^{(r+1)l} \to \mathcal M$ defined by
$$K_{\omega_{12}} = K_{\underline{\omega}_1} + \tau^m K_{\omega_2}$$
is an approximate solution of
$$F \circ K = K \circ T_{\omega_{12}}$$
in the following sense: given $0 < \tilde{\beta}< \beta$ we have the
estimate
\begin{align}
\| F\circ K_{\omega_{12}} &- K_{\omega_{12}} \circ
T_{\omega_{12}}\|_{\rho, \underline{c}_{12}, \Gamma_{\tilde
\beta}} \nonumber\\
&\leq \max \big(\| K_{\omega_1}\|_{\rho, \underline{c}_1,
\Gamma_\beta},\, \| K_{\omega_2}\|_{\rho, c_2, \Gamma}\big),
\Phi(m) \label{estimcrucial}
\end{align}
where $\underline{c}_{12} = (\underline{c}_1, c_2- m)$, and
 $\Phi $   depends  on $  F, \underline{c}_1, c_2, \beta, \tilde \beta$ 
  and  $$ \lim_{\vert m\vert \to \infty} \Phi(m)= 0.$$
\end{lemma}
\begin{remark}
Note  that the approximate torus
 $K_{\omega_{12}}$ is 
in an space of slightly slower decay 
than the space of the invariant tori $K_{\underline{\omega}_{1}}$ and 
$K_{\omega_{2}}$ since the decay estimate
\eqref{estimcrucial} involves the weight $\Gamma_{\tilde{\beta}}$
instead of the weight $\Gamma_\beta$.

It is important to emphasize that if 
we choose two solutions, fix a decay function slower than that of  the two 
solutions, any 
target     smallness for the error of 
the coupled solution
can be accomplished by setting the translated solution far enough. 
In other words, we can choose all the parameters of the 
Lemma~\ref{lem:coupling} and adjust all the requirements by 
setting the solutions far apart. 
\end{remark} 

\begin{remark}

Notice also that in  Lemma~\ref{lem:coupling}, the approximate torus
$K_{\omega_{12}}$ is defined on
 $\T^{rl} \times \T^l=\T^{(r+1)l}$. To
make the notations coherent we embed the tori $\T^{rl}$ and
$\T^{l}$ into $\T^{(r+1)l}$ identifying $\T^{rl}$ with $\T^{rl}
\times \{0\}$  and $\T^l$ with $\{0\} \times \T^l$ respectively.
Hence if $\th = (\th_1, \th_2) \in \T^{(r+1)l},
K_{\underline{\omega}_1}(\th) = K_{\underline{\omega}_1}(\th_1)$ and $K_{\omega_2}
(\th) = K_{\omega_2}(\th_2).$
\end{remark}

\begin{remark} 
For simplicity, we have stated Lemma~\ref{lem:coupling} as joining together
a breather around one site to an already constructed solution. 
It is possible  (and perhaps more natural) to prove a lemma that asserts 
that given two solutions (each containing oscillations around 
many sites) one can displace them and obtain a very approximate solution 
(in a slower decay space).  We leave the precise formulation and 
the proof to the reader. 
\end{remark} 

Before proving Lemma \ref{lem:coupling} we establish a lemma with two technical
estimates. Given $\underline{c}_1 = (c_{1,1},\dots, c_ {1,r}) \in (\Z^N)^r$ and
$c_{2},  m \in \Z^N$ we introduce the sets of indices
\begin{align*}
\mathcal{I}_1 &= \big\{i \in \Z^N\, \vert\, \min_k \vert i -
c_{1,k}\vert < \vert i +
m - c_2\vert\big\}\,,\\
\mathcal{I}_2 &= \Z^N \setminus
 \mathcal{I}_1\,,
\end{align*}
and the functions
\begin{align*}
B_1(\beta, \tilde\beta,m) &=\  \sup_{i \in \mathcal I_1}\
\frac{\Gamma_\beta (i + m - c_2)}{\max_{k} \Gamma_{\tilde
\beta}(i-c_{1,k})}\,,\\
 B_2(\beta, \tilde\beta,m) &=\  \sup_{i \in
\mathcal I_2}\ \frac{\max_{k}\Gamma_{\beta} (i-c_{1,k})}{
\Gamma_{\tilde\beta}(i+ m -c_2)}\,.
\end{align*}

\begin{lemma}
If\ $0 < \tilde\beta<\beta$ we have
$$\lim_{\vert m\vert\to \infty} B_1 (\beta, \tilde\beta, m)=
 \lim_{\vert m\vert\to\infty} B_2 (\beta, \tilde \beta, m) = 0\,.$$
\end{lemma}

\begin{proof}
First we note that if $i\in \mathcal I_1$ then
\begin{equation}\label{cotaI1}
\vert i + m -c_2\vert > \frac 1 2 \min_{k} \vert c_{1,k}+ m -
c_2\vert\,.
\end{equation}
Indeed, let $k_0$ be such that $\vert i - c_{1, k_0}\vert = \min_k
\vert i - c_{1,k}\vert$. Then
\begin{align*}
\vert i + m - c_2\vert &\geq \vert c_{1,k_0} + m - c_2\vert - \vert
i - c_{1, k_0}\vert\\[1mm]
&> \vert c_{1, k_0} + m - c_2\vert - \vert i + m - c_2\vert
\end{align*}
and hence
$$\vert i + m - c_2\vert > \frac 1 2\ \vert c_{1, k_0} + m - c_2 \vert \geq \frac 1 2\ \min_{k} \vert c_{1,k}
+ m - c_2\vert\,.$$ Moreover, if $i\in \mathcal I_1$, by the
monotonicity of $\Gamma_{\tilde\beta}$ we have $\max_{k}
\Gamma_{\tilde\beta}(i- c_{1,k})> \Gamma_{\tilde\beta}(i + m -
c_2)$. Now
\begin{equation*}
\frac{\Gamma_\beta (i + m - c_2)}{\max_{k}\Gamma_{\tilde\beta}(i -
c_{1,k})} < \frac{\Gamma_\beta (i+m - c_2)}{\Gamma_{\tilde\beta}(i +
m -c_2)}\,.
\end{equation*}
The bound \eqref{cotaI1} shows that when $\vert m \vert \to \infty,\
\vert i + m - c_2\vert$ goes to infinity uniformly in $i \in
\mathcal{I}_1$. Hence by the second property of the scale
$\Gamma_\beta$ we obtain the first limit. The second limit is proved
in an analogous way, checking first that if $i\in \mathcal I_2$
$$\min_{k} \vert i - c_{1,k}\vert \geq \frac{1}{2} \min_{k} \vert c_{1,k} + m - c_2\vert$$
and using that if $i\in \mathcal I_2$
\begin{equation*}
\frac{\max_{k}\Gamma_\beta (i - c_{1,k})} {\Gamma_{\tilde \beta}(i+m-c_2)}
\leq \frac{\max_{k} \Gamma_\beta  (i-
c_{1,k})}{\max_{k}\Gamma_{\tilde \beta}(i - c_{1,k})}\,.
\end{equation*}
\end{proof}
\begin{proof}
[Proof of the coupling lemma Lemma~\ref{lem:coupling}]

 We denote $E(K)= F \circ K - K \circ
T_{\omega_{12}}$ the error of the invariance equation for the coupled 
breather. 

We are going to estimate the $i$-th component of $E =
E(K_{\omega_{12}})$. Note that the torus $\tau^m K_{\omega_2}$ is
localized around the site $c_2-m.$ We distinguish two cases: either
$i \in \mathcal I_1$ or $i\in\mathcal I_2$. In the first case we write
\begin{align*}
E_i= F_i (K_{\underline{\omega}_1}) & +\ \int^1_0 \Big[DF \big(K_{\underline{\omega}_1} + s
\tau^m K_{\omega_{2}}\big)\ \tau^m K_{\omega_2}\Big]_i\ ds\\
&\quad -\ \big[K_{\underline{\omega}_1} \circ T_{\omega_{12}}\big]_i -
\big[\tau^m K_{\omega_2}\circ T_{\omega_{12}}\big]_i\,.
\end{align*}
We recall that $\th = (\th_1, \th_2) \in \T^{rl}\times
\T^l$. Since $K_{\underline{\omega}_1}$ does not depend on $\th_2$ then
$F(K_{\underline{\omega}_1}) = K_{\underline{\omega}_1} \circ T_{\omega_{12}}$. Therefore
\begin{align*}
\| E_i\|_\rho &\le \sum_j \| F\|_{C^1_\Gamma}\,
\Gamma(i-j)\, \| K_{\omega_2}\|_{\rho,c_2,\Gamma}\
\Gamma(j+m-c_2)\\
&\quad + \| K_{\omega_2}\|_{\rho,c_2,\Gamma}\ \Gamma(i + m-c_2)\\
&\leq \big(\| F\|_{C^1_\Gamma}+ 1\big)\, \|
K_{\omega_2}\|_{\rho, c_2, \Gamma}\  \Gamma(i+m-c_2)\,.
\end{align*}
Similarly, if $i\in \mathcal I_2$  we expand $F$ around $\tau^m\,
K_{\omega_2}$ and we obtain
\begin{equation*}
\| E_i\|_\rho \leq \Big(\frac{2}{1-e^{-\beta}}\, \|
F\|_{C^1_\Gamma}+ 1\Big)\, \| K_{\underline{\omega}_1}\|_{\rho,
\underline{c}_1, \Gamma_\beta}\, \max_{k}\, \Gamma_\beta(i-c_{1,
k})\,.
\end{equation*}
We take $\tilde\beta < \beta$ and we compute
\begin{align*}
\| E\|_{\rho,\underline{c}_{12}, \Gamma_{\tilde\beta}}& =\
\max \Big(\sup_{i\in\mathcal I_1}\ \min \big(\min_{k}\,
\Gamma^{-1}_{\tilde\beta} (i - c_{1,k})\,,\,
\Gamma^{-1}_{\tilde\beta}
(i+ m-c_2)\big)\, \| E_i\|_\rho\,,\\
&\qquad \sup_{i\in \mathcal I_2}\ \min\big(\min_{k}\
\Gamma^{-1}_{\tilde \beta}(i-c_{1,k}),\, \Gamma^{-1}_{\tilde
\beta}(i + m-c_2)\big)\, \| E_i\|_\rho\Big)\\
&\leq\ C\ \max \Big(\sup_{i\in \mathcal I_1}\ \min_{k}\
\Gamma^{-1}_{\tilde \beta}(i- c_{1,k})\, \Gamma(i + m-c_2),\\
&\qquad \sup_{i\in \mathcal I_2}\, \Gamma^{-1}_{\tilde \beta}
(i+m-c_2)\ \max_{k}\ \Gamma_\beta (i-c_{1,k})\Big)\\
&= C\ \max \big(B_1 (2\beta_0, \tilde\beta, m),\ B_2(\beta,
\tilde\beta, m)\big)\,,
\end{align*}
where
$$C = \Big(\frac{2}{1-e^{-\beta}}\ \| F\|_{C^1_\Gamma}+1\Big)\ \max
 \Big(\| K_{\underline{\omega}_1}\|_{\rho,\underline c_1,\Gamma_\beta}\, ,\
 \| K_{\omega_2}\|_{\rho, c_2, \Gamma}\Big)\,.$$
\end{proof}

\subsubsection{Statement and proof of 
Lemma~\ref{lem:nondeg-hyperbolic}. Verifying the 
non-degeneracy condition of the coupled 
solutions}

In this section, we verify the nondegeneracy conditions
provided that $\ep$ is small enough and that
$K$ is sufficiently close to an uncoupled 
solution with all the sites far enough apart. 
That is, we consider situations when 
we are close to the completely uncoupled solution. 

The result Lemma~\ref{lem:nondeg-hyperbolic} will be clear because  
all the uncoupled solutions for the uncoupled dynamics 
satisfy the non-degeneracy assumptions. The change of the non-degeneracy 
assumptions between this uncoupled case can be controlled by
elementary perturbation theories. Thanks to the systematic use of 
our framework, we have perturbation theories which are uniform on the 
excited sites. 

Let  $\underline{c}, \underline{\omega}$ be 
sequences of $r$ sites and frequencies. We consider 
$k_{\omega_i}$, parameterizations of invariant tori w.r.t. $f_0$, the time-one map of just one site. We denote
\[
K^* = \sum_{i = 1}^r \tau^{c_i} k_{\omega_i} 
\]
We note that $F_0 \circ K^* = K^*\circ T_{\underline{\omega}}$
and that $K^*$ is uniformly non degenerate. 

The hyperbolic splitting for $F_0\circ K^*$
is 
\begin{equation}\label{simplesplitting}
\Pi^{s,c,u} = \oplus_{i \in \ZZ^N} \Pi^{s,c,u}_i
\end{equation}
where $\Pi^{s,c,u}_i$ is the splitting corresponding 
to the $i$ torus. If $i$ is an index in $\underline{c}$, then we have 
$\Pi^c_i = \Id_{\RR^{2l}}$, $\Pi^s_i = 0$, 
$\Pi^u_i = 0$. Otherwise, one gets 
$\Pi^c_i = 0$, and $\Pi^s_i, \Pi^s_i$ are the 
projections corresponding to the stable and unstable directions 
at the fixed point. 

Notice also that the $rl \times rl$ matrix  $DK^\top DK$ 
is block diagonal. The diagonal has $r$ $l\times l$ blocks
$\{ Dk_{\omega_i}^\top Dk_{\omega_i}\}_{i = 1}^r$.

\begin{lemma} {\bf(Hyperbolicity conditions)}
\label{lem:nondeg-hyperbolic}
Assume the   hypothesis and the notation in Theorem~\ref{thgl}.
In particular, $F_\ep$ is an analytic family of exact
symplectic maps in $C^2_\Gamma (\B)$. 
Let $K\in \A_{\rho, \underline{c}, \tilde\Gamma}$
with $\tilde \Gamma < \Gamma_\beta$ for $\beta < \beta_0$.

Assume that
$\ep$, $\eta \equiv \| K - K^*\|_{\rho, \underline{c}, \tilde \Gamma} $ are
smaller than a number that is independent of $\underline{c}$
and of $\Gamma$ -- it depends only on 
$\| F\|_{C^2_\Gamma(\B_1)}$, 
$\|\partial_\ep F\|_{C^2_\Gamma (\B_1\times \{|\ep| \le \ep^*\})}$, 
$\|\partial_\ep^2 F\|_{C^2_\Gamma (\B_1\times \{|\ep| \le \ep^*\})}$
and the hyperbolicity constants of the 
uncoupled splitting.

Then, $K$ and $ F_\ep$ satisfy the non-degeneracy conditions in 
Definition~\ref{NDloc-hyp} with
uniform constants.
\end{lemma}

\begin{proof}
We  make the elementary remark
\begin{equation}\label{addsubstract} 
DF_\ep\circ K = DF_0\circ K^* + 
\big( DF_0 \circ K - DF_0\circ K^* \big)+
\big( DF_\ep \circ K - DF_0\circ K \big)
\end{equation}
and we will control the terms in parenthesis. 

By the estimates in composition in Section~\ref{sec:composition},
 we obtain that:
\[
\begin{split}
\|DF_0 \circ K - DF_0\circ K^* \|_{\rho, \underline c, \tilde \Gamma}
& \le C \| K  - K^*\|_{\rho, \underline{c}, \tilde\Gamma}\\
\| DF_\ep \circ K - DF_0\circ K\|_{\rho,\underline c,\tilde \Gamma}
&\le C |\ep|
\end{split}
\]
so we 
obtain that 
$ \| DF_\ep\circ K - DF_0\circ K^* \|_{\rho, \underline c, \tilde \Gamma}$ is 
small. 

The splitting indicated in \eqref{simplesplitting} is 
invariant for $DF_0$. Hence, it is approximately invariant 
for  $DF_\ep\circ K$ and this satisfies the conditions for
approximately invariant splittings 
Definition~\ref{NDloc2}.

We note that Proposition~\ref{NDmoveloc} ensures that, if $\ep$ and $\eta$ are small enough, there is an invariant splitting satisfying Definition 
\ref{NDloc-hyp}. 
\end{proof} 

\subsubsection{Statement and proof of 
Lemma~\ref{lem:nondeg}. Verifying the 
non-degeneracy assumptions of coupled solutions} 

\begin{lemma} {\bf(Twist conditions)}\label{lem:nondeg} 
Assume that 
$K_1, K_2$ are embeddings in 
$\A_{\rho, \underline{c}_1, \Gamma_\beta}$,
$\A_{\rho, \underline{c}_2, \Gamma_\beta}$, resp. 
both $\underline{c}_1$,
$\underline{c}_2$ being finite sequences, and $\Gamma_\beta$ as before. 
Assume that $K_1$, $K_2$  satisfy the non-degeneracy conditions 
in Definition~\ref{NDloc}. 

Then,  for $m$ sufficiently large,  
\[
\tilde K(\th_1, \th_2) = K_1(\th_1)  + \tau^m K_2(\th_2)
\]
satisfies the non-degeneracy assumptions 
in Definition~\ref{NDloc}. Furthermore, the 
non-degeneracy constants of $\tilde K$ 
can be made as close to desired to the 
constants verified both by $K_1, K_2$ 
if we choose $|m|$ large enough. 
\end{lemma}

We will be using the notation that 
$n_1$ is the number of sites in $\underline{c}_1$ 
and that $\th_1$ stands for all the $n_1\times l$ variables
corresponding to all the sites in $\underline{c}_1$. 
Similarly for $K_2$.

\begin{proof} 
We  introduce the notation that $\Phi(m)$ 
stands for any quantity (vector, matrix, function, etc. ) 
which can be made arbitrarily small by making $m$ large. 

We start by estimating the non-degeneracy condition
of the embedding. 

We see that the $l (n_1 + n_2) \times l (n_1 + n_2) $ matrix 
$D\tilde K^\top D \tilde K$ 
splits naturally into blocks depending on whether
we take derivatives with respect to 
variables in $\th_1$ or in $\th_2$: 

\begin{equation} \label{blocks} 
D \tilde K^\top D \tilde K = 
\begin{pmatrix}
D_{\th_1}K_1^\top D_{\th_1}K_1&D_{\th_1} K^\top _1 \tau^m
D_{\th_2} K_2\\
\tau^m D_{\th_2} K_2^\top D_{\th_1}K_1 &
D_{\th_2} \tau^m K^\top_2 D_{\th_2} \tau^m K_2
\end{pmatrix}\,.
\end{equation}

Since 
\[
D_{\th_2} \tau^m K^\top_2 D_{\th_2} \tau^m K_2
= D_{\th_2} K^\top_2 D_{\th_2} K_2
\]
we see that the diagonal elements 
of $D \tilde K^\top D \tilde K$ 
are precisely those of the uncoupled system and
are therefore invertible.

We will show that the non-diagonal elements in 
\eqref{blocks} can be made arbitrarily small 
by choosing $m$ large enough. 
Then, it will follow that 
$D \tilde K^\top D \tilde K$ is invertible and that 
\begin{equation} 
\tilde N = (D \tilde K^\top D\tilde  K)^{-1}
= 
\begin{pmatrix}
(D_{\th_1}K_1^\top D_{\th_1}K_1)^{-1} & 0 \\
0  &
(D_{\th_2} K^\top_2 D_{\th_2}  K_2)^{-1} 
\end{pmatrix} + \Phi(m)=
 \end{equation} 
$$\begin{pmatrix}
N_1 & 0 \\
0  &
N_2
\end{pmatrix} + \Phi(m) .$$

We estimate the off-diagonal elements of  \eqref{blocks}. We
observe that, we can estimate the  entries of $n_1 l \times n_2 l$ 
upper right block as follows:

\begin{align*}
\Big\vert \big(D_{\th_1}&K^\top _1\, \tau^m\,
D_{\th_2}K_2\big)_{p,q}\Big\vert  \leq \sum_{i} \Big\vert
\frac{\partial K_{1,i}}{\partial \th_{1,p}}\ \frac{\partial K_{2,
i+m}}{\partial\, \th_{2,q}}\Big\vert\\
&\leq\ \| D K_1\|_{\rho,\underline{c}_1, \Gamma_\beta} \|
DK_2\|_{\rho,\underline{c}_2, \Gamma}  
\sum_ {i}\, \max_{k}\ \Gamma_\beta (i - c_{1,k})\
\max_{l} \Gamma(i+m- c_{2,l})\\
&\leq \frac{2}{1-e^{-\beta}} \| D K_1\|_{\rho,\underline{c}_1, \Gamma_\beta}\
\| DK_2\|_{\rho, c_2, \Gamma} \ \max_{k,l}\ \Gamma_\beta (c_{2,l}-m-c_{1,k})\,.
\end{align*}

Since the block is finite dimensional, making the previous elements 
small enough makes its norm small. 

Estimates for the lower left block are obtained 
just noticing that it is the transposed of the upper right one.  
\medskip 

Now, we turn to estimating the twist condition. 
Again the strategy is very similar to the one used 
in checking the non-degeneracy condition. We just check 
that the matrix we need to invert is arbitrarily close
(by taking $|m|$ large enough) to a matrix 
which is block diagonal and whose blocks correspond to 
the non-degeneracy 
conditions of each of the uncoupled solutions. 

We proceed to estimate systematically all the ingredients of 
$A$ defined in \eqref{Adefined}. We have first 
\begin{align*}
\tilde P(\th) &= \big(D_{\th_1}K_1,\ \ D_{\th_2}\, [\tau^m\,
K_2]\big)\begin{pmatrix}N_1 &0\\
0&N_2\end{pmatrix}\, + \Phi(m)\\[2mm]
&= (P_1, P_2) + \Phi(m)\,.
\end{align*}
Since $N_1$ and $N_2$ are finite dimensional matrices, $P_1$ and
$P_2$ are also in $\A_{\rho, \underline{c}_1, \Gamma}$
and $\A_{\rho, \underline{c}_2, \Gamma}$, respectively.

Using that $J_\infty$ is uncoupled and constant
for the models \eqref{coupledlattice} we are 
considering now,  we can write:
\begin{equation}
\label{DFJprimer}
\begin{split} 
(DF (J^c)^{-1})&\ \big(\tilde K (\th)\big)\ P_1(\th_1)\\
&= (DF (J^c)^{-1})\ \big(K_1(\th_1)\big)\ P_1(\th_1)  \\
&\quad +\int^1_0 (D^2 F(J^c)^{-1})\,\big(K_1(\th_1)+s\tau^m
K_2(\th_2)\big)\big(\tau^mK_2(\th_2), P_1(\th_1)\big)\,ds
\end{split}
\end{equation}

\begin{equation} 
\label{DFJsegon}
\begin{split}
(DF (J^c)^{-1})\ \big(\tilde K &(\th)\big)\ P_2(\th_2) \\
& = (DF (J^c)^{-1})\ \big(\tau^m K_2(\th_2)\big)\ P_2(\th_2)
 \\
&\quad +\int^1_0 (D^2 F(J^c))\,\big(s
K_1(\th_1)+\tau^m K_2(\th_2)\big)\ \big(K_1(\th_1),\,
P_2(\th_2)\big)\,ds\,. 
\end{split} 
\end{equation}

We denote $T_{21}$ and $T_{12}$ the integral terms in
\eqref{DFJprimer} and \eqref{DFJsegon} respectively. We bound from
above the $i$-th component of $T_{21}$ by
\begin{align*}
\|(J^c)^{-1}\|\ &\sum_{j,n}\ \left|\frac{\partial^2 F_i}{\partial x_j\
\partial x_n}\right|_\rho\ \left|(K_2)_{j+m}\right|_{\rho}
\left|(P_1)_n\right|_\rho \\
&\le \|(J^c)^{-1}\|\ \|F\|_{C^2_\Gamma}\ \|K_2\|_\rho\
\|P_1\|_{\rho,\underline c_1,\Gamma_\beta}\\
&\quad \times \sum_{j,n}\ \min\big(\Gamma(i-j), \Gamma (i-n)\big)\
\max_{k,l} \Gamma(j+m-\underline{c}_{2,l})\ \Gamma_\beta(n-c_{1,k})\,.
\end{align*}
The last sum is bounded by
$$
\frac{4}{1-e^{-\beta}}\ \max_{k,l}
\Gamma_\beta (i+m-c_{2,l})\cdot \Gamma_\beta(i-c_{1,k}) \le \Phi(m). 
$$

Indeed, let $\J_1(i)= \{j,n\in \ZZ^N \mid |i-j|\le |i-n|\}$ 
and $\J_2(i) = \ZZ^N \setminus \J_1(i)$.
Using that $\Gamma (i) \le \frac{1}{a}\Gamma_\beta (i) \Gamma_\beta (i)$
the previous sum is bounded by 
\begin{align*}
&\sum_{j,n\in \J_1(i)} \Gamma(i-j)
\max_l \Gamma(j+m- c_{2,l})\ \max_k\ \Gamma_\beta(n-c_{1,k})\\
&\quad + 
\sum_{j,n\in \J_2(i)} \Gamma(i-n)
\max_l \Gamma(j+m- c_{2,l})\ \max_k\ \Gamma_\beta(n-c_{1,k})
\\
& \le \frac{2}{a} \sum_{j\in\Z^N} 
\Gamma_\beta (i-j) 
\max_l \Gamma(j+m- c_{2,l})\ 
\sum_{n\in\Z^N} 
\Gamma_\beta (i-n) 
\max_k\ \Gamma_\beta(n-c_{1,k}) .
\end{align*}
Analogously $T_{12}$ is bounded by
$$
\|(J^c)^{-1}\|\ \|F\|_{C^2_\Gamma}\ \|K_1\|_{\rho,\underline c_1,\Gamma_\beta}\
\|P_2\|_\rho\ \frac{4}{1-e^{-\beta}}\ \max_{k,l}\ \Gamma_\beta (i+m-c_{2,l}
)\
\Gamma_\beta(i-c_{1,k}).
$$ 
Note that $P^\top _1 (\th+\omega)\
T_{21} (\th)$ and $P^\top _2 (\th+\omega)\ T_{12} (\th)$ are
bounded by $C\ \displaystyle \max_{k,l}\ \Gamma_\beta (c_{1,k}
+m-c_{2,l})$, where $C$ depends on $\|(J^c)^{-1}\|,\ \|F\|_{C^2_\Gamma},\
\|K_1\|_{\rho,\underline{c}_1,\Gamma_\beta},$ 
$\|K_2\|_{\rho, \underline{c}_2,\Gamma}, \
\|P_1\|_{\rho,\underline{c}_1,\Gamma_\beta}$ and 
$\|P_2\|_{\rho, \underline{c}_2,\Gamma}$.

Also note that
\begin{align*}
\big|\big[P_1(\th+\omega)^\top \ (J^c)^{-1}(K(\th))
P_2(\th)\big]_i\big|&\le \|(J^c)^{-1}\|\ 
\|P_1\|_{\rho,\underline{c}_1,\Gamma_\beta}\ 
\|P_2\|_{\rho, \underline{c}_2,\Gamma}\\
&\quad \times \sum_{i\in\ZZ^N}  \max_{k,l}\ \Gamma_\beta
(i-c_{1,k})\ \Gamma(i+m-c_{2,l})\\
& \le C\ \max_{k.l}\ \Gamma_\beta\ (c_{2,l}-m-c_{1,k}).
\end{align*}
Now we consider the terms $P^\top _1(\th+\omega)\ (DF\tilde (J^c)^{-1})\
\big(\tilde K(\th)\big)\ P_2(\th)$ and
$P^\top _2(\th+\omega)\ (DF\tilde (J^c)^{-1})\ \big(\tilde K(\th)\big)\
P_1(\th)$. We evaluate the first one, the other being
analogous. Given $n\in \{1,\dots, r\}$,
\begin{align*}
\big|\big[P^\top _1\ &(\th+\omega)\,(DF J^c)\,\big(\tilde
K(\th)\big)\ P_2(\th)\big]_n\big|\\
&\le \sum_{i,j}\ \big|(P_1)_{i,n}\big|_\rho\ \|(J^c)^{-1}\|\
\Big\|\frac{\partial F_i}{\partial x_j}\Big\|_\rho\
|(P_2)_j|_\rho\\
&\le \|F\|_{C^1_\Gamma}\ \|(J^c)^{-1}\|\ \|P_1\|_{\rho,\underline
c,\Gamma_\beta}\ \|P_2\|_{\rho, c_2,\Gamma}\,\sum_{i,j}\,
\max_{k,l}\
\Gamma_\beta(i-c_{1,k})\ \Gamma(i-j)\ \Gamma(j+m-c_{2,l})\\
&\le C\ \max_k\ \Gamma_\beta\,(c_{1,k}+ m-c_{2,l}).
\end{align*}

With all these previous estimates we can write:
\begin{align*}
A &= \big[(P_1, P_2)\circ T _\omega+\Phi(m)\big]^\top (DF J^c)\,(\tilde
K)\ \big[(P_1,P_2)+\Phi(m)\big]- J^c P\circ T_\omega\\
&= \begin{pmatrix}\scriptstyle P^\top _1\circ
T_{\underline{\omega}_1}\big[(DF (J^c)^{-1})\,(K_1)\
P_1\big] 
-(J^c)^{-1}\,P_1\circ T_{\underline{\omega}_1}&0\\[2mm]
0&\scriptstyle P^\top _2\circ T_{\omega_2}\big[(DF J^c)\,(K_2)\
P_2\big]-(J^c)^{-1} \,P_2\circ T_{\omega_2}
\end{pmatrix}  \\
& +\Phi(m) =\begin{pmatrix}
A_1 & 0\\
0 & A_2 \\
\end{pmatrix} + \Phi(m) 
\end{align*}
which shows that it is invertible if $|m|$ is big enough
and that the norm of the inverse of $A$ can be bounded 
from above  by $\max( |A_1^{-1}|, |A_2^{-1}|) + \tilde \Phi_m$.

\end{proof}

The estimates about  the non-degeneracy with respect to 
parameters in the construction are automatic since, in the construction in 
Section~\ref{maps}, which is the one we use here, the matrix $Q$ 
is the identity, whose norm is bounded by $1$ independently of the
number of sites considered and independently of the $K$ considered.

\subsection{Adding  oscillating sites inductively}
\label{inductivesites} 
Recall that we are assuming that $\ep\le \ep^*$ 
and that we have a set $\Xi_1(\ep^*) \subset D(\nu_0, \kappa_0) 
\subset \real^l $ of 
positive measure  such that, for all $\omega \in \Xi_1(\ep^*)$, 
the system \eqref{coupledlattice} has a breather of 
frequency $\omega$ in $\A_{\rho, \{0\}, \Gamma_\beta}$. 
The non-degeneracy and hyperbolicity constants of all 
these solutions are uniformly bounded. 

The remaining part to be shown   is that 
given a sequence 
$\underline{\omega} \in \mathcal{D} \cup \Xi_1^*(\ep^*)^\infty$, 
we can find a sequence of
tori 
parameterized by 
$K_{\omega^{(n)}} 
\in \A_{\rho_n, \underline{c}^{(n)}, \Gamma_{\beta_n}}$ for a suitable sequence of centers $\underline c^{(n)}$. Here we have  that 
$$\omega^{(n)} = (\omega_1, \ldots ,\omega_n) $$ 
is the sequence of truncations  of $\underline{\omega}$
and $\rho_n, \beta_n$ are strictly decreasing sequences
so that $\rho_n \to \rho_\infty> 0 $,
$\beta_n \to \beta_\infty> 0 $. 

Our unknowns are $\rho_n, \beta_n$, the infinite sequence of 
centers $\underline{c}^{(n)}$ and the embeddings $K_{\omega^{(n)}}$ . 

The choices of $\rho_n, \beta_n$ are almost irrelevant for 
our purposes, so we choose them right away. For example we take $0<\rho_\infty <\rho_0$, $0<\beta_\infty <\beta_0$ and 
$\rho_n = \rho_\infty + 2^{-n}(\rho_0 - \rho_\infty)$, 
$\beta_n = \beta_\infty + 2^{-n}(\beta_0 - \beta_\infty)$,

So that now, our only task is to choose a sequence of
sites $\underline c^{(n)}$ (without loss of generality, 
we will assume $c_1 = 0$), such that, recursively, we 
have that taking $c_{n+1}$ far apart from the previous sites, 
$K_{\underline{\omega}^{(n)}} + \tau^{-c_{n+1}} K_{\omega_{n+1}}$
is a very approximate solution of the invariance equation 
which, furthermore, 
satisfies  uniform hyperbolicity and non-degeneracy 
conditions. Then, an application of Theorem~\ref{existenceembedloc}
will produce a true solution 
$K_{\underline{\omega}^{(n+1)}}$.

Of course, we will 
have to recover the inductive hypothesis we have made 
to construct this sequence. We will show that, we can 
ensure that
 $\| K_{\underline{\omega}^{(n)}} - K^*
\|_{\rho_n,\underline{c}^{(n)}, \Gamma_{\beta_n}}
  \le \eta/2$ where $\eta > 0 $ is the constant
introduced in Lemma~\ref{lem:nondeg-hyperbolic} and
 $K^*_n = K_{\omega_1} + \tau^{-c_2} K_{\omega_2} + \cdots 
\tau^{-c_n}K_{\omega_n}$.

After this sequence of tori with increasing number of
frequencies is produced, we will have 
to study the limit of the sequence and show that it solves the
invariance equation (this will be accomplished in 
Section~\ref{tothelimit}).  Note that, since each step
changes the number of centers,  the 
convergence of the embeddings cannot be uniform (even in
a space of decay functions). 
Nevertheless, we will show 
that there is  coordinatewise  convergence and that 
this is  enough to show that the limit satisfies
the invariance equation. 

We note that the existence of the sequence and the study of 
the limit will be accomplished because if we place the centers
very far apart from the previously placed ones, we can 
obtain that the error is small enough to beat the 
smallness requirements of Theorem~\ref{existenceembedloc}, 
to ensure that the non-degeneracy and hyperbolicity constants 
deteriorate an arbitrarily small amount and to ensure 
the passage to the limit, so that, by recursively assuming that the 
new center is far away from all the previously placed ones, we can 
ensure any smallness conditions we wish on the 
error, on the increment of the distance from the uncoupled solution 
and on the deterioration of the non-degeneracy and hyperbolicity 
constants.

We start with $K_{\omega_1}$ and $K_{\omega_2}$
localized at the node $c_1=0$ and $c_2$ respectively 
and we take $|c_2|$ big enough so
that
$$\tilde K=K_{\omega_1}+\tau^{-c_2}\ K_{\omega_2}$$
is a sufficiently approximate solution of $F\circ K-K\circ
T_{{\underline\omega}^{(2)}}=0$ and satisfies both the spectral
and the twist non-degeneracy conditions. 
Then Theorem
\ref{existenceembedloc} provides the existence of a true invariant
torus $K_{{\underline\omega}^{(2)}}\in \A_{\rho_2,\underline
c^{(2)},\Gamma_{\beta_2}}$ such that it is non-degenerate and
$$e=\|K_{\underline\omega^{(2)}}
-\tilde K\|_{\rho_2,\underline c^{(2)}, \Gamma_{\beta_2}}$$
is small. Actually it can be made as small as we want by taking
$|c_2|$ sufficiently big.
Remembering  that 
$(\omega_1, \omega_2)$ is Diophantine (and chosen from the 
start of the procedure), we see that Theorem~\ref{existenceembedloc} 
guarantees that, if we make the initial error small enough, 
we can produce a solution $K_{\underline{\omega}^{(2)}}$ 
of the invariance equation with frequency $\underline{\omega}^{(2)}$.

In the $n+1$ step of the process we assume we have the torus
$K_{{\underline\omega}^{(n)}}\in \A_{\rho_n,\underline
c^{(n)},\Gamma_{\beta_n}}$ localized around the nodes
$\underline c^{(n)}=(c_1,\dots, c_n)$, which is non-resonant, that
is
$K{_{\underline\omega^{(n)}}}\in 
ND_{\rm loc} (\rho_n,\Gamma_{\beta_n})$ 

We consider the
parameterization
$$
\tilde K(\th)-K{_{\underline\omega^{(n)}}}(\th_1)
+\tau^{m_{n+1}}\ K_{\omega_{n+1}}(\th_2),\qquad
\th=(\th_1,\th_2)\in\T^{nl}\times\T^l,
$$ as an
approximation for the new torus, which we will denote
$K{_{\underline\omega^{(n+1)}}}$, with some $m_{n+1}\in\Z^N$.

By the coupling lemma (Lemma \ref{lem:coupling}) if we take a suitable
$m_{n+1}$ big enough we obtain
$$E_{n+1}=F\circ \tilde K-\tilde K\circ T_{\underline\omega^{(n+1)}}$$
as small as we want. In particular we take
$0<\delta_{n+1}<\min\,\big(1,\rho_n/12,\
(\rho_{n+1}-\rho_n)/6\big)$
and we require
\[
C\kappa^4_{n+1}\ \delta^{-4\nu_{n+1}}_{n+1}\ \|E_{n+1}\|_{\rho_{n+1},
\underline c^{(n+1)},\Gamma_{\beta_{n+1}}}\le 1 \,,
\]
and 
\[
C\kappa^2_{n+1}\ \delta^{-2\nu_{n+1}}_{n+1}\ \|E_{n+1}\|_{\rho_{n+1},
\underline c^{(n+1)},\Gamma_{\beta_{n+1}}}\le\frac{e}{2^{n-1}}\,.
\]
We denote $c_{n+1}=-m_{n+1}$. Then
Theorem \ref{existenceembedloc} provides a true invariant torus
$K_{\underline\omega^{(n+1)}}\in\A_{\rho_{n+1},\underline
c^{(n+1)},\Gamma_{\beta_{n+1}}}$, non-degenerate and satisfying the
estimate
\begin{equation}\label{126bis}
\|K_{\underline\omega^{(n+1)}}-\tilde K\|_{\rho_{n+1},\underline
c^{(n+1)},\Gamma_{\beta_{n+1}}}\le\frac{e}{2^{n-1}}\, .
\end{equation}
\subsection{Passage to the limit (Part B of 
Theorem~\ref{thgl})}
\label{tothelimit}
The issue now is to study the limit $n\rightarrow\infty$. Thanks
to our weighted spaces and the fact that the solutions we
construct have bumps whose distance from each other tends 
to infinity fast enough, 
we can prove the following Lemma~\ref{convergenceinfinite}
which establishes that for any bounded sets in the lattice, 
the trajectories of the particles  in this set  converge uniformly.

\begin{lemma}\label{convergenceinfinite}
The sequence $\big \{ K_{\underline{\omega}^{(n)}} \big \}_{n\geq
1}$ converges component-wise and uniformly on every compact set of
$(\T^l)^\NN$. We denote $K_{\underline{\omega}}$ the limit
obtained in this sense. Furthermore, each component of
$K_{\underline{\omega}}$ is analytic from $(\T^l)^\NN$
into $M$.
\end{lemma}
\begin{remark} \label{rmk-analit}
Here, by analytic on the infinite dimensional torus $(\T^l)^\NN$,
we mean $K_{\underline{\omega}}$ writes component-wise
\begin{equation*}
(K_{\underline{\omega}})_i(\underline{\th})=\sum_{n \geq
0} (H^{(n)})_i(\th_1,\dots ,\th_n),
\end{equation*}
where $(H^{(n)})_i(\th_1,\dots ,\th_n)$
are analytic in the usual
 sense on $(\T^l)^n$ and moreover we have
\begin{equation*}
\sum_{n \geq 0} \|(H^{(n)})_i\|_{\rho_n} <
\infty ,
\end{equation*}
where $D_{\rho_n} \supset (\T^l)^n$.
\end{remark}
\begin{proof}
We represent $K_{\underline\omega}(\underline\th)$ as
\begin{align} \nonumber
&\lim_{n\to\infty}\
K_{\underline{\omega}^{(n)}}(\underline\th)\\
\label{sumaKn} &=K_{\omega_1}(\th_1)+\sum^\infty_{n=1}\
\big[ K_{\underline{\omega}^{(n+1)}}\ (\th_1,\dots,
\th_{n+1})- K_{\underline{\omega}^{(n)}}(\th_1,\dots,
\th_{n}) \big]\,.
\end{align}
We fix $i\in \Z^N$ and we estimate the $i$-th  component of
$K_{\underline{\omega}^{(n+1)}}-K_{\underline{\omega}^{(n)}}$. By
the triangle inequality
\begin{align}\nonumber
\big|\big[K_{\underline{\omega}^{(n+1)}}-K_{\underline{\omega}^{(n)}}\big]_i\big|_{\rho_{n+1}}
&\le\big|\big[K_{\underline{\omega}^{(n+1)}}-K_{\underline{\omega}^{(n)}}-\tau^{m_{n+1}}\
K_{\omega_{n+1}}\big]_i\big|_{\rho_{n+1}}\\ \label{cotesKn}
&\quad +\big| \tau^{m_{n+1}}\ K_{\omega_{n+1}}\big]_i\big|_{\rho_{n+1}}\,.
\end{align}
The first term in the right-hand side of \eqref{cotesKn} is
bounded by
$$\frac{e}{2^{n-1}}\ \max_{1\le k\le n+1}\ \Gamma_{\beta_{n+1}}\ (i-c_k)$$
and the second one is bounded by (see \eqref{126bis})
$$
\big|\big[\tau^{m_{n+1}}\
K_{\omega_{n+1}}]_i\big|_{\rho_1}=\big|\big[K_{\omega_{n+1}}]_{i+m_{n+1}}|_{\rho_1}
\le \|K_{\omega_{n+1}}\|_{\rho_1,0,\Gamma}\ \Gamma(i+m_{n+1}) .
$$
This implies that the $i$-th component of the sum in \eqref{sumaKn}
is bounded by
\begin{equation}
\sum^\infty_{n=1}\ \frac{e}{2^{n-1}}+\sum^\infty_{n=1}\
\|K_{\omega_{n+1}}\|_{\rho_1,0,\Gamma}\ \Gamma(i+m_{n+1})\,.
\label{sumaKronecker}
\end{equation}

Therefore the $i$-th component of \eqref{sumaKn} converges uniformly
on compact sets of $(\T^l)^\NN$ and
$K_{\underline\omega}$ is analytic in the sense of
Remark \ref{rmk-analit}.
\end{proof}
The following result proves that $K_{\underline\omega}$
is a solution of the invariance equation and therefore is an
almost-periodic function of the initial system.
\begin{lemma}
The limit function $K_{\underline\omega}$ satisfies
$$F\circ K_{\underline\omega}=K_{\underline\omega}\circ T_{\underline\omega}$$
component-wise.
\end{lemma}
\begin{proof}
For every $n\in\NN$ one has
\begin{equation}\label{invarianceomegan}
F\circ K_{\underline\omega^{(n)}}=K_{\underline\omega^{(n)}}\circ
T_{\underline\omega^{(n)}}\,.
\end{equation}
We fix a component $i\in\Z^N$. The passage to the limit in the
right-hand side of \eqref{invarianceomegan} is immediate. For the
left-hand side we take $n_0$ such that $|c_n|>|i|$ for $n>n_0$.
Then for $n>n_0$ we have
\begin{equation}\label{difFK}
\big|F_i \circ K_{\underline\omega^{(n)}}-F_i\circ
K_{\underline\omega}\big|_{\rho_\infty}\le \sum_j\
\Big|\frac{\partial F_i}{\partial x_j}\Big|\
\big|\big[K_{\underline\omega^{(n)}}-K_{\underline\omega}\big]_j\big|_{\rho_\infty}\,.
\end{equation}
We estimate
\begin{align*}
&\big|\big[K_{\underline\omega^{(n)}}-K_{\underline\omega}\big]_j\big|_{\rho_\infty}\le
\sum^\infty_{p=n}\
\big|\big[K_{\underline\omega^{(p)}}-K_{\underline\omega^{(p+1)}}\big]_j\big|_{\rho_{p+1}}\\
&\le \sum^\infty_{p=n}\ \Big[\frac{e}{2^{p-1}}\ \max_{1\le k\le
p+1}\ \Gamma_{\beta_{p+1}}\ (j-c_k)+\big|\big[\tau^{-c_{p+1}}\
K_{\omega_{p+1}}\big]_j\big|_{\rho_{p+1}}\Big]\\
&\le \sum^\infty_{p=n}\ 
\Big[\frac{e}{2^{p-1}}\ \max_{1\le k\le p+1}\
\Gamma_{\beta_{\infty}}\ (j-c_k)
+ \|K_{\omega_{p+1}}\|_{\rho_1,0, \Gamma} \Gamma (j-c_{p+1}) 
\Big]\,.
\end{align*}
Then \eqref{difFK} is bounded by
\begin{align*}
\sum^\infty_{p=n}&
\|F\|_{C^1_\Gamma} 
\Big[\frac{e}{2^{p-1}} \sum_j \max_{1\le k\le p+1}\
\Gamma_{\beta_{\infty}} (j-c_k) \Gamma(i-j) \\ 
& \qquad + \|K_{\omega_{p+1}}\|_{\rho_1,0, \Gamma} 
\sum_j \Gamma(i-j) \Gamma (j-c_{p+1}) \Big]\\
& \le \|F\|_{C^1_\Gamma}\Big( \sum^\infty_{p=n} \frac{e}{2^{p-1}}
\frac{2}{1-e^{-\beta_\infty}} 
+\sum^\infty_{p=n}
\|K_{\omega_{p+1}}\|_{\rho_1,0, \Gamma} 
\Gamma (i-c_{p+1})\Big) \\
& \leq \|F\|_{C^1_\Gamma} \Big(\frac{2}{1-e^{-\beta_\infty}} 
\frac{e}{2^{n-2}} +C_K\frac{2}{1-e^{-\beta}} \max _{p\ge n}\Gamma(i-c_{p+1})
\Big),
\end{align*} 
which leads to the desired result.
\end{proof}

\section*{Acknowledgements} 
The work of E.F. has been partially supported by the spanish grant MTM-16425 and the catalan grant 2009SGR67. The work of R.L. has been supported by NSF grants. 
R.L. also thanks the hospitality of Univ. Polit. Catalunya
and Centre de Recerca Matem\`atica which enabled 
face to face collaboration.  The work of Y.S. has been supported byt the ANR project "KAMFAIBLE". 

We thank very specially Prof.  P. Mart\'{\i}n for very 
illuminating comments 
and suggestions. We also thank many suggestions 
and encouragement from Profs. M. Jiang, X. Li, L. Sadun,  E. Valdinoci,
Drs. R. Calleja, T. Blass, D. Blazevski, X. Su. D. Blazevski gave 
a detailed reading to the paper which improved the exposition. 

\appendix

\section{Appendix: Decay functions}
\label{decayfunctions}

This appendix is devoted to the properties of spaces of decay functions. 

\subsection{Linear and $k$-linear maps over $\ell^\infty(\ZZ^N$)}
\label{sec:linear}

We are going to consider linear  maps from $\ell^\infty (\ZZ^N)$ into itself
such that
\begin{equation}\label{limitvanish}
\lim_{m \to \infty} \sup_{u\in \ell^\infty, \, |u|\le 1\atop u_j=0,\, |j-i|\le m}
(A u)_i =0 , \qquad  \forall i\in \Z^M.
\end{equation} 
The  condition 
\eqref{limitvanish} 
is equivalent to the fact that $A$ can be written in the form
\begin{equation} \label{sumform}
(Au)_i=\sum_{j \in \ZZ^N} A_{ij} u_j, \qquad i \in \ZZ^N,
\end{equation}
where $A_{ij}$ are linear maps, $u \in \ell^\infty (\ZZ^N)$ and the series are convergent. 

We denote by $\mathcal L(\ell^\infty(\ZZ^N))$ the space of 
linear maps that satisfy \eqref{sumform}. This is a non-trivial assumption, since $\ell^\infty(\Z^N)$ is not reflexive. Furthermore, the space $\mathcal L(\ell^\infty(\ZZ^N))$ is a strict subspace of the space of bounded linear operators on $\ell^\infty(\ZZ^N)$. 

The second assumption we will make is 
that  there exists $C>0$ and a decay function $\Gamma$ such that
\begin{equation*}
|A_{ij}| \leq C \Gamma(i-j),
\end{equation*}
for all $(i,j) \in (\ZZ^N)^2$.

In this case,
\begin{equation}\label{cotaserieA}
\sum_{j \in \ZZ^N} \,|A_{ij}u_j|\le \sum_{j \in
\ZZ^N} C\Gamma(i-j) |u_j| \le \sum_{j \in \ZZ^N}
C\Gamma(i-j) \|u\|_\infty \le C \|u\|_\infty.
\end{equation}
Then we define
$$
\L_{\Gamma}(\ell ^\infty(\ZZ^N)) =\Big\{A\in \L(\ell
^\infty (\Z^N)) \mid\, \sup_{i,j \in
\ZZ^N}\Gamma(i-j)^{-1} |A_{ij}| <\infty \Big\}
$$
and we endow it with
the norm
\begin{equation}\label{normMap}
\|A\|_{\Gamma}=\sup_{i,j \in \ZZ^N} \Gamma(i-j)^{-1}|A_{ij}|.
\end{equation}
The following Lemma has a simple proof that can be found 
in \cite{FLM}.  The most subtle point is 
that we need to verify that the linear operators are 
given by the matrix.

\begin{lemma}\label{lgamma}
The space $\L_{\Gamma}(\ell ^\infty (\ZZ^N))$ is a Banach
space.
\end{lemma}
{From} the definition of the norm of $A$ and the inequalities
\eqref{cotaserieA} we deduce that
$$
\|Au\|_\infty \le \|A \|_\Gamma\, \|u\|_\infty
$$
for all $u\in \ell^\infty(\Z^N)$.

\begin{remark}
The previous definition will also be used for matrices. If one considers a {\sl finite} set of indexes $I \times J \subset \ZZ^N \times \ZZ^N$, we will use
\begin{equation*}
\|A\|_{\Gamma}=\sup_{i \in I, j \in J} |A_{ij}| \Gamma^{-1}(i-j).
\end{equation*}

This is just a way to say that the set of tensors of order $2$
with finite indexes are naturally embedded into the set of tensors
of order $2$ on $\ZZ^N$ by just setting all the remaining values
to $0$.
\end{remark}

Similarly to the previous definition, we define
$\L^k(\ell^\infty(\ZZ^N))$ as the space of  $k$-linear
maps on $\ell^\infty(\ZZ^N)$ which 
are represented by a multilinear matrix. 
\begin{equation}\label{k-multilinear}
B(u^1,\dots ,u^k)_i= \sum_{(i_1,\dots ,i_k) \in (\ZZ^N)^k}
B_{i,i_1,\dots ,i_k} u^1_{i_1}\dots u^k_{i_k},
\end{equation}
where $i,i_1,\dots,i_k\in \ZZ^N$, $(u^1,\dots ,u^k) \in
(\ell^\infty(\ZZ^N))^k$ and $B_{i,i_1,\dots,i_k}\in\L^
k(M,M)$.

Given a decay function $\Gamma$, 
we define $\L^k_\Gamma(\ell^\infty(\ZZ^N))$ as  the space of  maps in $\L^k(\ell^\infty(\ZZ^N))$ such that 
\begin{equation*}
|B_{i,i_1,\dots ,i_k}| 
\leq C \min(\Gamma(i-i_1),\dots ,\Gamma(i-i_k)),
\end{equation*}
for some $C \ge 0$. 

We define
\begin{equation*}
\|B\|_\Gamma= \sup_{i,i_1,\dots ,i_k \in \ZZ^N}
|B_{i,i_1,\dots ,i_k}|\max (\Gamma^{-1}(i-i_1),\dots
,\Gamma^{-1}(i-i_k)).
\end{equation*}

We have
the following lemma (see \cite{FLM}).
\begin{lemma}
The space $\L^k_{\Gamma}(\ell^\infty(\ZZ^N))$ is a
Banach space.
\end{lemma}

The next result provides the Banach algebra property of
$\L_\Gamma (\ell^\infty(\ZZ^N))$. This property will be
crucial for our estimates. It makes it possible to 
work with infinite dimensional systems in a way which is 
not very different from the finite dimensional case. 

\begin{lemma}
If $A,B \in \L_\Gamma(\ell^\infty(\ZZ^N))$ then $AB
\in \L_\Gamma(\ell^\infty(\ZZ^N))$ and we have the
estimate

$$\|AB\|_\Gamma \leq \|A\|_\Gamma \|B\|_\Gamma. $$
\end{lemma}

\begin{proof}
It is easy to verify that if $A$ and $B$ 
can be represented by matrices, so is the product. Since $A,B \in \L_\Gamma(\ell^\infty (\Z^N))$, we have
\begin{equation*}
|A_{ij}| \leq  \Gamma(i-j) \|A\|_\Gamma, \quad
|B_{jk}| \leq  \Gamma(j-k) \|B\|_\Gamma.
\end{equation*}

Therefore, we have:
\begin{equation*} 
\begin{split} 
\|AB\|_\Gamma &= \sup_{n,m \in \ZZ^N }
|(AB)_{nm}|\Gamma^{-1}(n-m)\\
& \leq \sup_{n,m \in \ZZ^N} \sum_{k \in \ZZ^N}
|A_{nk}|\,|B_{km}| \Gamma^{-1}(n-m)\\
 & \leq  \|A\|_\Gamma \|B\|_\Gamma \sup_{n,m \in \ZZ^N} \sum_{k \in \ZZ^N}
\Gamma(n-k)\Gamma(k-m)  \Gamma^{-1}(n-m).
\end{split} 
\end{equation*}
Using property (2) of Definition \ref{defDecay} we
obtain the desired result.
\end{proof}

By induction on $k$ the same result holds for
$k-$linear maps. See \cite{FLM}.

\begin{lemma}\label{lemAB}
Let $A\in\L^k_\Gamma(\ell^\infty(\ZZ^N))$ and
$B_j\in\L^{n_j}_\Gamma (\ell^\infty(\ZZ^N))$ for $1\le j\le k$.
Then the composition $A B_1\dots B_k\in\L_\Gamma^{n_1+\dots+n_k}
(\ell^\infty( \Z^N))$ and
$$\|A B_1\dots B_k\|_\Gamma \le \|A\|_\Gamma\,\|B_1\|_\Gamma \dots
\|B_k\|_\Gamma\,.$$
\end{lemma}

\subsection{Spaces of differentiable and analytic functions on lattices}
\label{sec:differentiable}

We now define the space of $C^r$ functions, based on the previous weighted norms. Given an open set $\B \subset
\M$ we define
\begin{equation*}
C^1_\Gamma(\B)=\left \{
\begin{array}[c]{cc}
F: \B \to \M\mid\, F \in
  C^1(\B),\,\,DF(x) \in
  \L_\Gamma(\ell^\infty (\Z^N)),\\
\sup_{x \in \B} \|F(x)\| < \infty  ,\ \sup_{x \in \B}
\|DF(x)\|_\Gamma < \infty
\end{array} \right \}.
\end{equation*}

We endow $C^1_\Gamma (\B)$ with the norm $\|F\|_{C^1_\Gamma}=\max
\big(\sup_{x\in\B} \|F(x)\|\, ,$\break $\sup_{x\in\B}
\|DF(x)\|_\Gamma\big).$ In the definition of $C^1_\Gamma (\B)$, when
$\M$ is complex, the derivative has to be understood as complex
derivative.

We emphasize that the definition of 
$C^1_\Gamma$ includes that
$DF(x)\in
\L_\Gamma(\ell^\infty (\Z^N))$ and, in particular, that the 
derivative of the function is given by the matrix of its 
partial derivatives. 
Concretely if $F \in C^1_{\Gamma}$ then we have the following
formula:
\begin{equation*}
DF_i(x)v=\sum_{j\in \ZZ^N} \frac{\partial F_i}{\partial
  x_j}(x)v_j ,
\end{equation*}
where $x_j$ is the variable in $M$, the j-th component of $\M$.

Now we proceed to define the space of finite differentiable maps.
In \cite{JL00,FLM}, one can find 
definitions for H\"older spaces, which is useful for other 
applications (such as thermodynamic formalism). 

\begin{defi}\label{Crgamma}
Given $\mathcal B$ an open subset of $\M$  and $r \in \nat$
\begin{equation*}
C^r_\Gamma(\B)=\left \{
\begin{array}[c]{cc}
F: \B \to \M\mid\, F \in
  C^r(\B),\, D^j F(x) \in
  C^1_\Gamma(\B),\\
0\le j\le r-1
\end{array} \right \}.
\end{equation*}
\end{defi}

\subsection{Spaces of embeddings from $\C^l$ to $\M$ with decay properties}
\label{sec:embeddings} 
In this section, we will  consider embeddings from 
finite dimensional tori into the phase space. 
We should think of these embeddings as describing some oscillations
centered around some sites.

We define the complex strip
\begin{equation*}
D_{\rho}=\left \{ z\in \mathbb{C}^l/\ZZ^l| \,\,|\mbox{Im}\,z_i|
  < \rho,\,\,i=1,\dots ,l\right \}.
\end{equation*}
Let $R \geq 1$ be an integer and consider $\underline{c} \in (\ZZ^N)^R$, i.e.
\begin{equation*}
\underline{c}=(c_1,\dots ,c_R).
\end{equation*}
Given $f:D_\rho \to \M$, we introduce the following quantity
\begin{equation*}
\|f\|_{\rho,\underline{c},\Gamma}=\displaystyle{\sup_{i \in \ZZ^N} \min_{j=1,\dots ,R}}
\Gamma^{-1}(i-c_j) \|f_i\|_{\rho},
\end{equation*}
where
\begin{equation*}
\|f_i\|_{\rho}=\sup_{ \th \in D_\rho} |f_i(\th)|.
\end{equation*}

\begin{defi}\label{embeddingspace}
We denote

\begin{equation}
\mathcal{A}_{\rho,\underline{c},\Gamma}= \left \{
\begin{array}[c]{cc}
f: D_\rho \to \M\mid f \in
C^0(\overline{D}_\rho),\,f\mbox{ analytic in} \,D_\rho,\\
\|f\|_{\rho,\underline{c}, \Gamma} < \infty
\end{array}\right \}.
\end{equation}
\end{defi}

This space, with the norm $\|\cdot \|_{\rho,\underline{c},
\Gamma}$, is a Banach space.

The parameter
$\underline{c}$ is the location of the centers of the oscillations 
of the map $f:
\torus^l \to \M$. As the argument of $f$ 
changes, the range of the embedding, will oscillate mainly 
on the sites in neighborhoods of  $c_1,\dots,c_R$.

The next result is a version of the 
Cauchy estimates in our context.  
\begin{lemma} \label{Cauchyestimates}
Let $f:D_\rho \to \M$ be an analytic function. Then for
all $\delta \in (0,\rho)$, the following holds
\begin{equation*}
\|D_\th f\|_{\rho-\delta,\underline{c},\Gamma} \leq l
\delta^{-1} \|f\|_{\rho,\underline{c}, \Gamma}.
\end{equation*}
\end{lemma}
\begin{proof}
Consider the components $f_i$ of $f$, for $i \in \ZZ^N$.
Each $f_i$ maps $D_\rho$ into $M$ and we have the standard Cauchy
estimates for $k=1,\dots ,l$
\begin{equation*}
\|\partial_{\th_k}f_i\|_{\rho-\delta} \leq \delta^{-1}
\|f_i\|_\rho.
\end{equation*}
The result then follows by just multiplying this inequality by
$\Gamma^{-1}(i-c_j)$ which is positive, summing with respect to
$k$ and taking the supremum for $i$ and the minimum for $j$.
\end{proof}

\begin{remark}
Note that in the previous Cauchy estimate the bound depends {\sl
linearly} on the dimension of the torus $l$. 
This will be important when we consider the limit of many 
dimensions. 

On the other hand, 
taking the supremum over components, makes it clear that the 
bounds are independent of the dimension of the range. In particular, 
we can discuss mappings into infinite dimensions. Note that, if we had  taken another norm, the constants would have depended on the dimension 
of the range. 
\end{remark}

If we consider a map  $A$ from $D_{\rho}$ into the set of
linear maps $\L_\Gamma(\ell^\infty(\ZZ^N))$, the associated norm is
\begin{equation*}
\|A\|_{\rho,\Gamma}=\displaystyle{\sup_{i,j\in
  \ZZ^N}}\ \displaystyle{ \sup_{\th \in D_\rho}}\Gamma^{-1}(i-j)|A_{ij}(\th)| = \sup_{\th \in D_\rho} \|A(\theta)\|_\Gamma .
\end{equation*}

\begin{remark}
The previous definition of the space of analytic maps in the strip can be
generalized to any open subset $\B$ of the complex extended manifold
$\M^{\mathbb{C}}$. We define
\begin{equation*}
\mathcal{A}_{\B,\underline{c},\Gamma}= \left \{
\begin{array}[c]{cc}
F: \B\to \M\mid F \in
C^0(\overline \B),\,\,F\,\,analytic\,\,in\,\,\B,\\
\|F\|_{\B,\underline{c},\Gamma} < \infty
\end{array}\right \},
\end{equation*}
where
\begin{equation*}
\|F\|_{\B,\underline{c},\Gamma}=\sup_{i \in \ZZ^N}
\min_{j=1,\dots ,R} \Gamma^{-1}(i-c_j) \sup_{ z \in \B} |F_i(z)|.
\end{equation*}

Note that the space $\mathcal{A}_{\B,\underline{c},\Gamma}$ is a
closed subspace of $C^1_\Gamma(\B)$ for the $C_\Gamma^1$ topology.
Here, the derivatives are understood as complex derivatives.
\end{remark}

\subsection{Spaces of localized vectors} 
\label{sec:localized} 
The space of localized vectors 
$\ell^\infty_{\underline{c}, \Gamma}$ defined below, 
plays  a role 
as the space of  infinitesimal deformation of the space of 
localized embeddings defined above. We also isolate a class of linear operators 
$\L_{\underline{c}, \Gamma}$ which send 
$\ell^\infty$ into $\ell^\infty_{\underline{c}, \Gamma}$.  

The key property (Proposition~\ref{pro:ideal})  
is that $\L_{\underline{c}, \Gamma}$
is an ideal of the Banach algebra $\L_\Gamma$. It will be important for 
future developments that the bounds obtained
are independent of the parameter 
$\underline{c}$. This will be an easy consequence 
of the Banach algebra properties of the decay functions.

\begin{defi}\label{def:localized} 
Given a decay function $\Gamma$ and a finite (or infinite)  collection of 
sites $\underline{c} = \{c_k\}_{k \in \K} \subset \ZZ^N$ with $\K \subset \nat$
we define
\begin{equation}\label{normlocalized} 
\| v \|_{\underline{c}, \Gamma} = 
\sup_{i \in \ZZ^N} \inf_{k \in \K} |v_i| \Gamma(i - c_k)^{-1} .
\end{equation} 

We denote 
\[
\ell^\infty_{\underline{c}, \Gamma} = 
\{ v \in (\real^l)^{\ZZ^N} \, | \, \| v \|_{\underline{c}, \Gamma} < \infty \}.
\]

We denote by $\L_{\underline{c},\Gamma}$ the 
space of linear operators on $\ell^\infty$ such that 
\[
\begin{split} 
&(A v)_i = \sum_{j \in \ZZ^N}  A_{ij} v_j, \\
& |A_{ij}| \le  C \min( \sup_{k \in \K} \Gamma(i - c_k), \Gamma(i -j) ).
\end{split} 
\]
We denote by $\| A \|_{\underline{c}, \Gamma}$ the 
best constant $C$ above, i.e. 
\[
\| A\|_{\underline{c}, \Gamma}  
= \max \Big ( \sup_{i, j \in \ZZ^N} |A_{ij}| \Gamma^{-1}(i-j),
\sup_{i, j \in \ZZ^N} |A_{ij}| \Gamma^{-1}(i-j),
\sup_{i, j \in \ZZ^N}\inf_{k \in \K} |A_{ij}| \Gamma^{-1}( i - c_k) \Big ).
\]
\end{defi} 

Note that we use $\| \cdot \|_{\underline{c}, \Gamma}$ 
both for the norm in a linear space and the norm in 
the space of operators. This will not cause any 
confusion since in this space we will not use 
the norm of operators from $\ell^\infty_{\underline{c}, \Gamma}$ 
to itself. 

The following is an easy exercise. 

\begin{pro} \label{pro:ideal}
We have the following results:

\begin{itemize} 
\item{a)}
The space $\ell^\infty_{\underline{c}, \Gamma}$ endowed with
$\| \cdot \|_{\underline{c}, \Gamma}$ is a Banach space. 

The embedding $\ell^\infty_{\underline{c}, \Gamma} \hookrightarrow \ell^\infty$
is continuous. 

$\ell^\infty_{\underline{c}, \Gamma}$ is a closed subspace of $\ell^\infty$. 
\item{b)}
The space $\L_{\underline{c}, \Gamma}$ endowed with
$\| \cdot \|_{\underline{c}, \Gamma}$ is a Banach space.

The embedding $\L_{\underline{c}, \Gamma} \rightarrow \L_\Gamma$
is continuous. 

$\L_{\underline{c}, \Gamma}$ is a closed subspace of $\L_\Gamma$.
\item{c)} Ideal character: If $A \in \L_{\underline{c}, \Gamma}$, $B \in \L_{\Gamma}$, we
have 
\begin{equation}\label{ideal}
\begin{split} 
& A B \in  \L_{\underline{c}, \Gamma}, \quad 
\| A B \|_{\underline{c}, \Gamma} \le \| A\|_{\underline{c}, \Gamma} 
\| B\|_{\Gamma}, \\
& B A \in  \L_{\underline{c}, \Gamma}, \quad 
\| B A \|_{\underline{c}, \Gamma} \le \| A\|_{\underline{c}, \Gamma} 
\| B\|_{\Gamma} \\
\end{split}
\end{equation} 

As a consequence of the above, if $A \in \L_{\underline{c}, \Gamma}$,
 $B \in \L_{\underline{c}, \Gamma}$, we
have 
\begin{equation}
\| A B \|_{\underline{c}, \Gamma}, 
\| B A \|_{\underline{c}, \Gamma} \le \| A\|_{\underline{c}, \Gamma}\| B\|_{\underline{c}, \Gamma}.  
\end{equation} 
\end{itemize} 
\end{pro} 
Note also that if $x \in \M$ and $A \in \L_{\underline c, \Gamma}$ then $Ax \in \ell^\infty_{\underline c, \Gamma}. $

\subsection{Regularity of the composition operators}
\label{sec:composition}

The following propositions (see \cite{JL00}) 
establish the regularity of composition operators
and provide estimates for the composition. 

\begin{pro}\label{compo}
The mapping defined by
\begin{equation*}
\mathcal{C}(G,h)=G \circ h
\end{equation*}
is locally Lipschitz when considered as
\begin{equation*}
\mathcal{C}: C^2_{\Gamma}
\times C^1_{\Gamma}\mapsto
C^1_{\Gamma},
\end{equation*}
and we have the estimate

\begin{equation*}
\|\mathcal{C}(G,h+\bar{h})-\mathcal{C}(G,h)\|_{C_\Gamma^1}\leq
\|G\|_{C_\Gamma^2} \|\bar{h}\|_{C_\Gamma^1}(1+\|h\|_{C_\Gamma^1}),
\end{equation*}
Furthermore, when considered as a mapping from $C^3_{\Gamma}
\times C^1_{\Gamma}$ into $ C^1_{\Gamma}$, we have the formula
\begin{equation} \label{derivativecompos}
D_2 \mathcal{C}( G, h ) \Delta =  (DG \circ h )\, \Delta .
\end{equation}
\end{pro}

We will also need the following estimate on the composition operator.
\begin{lemma}\label{ineqGamma}
Consider two functions $G,h \in
C^1_\Gamma$. Then we have
\begin{equation*}
\|G \circ h \|_{C^1_\Gamma} \leq C \max(\|G\|_{C^0},  \|G \|_{C^1_\Gamma} \|h\|_{C^1_\Gamma}). 
\end{equation*}
\end{lemma}
\begin{proof}
Clearly, we have
\begin{equation*}
\|G \circ h \|_{C^0} \leq \|G \|_{C^0}.
\end{equation*}
We now estimate the norm of the derivatives:
\begin{equation*}
D_j(G \circ h)_i= \sum_{k \in \ZZ^N} D_kG_i \circ h D_j h_k.
\end{equation*}
This leads
\begin{eqnarray*}
|D_j(G \circ h)_i| \leq \sum_{k \in \ZZ^N} \Gamma(i-k)
\Gamma(k-j) \|G\|_{C^1_\Gamma} \|h\|_{C^1_\Gamma}\\
\leq \Gamma(i-j) \|G\|_{C^1_\Gamma} \|h\|_{C^1_\Gamma}.
\end{eqnarray*}
This ends the proof.
\end{proof}

The next lemma gives an estimate on the composition of a mapping
defined on the manifold and an embedding.

\begin{lemma}\label{compo2}
Let $\B\subset \M$ be a star-like from the origin open set such that $0\in\B$. Suppose that $F\in C^1_\Gamma
(\B)$ and is analytic. Let $K:D_\rho\to \M$ belong to
$\A_{\rho,\underline c,\Gamma}$, with $\underline c\in(\Z^N)^R$
and such that $K(\overline D_\rho)\subset\B$.
\begin{enumerate}
\item Assume that $F(0)=0$.
 Then $F\circ K\in
\A_{\rho,\underline c,\Gamma}$, and 
\begin{equation}\label{cotaFK}
\|F\circ K\|_{\rho,\underline c,\Gamma}\le R\|F\|_{C^1_\Gamma}\
\|K\|_{\rho,\underline c,\Gamma}\,.
\end{equation}
\item Let $\J\subset\Z^N$ be a finite set of indexes and assume
that $F_j(0)=0$ for $j\in \Z^N-\J$. Then $F\circ K\in
\A_{\rho,\underline c,\Gamma}$ and
\begin{equation}\label{cotaFK2}
\|F\circ K\|_{\rho,\underline c,\Gamma}\le \|F\|_{C^1_\Gamma}\
\big(C+R\|K\|_{\rho,\underline c,\Gamma}\big)\,,
\end{equation}
where $C$ depends on $\underline c$ and $\J$.
\end{enumerate}
In both cases
\begin{equation}\label{cotaDFK}
\|D(F\circ K)\|_{\rho,\underline c,\Gamma} \le
R\|F\|_{C^1_\Gamma}\ \|DK\|_{\rho,\underline c,\Gamma}.
\end{equation}
\end{lemma}
\begin{proof}
By Definition \ref{analyt} of analytic functions,  we have that 
$F\circ K$ is analytic. To estimate $\|F\circ K\|_{\rho,\underline
c,\Gamma}$ take $i\in\Z^N$ and $j\in \{1,\dots,R\}$. If $F_i(0)=0$
we can write
$$F_i\big(K(\th)\big)=\int^1_0 DF_i \big(sK(\th)\big)\, K(\th)ds=
\int^1_0 \sum_{p\in\Z^N} \frac{\partial F_i}{\partial x_p}
\big(sK(\th)\big)\,K_p(\th)ds\,.$$ Taking norms
\begin{align*}
|F_i\big(K(\th)\big)|&\le
\sum_{p\in\Z^N}\,\|DF\|_\Gamma\,\Gamma(i-p)\,\|K\|_{\rho,\underline
c,\Gamma}\ \max_{1\le j\le R}\,\Gamma(p-c_j)\\
&\le R \|F\|_{C^1_\Gamma}\, \|K\|_{\rho,\underline c,\Gamma} \
\max_{1\le j\le R}\,\Gamma(i-c_j)
\end{align*}
and then, if $F(0)=0$, \eqref{cotaFK} follows.

In the second case $F_m(0)\ne 0$ for $m \in \mathcal J$,  we have
$$|F_m(K(\th))|\le |F_m(0)|+ R \|F\|_{C^1_\Gamma}\,\|K\|_{\rho,\underline c,\Gamma}
\max_{1\le j\le R} \Gamma(i-c_j)$$ and then we obtain
\eqref{cotaFK2} with
$$C=\max_{m\in\J}\ \min_{1\le j\le R}\ \Gamma^{-1}(m-c_j)\,.$$
The estimate \eqref{cotaDFK} follows from the chain rule and the
definitions of the norms.
\end{proof}

\section{Appendix: Symplectic geometry on lattices}
\label{sec:symplectic}

A symplectic structure on an infinite dimensional manifold is not easy
to define (see \cite{chernoffM74}, \cite{bambusi99}). 
Fortunately, the KAM theory presented  here 
uses  only very few properties of 
symplectic geometry.

The aim of the next sections is to develop such ideas
and give precise definitions of the theory of symplectic forms 
we will need. Note that we do not need to develop 
a systematic geometry. We just need to deal with the 
standard symplectic form in $\M$, its primitives, its push-forward
and perform just a few operations. This can be readily 
justified in spite of the difficulties with more sophisticated material.

\subsection{Forms on lattices}

Remember that a form is just an antisymmetric real valued
multilinear operator
on the tangent space. 

We just need to study local forms which are 
the product of forms in each of the ambient spaces.

We introduce $\pi_i : \M \to M$, the projection $\pi_i(x)=x_i$ for $i \in \Z^N$.  Given a collection of smooth $k$-forms $\gamma_i\in \Lambda^k(M)$,
such that $\sup_i  \|\gamma_i\|<\infty$, we define a formal form in
$\M$ as follows
\begin{equation}\label{formalform}
\gamma=\sum_{i \in \mathbb{Z}^N} \pi^*_i\gamma_i\,,
\end{equation}
that is
\begin{equation*}
\gamma(x)(u_1,\dots ,u_k)=\sum_{i \in \mathbb{Z}^N} \gamma_i(\pi_i
(x))(\pi_i u_1,\dots ,\pi_i u_k)
\end{equation*}
for $x\in\M$ and  $(u_1,\dots ,u_k) \in (T_x \M)^k$.
We denote
$$\bar{\Lambda}^k_\infty(\M)=\Big\{\gamma= \sum_i\pi^*_i \gamma_i\Big\}$$
the set of such forms. 

Of course, this form \eqref{formalform} in general does not define a
multilinear function on bounded vector fields, so that it should be 
understood only formally. Nevertheless, we will
show that there are several operations among forms
that can be made sense of in the infinite dimensional setting.

Roughly, we will see that these formal forms make sense acting on 
vectors that decay away from a finite set of centers. We can also
push them forward by a decay diffeomorphism and pull them back
by a decay embedding.  They can also be integrated and, in 
some weak sense, differentiated. These will be all the operations 
that we will need. Moreover, we will only need $k = 1,2$.

When $k = 2$, if each of the $\gamma_i$ are 
uniformly non-degenerate, we can define an identification operator
defined by 
\begin{equation}\label{jdefined} 
(J_\infty u)_i = J_i \pi_i u, \quad i \in \ZZ^n,
\end{equation}
where $J_i$ is the operator of identification on the $i$ copy of 
the manifold i.e. 
 $$ \gamma_i(x) (\xi,\eta) = \langle \xi, J_i(x)  \eta \rangle, \,\, \forall
\xi, \eta \in T_x M_i .$$ 
We emphasize that, given the formula \eqref{jdefined}, it is 
clear that when the $\gamma_i$ are uniformly non-degenerate
(i.e. $\| J_i(x) \|, \|J_i^{-1}(x) \|$ are  bounded uniformly in $i,x$) 
we have that the operator $J_\infty$ is bounded and its inverse is 
also bounded. Note that in the KAM method 
of \cite{LGJV05, FontichLS09, FontichLS09b}, the symplectic properties appear 
mainly  through $J, J^{-1}$ and their invariance properties. 

In the main application to the construction of almost periodic 
solutions, when the system has 
translation invariance, all the $\gamma_i$ are identical. Nevertheless
we do not assume that the $\gamma_i$ are given in the standard form.
This is useful e.g. in dealing with oscillators, or chemical molecules
whose action angle variables are singular.

Let $\gamma \in \bar{\Lambda}^k_\infty(\B)$, $F:\B_1 \to \M$ with $F(\B_1) \subset \B $.
 We define the pull-back $F^*\gamma$ by
\begin{equation}\label{pullback1} 
F^*\gamma =\sum_{i\in\Z^N} F^*\gamma_i\,,
\end{equation}
that is
\begin{equation}\label{pullback2}
F^*\gamma(x)(u_1,\dots ,u_k)=\sum_{i \in \ZZ^N} \gamma_i(F (x))(DFî (x) u_1,\dots ,DFî (x)u_k).
\end{equation}

For a general diffeomorphism, the sums in 
\eqref{pullback1}, \eqref{pullback2} are 
purely formal. On the other hand, when $F \in C^1_\Gamma$, 
the sums for 
$ F^* \gamma \circ \pi_j(\pi_{i_1} u_1, \ldots \pi_{i_k} u_k)$ 
make sense and converge uniformly.

If $\psi:D_\rho \supset \T^l \to \M$ is a smooth map we define
$\psi^* \gamma$ in the analogous way.
It will be important to emphasize for 
future applications that  when 
$\gamma$ is a formal form and $\psi$  has decay, then 
$\psi^* \gamma$ is a smooth form in 
$D_\rho$. 

An easy computation shows that if $F\in C^\infty(\B_1) $ and $G\in
C^\infty(\B_2) $, with $F(\B_1) \subset \B_2$, then
$$
(G\circ F)^* = F^*\circ G^*.
$$
Also, if $\psi(D_\rho) \subset \B_1 $ we have
\begin{equation}\label{p-bpsiF}
(\psi\circ F)^* = F^*\circ \psi^*.
\end{equation}
Again, this is a formal computation for diffeomorphisms, 
but, when $F,G \in C^1_\Gamma$, then, the calculation 
can be justified. 
Also if $G \in C^1_\Gamma$ and $F \in \A_{\rho,c, \Gamma} $,
then $(G \circ F)^* \gamma$ is a well defined form.

\begin{defi}
Given $\gamma \in\tilde {\Lambda}^k_\infty $ we define
$$
d\gamma= \sum_{i \in \ZZ^N} d \gamma_i.
$$
\end{defi}
We clearly have that $d^2 \gamma= 0$.

\begin{lemma}
Let $F\in C^2_\Gamma(\B_1) $, $\psi:\A_{\rho, c. \Gamma} \to\M$ 
and $\gamma\in \tilde{\Lambda}^k_\infty(\B)$ so that the composition makes 
sense. 
\begin{align*}
F^*d\gamma&=d(F^*\gamma),\\
\psi^*d\gamma&=d(\psi^*\gamma).
\end{align*}
\end{lemma}
\begin{proof}
It consists mainly in going over the formal 
computation, but paying attention to the fact that 
all the steps can be justified by the 
convergence.  We carry out explicitly the 
first one and we let the other one to the reader.

\begin{align*}
F^* d\gamma &=F^* \Big(\sum_i d  \gamma_i\Big)=\sum_i F^* d  \gamma_i =\sum_i d(F^*    \gamma_i)\\
& =d \sum_i  F^*    \gamma_i =  d(F^*    \gamma),
\end{align*}
where we have used that $\gamma _i$ are true differential forms.
\end{proof}

In the case that $\gamma=\sum_{i \in \ZZ^N} \pi^*_i\gamma_i$,
with $\gamma_i\in \Lambda^k(M)$,
$$
d\gamma= \sum_{i \in \ZZ^N} d (\pi^*_i\gamma_i) = \sum_{i \in \ZZ^N} \pi^*_i d\gamma_i,
$$
where $d\gamma_i$ is the exterior differential of $\gamma_i$ in $M$,
and if $\psi:D_\rho \to \M$,
$$
\psi^*\gamma = \sum_{i \in \ZZ^N} \psi^* \pi^*_i\gamma_i  = \sum_{i \in \ZZ^N} \psi^*_i \gamma_i.
$$

We can also define the contraction operator. Given a smooth vector
field $X$ in $\M$ and a $k$-form $\gamma\in
\tilde{\Lambda}^k_\infty$ we set
$$
(i_X \gamma)(x)(u_1,\dots ,u_{k-1}) =\sum_{j \in \ZZ^N}(i_{X}
\gamma_j)(x)(u_1,\dots ,u_{k-1}).
$$
Hence we can also introduce the Lie derivative for formal forms by the usual formula
$$
\L_X \gamma= i_Xd\gamma + d(i_X \gamma) .
$$

\begin{lemma}\label{lemtrueform}
Let $\gamma = \sum_{i \in \ZZ^N} \pi^*_i\gamma_i \in
\bar{\Lambda}^k_\infty$ be a formal form and $\psi: D_\rho\supset
\torus^l \rightarrow \M$ a map with decay,  i.e.
$\psi \in \mathcal{A}_{\rho,\underline{c},\Gamma}$ with
$\underline{c}=(c_1,\dots ,c_R)$. Then for $0<\delta<\rho$,
$\psi^* \gamma$ is a well-defined $k$-form in $D_{\rho-\delta}$.
\end{lemma}
As a consequence if $F\in C^1_\Gamma(\B)$ is analytic, $\psi \in \mathcal{A}_{\rho,\underline{c},\Gamma}$
 and $\psi(D_\rho)\subset \B$, by Lemma \ref{compo2} and \eqref{p-bpsiF}
we have that $\psi^*F^* \gamma$ is a well-defined $k$-form in
$D_{\rho-\delta}$.
\begin{proof}[Proof of Lemma \ref{lemtrueform}]

By definition of $\gamma$, we have
\begin{equation*}
(\psi^*\gamma)(\th)(u_1,\dots,u_k)=\sum_{i \in \ZZ^N}
\gamma_i (\psi_i(\th))(D\psi_i(\th) u_1,\dots ,D\psi_i(\th) u_k)
\end{equation*}
for $\th \in \T^l$ and $ u_1,\dots ,u_k\in T_\th \T^l$.
Then
\begin{equation*}
\begin{split}
|(\psi^*\gamma)&(\th)(u_1,\dots,u_k)|\\
& \le \sum_{i \in \ZZ^N}
\|\gamma_i \|\,\sum_{m_1}|D_{m_1} \psi_i(\th) (u_1)_{m_1} | \cdots \sum_{m_k}|D_{m_1}  \psi_i(\th) (u_k)_{m_k}|\\
& \le \sum_{i \in \ZZ^N}
\|\gamma_i \|\,\max_{j}\Gamma(i-c_j) |D  \psi| \|u_1 \| \cdots \max_{j}\,\Gamma(i-c_j) |D  \psi| \|u_k \|\\
& \le R^k \|\gamma\| \|D\psi \|_{\rho-\delta,\underline{c},\Gamma
} \|u_1 \| \cdots \|u_k \|\,.
\end{split}
\end{equation*}
We have used that
$\sum_{i \in \ZZ^N}
\max_{j}\Gamma(i-c_j) \le R$.
This proves that the series is absolutely convergent and
so  $\psi^* \gamma$ is well-defined on $T^* \T^l$.
\end{proof}

Therefore, by the previous construction, by pulling back formal
forms on the lattice to the torus, one obtains well-defined
quantities.
\begin{lemma}
For every function $\psi \in \mathcal{A}_{\rho,\underline{c},\Gamma}$, we have
\begin{equation*}
\psi^* d\gamma= d(\psi^* \gamma).
\end{equation*}
\end{lemma}
\begin{proof}
By the definition of $d\gamma$ and the convergence of the series, we have
\begin{equation*}
\psi^* (d\gamma)=\psi^* \Big(\sum_i d \pi_i^* \gamma_i\Big)=\sum_i
\psi^* d \pi_i^* \gamma_i.
\end{equation*}
But by the definition of the exterior differentiation, we have \
\begin{equation*}
\sum_i \psi^* d \pi_i^* \gamma_i=\sum_i d \psi^* \pi_i^* \gamma_i=d \sum_i \psi^* \pi_i^* \gamma_i=d \psi^* \gamma.
\end{equation*}
\end{proof}

\subsection{Some symplectic geometry on lattices}
In this section we discuss the elements of symplectic geometry that 
we will need.  This will play a role in the vanishing lemma
Lemma \ref{vanishingloc}  in Section~\ref{secvanishing}.

Consider a finite dimensional exact symplectic manifold $(M,\Omega=d\alpha)$ and the associated lattice
\begin{equation*}
\M=\ell^\infty(\ZZ^N).
\end{equation*}

Let $\alpha_\infty$ and $\Omega_\infty$ be defined by
\begin{equation*}
\alpha_\infty=\sum_{j \in \ZZ^N} \pi^*_j\alpha, \qquad
\Omega_\infty=\sum_{j \in \ZZ^N} \pi^*_j\Omega.
\end{equation*}
Then $\alpha_\infty\in \bar \Lambda ^1_\infty$ and
$\Omega_\infty\in \bar \Lambda^2_\infty$. Moreover note that
$d\alpha_\infty=\Omega_\infty$. We introduce the following
definitions.
\begin{defi}
We say that a $C^1_\Gamma$ function $F:\M \rightarrow
\M$ is symplectic if the following identity holds for any $z \in \mathcal M$
$$DF^\top(z) J_\infty (F(z)) DF(z)=J_\infty(z) $$
\end{defi}

\begin{defi}
We  say that a $C^1_\Gamma$ function $F:\M \rightarrow \M$ is
exact symplectic on $\M$ if  there exists a one-form $\tilde \alpha$ defined on $T \mathcal M$  with  matrix  $ \tilde A$ such that
\begin{itemize}
\item For every $j \in \ZZ^N$, there exists a smooth function $W_j$ on $M$ such that 
$$\tilde \alpha_j = dW_j  $$ 
where $d$ is the exterior differentiation on $M$. 
\item The following formula holds component-wise on the lattice 
\begin{equation*}
DF(z)^\top \hat A_\infty(F(z)) =\hat A_\infty(z)+ \tilde A(z). 
\end{equation*}
\end{itemize}
\end{defi}

The previous definitions are completely equivalent to the standard
definitions of symplectic and exact symplectic 
maps in the finite dimensional 
case, but they are among the mildest ones that 
we can imagine in infinite dimensions. The following is a straightforward result. 

\begin{lemma}
Let $F \in C^1_\Gamma$ be a map from $\M$ into itself. If $F$ is exact symplectic then it is symplectic.
\end{lemma}

\begin{remark}
Through a localized embedding, this is even easier.  

Since $F$ is exact symplectic, for every decay function
$\psi\in\mathcal{A}_{\rho,\underline{c},\Gamma}$,
 there exists a smooth function $W_\psi$ defined on the torus such that
\begin{equation*}
\psi^*F^*\alpha_\infty=\psi^*\alpha_\infty+dW_{\psi}.
\end{equation*}
By the property of the exterior differentiation and the fact that, by the hypotheses,
$F \circ \psi \in \mathcal{A}_{\rho,\underline{c},\Gamma}$, we have
\begin{equation*}
(F \circ \psi)^* d\alpha_\infty=d(( F\circ \psi)^* \alpha_\infty)=d (\psi^*\alpha_\infty+dW_\psi)=d(\psi^* \alpha_\infty)
= \psi^* d \alpha_\infty.
\end{equation*}
Since $\Omega_\infty= d\alpha_\infty$ this gives the desired
result.
\end{remark}

We now turn to the symplectic geometry of vector fields. We will always be considering vector-fields of the form 
$$X= J_\infty \nabla H$$
where the operator $\nabla$ has to be understood w.r.t. the inner product on $\ell^2(\Z^N)$. The following result is proved.   

\begin{pro} \label{prop:symplectic}
Aasume that the  vector-field $X$ previously defined has decay. Then it generates 
flows consisting of exact symplectic diffeomorphisms. 
\end{pro} 
\begin{proof}
Since $X$ has decay, the operation $i_X \Omega_\infty$ makes sense and one has 
$$i_X \Omega_\infty=dH. $$

The proof then follows the standard one by using the fact that decay vector fields generate decay diffeomorphisms. 
\end{proof}

\section{Appendix : Construction of 
deformations of  symplectic maps which are not exact symplectic}\label{maps}

In the construction of the invariant torus, we are going to use a
family of maps $F_\lambda$ such that $F_0$ is exact symplectic and
$F_\lambda$ is symplectic for all $\lambda$ but not exact symplectic for
$\lambda \neq 0$. Indeed, these  maps will be used 
to kill some averages in the invariance equations, so it will be important 
that, by choosing $\lambda$ appropriately we can obtain all 
the possible cohomology obstructions to exactness. The change on cohomology is more or less proportional 
to the change in the parameter $\lambda$.

The construction of $F_\lambda$ will be done in this section by  considering
flows which are 
locally but not globally Hamiltonian. 
We emphasize that the diffeomorphisms introduced will be 
quite simple. They will just deform a finite number of sites on the lattice. 
In the case that the phase space is $\torus^l \times \real^{2 d - l}$ 
endowed with the standard  symplectic form
the map $F_\lambda$ will be given by 
$A_i \rightarrow A_i + \lambda_i$ where $A_i$ are the variables symplectic 
conjugate to angles. We note that the obstruction 
to exactness are the integrals of the forms 
$A_i d\phi_i$ around a cycle in the torus along $\phi_i$. 
The rest of the section is devoted to 
make a geometrically natural construction that works in all 
manifolds.

Consider $(M_i,\Omega_i=d\alpha_i)_{i \in \ZZ^N}$
a family of finite dimensional exact symplectic manifolds
and denote $\M$ the phase space of the associated lattice map. 

Let $\J \subset \ZZ^N$ be a {\sl finite} set of indexes. We denote
by $H^1(M_i)$ the first de Rham cohomology group of the manifold
$M_i$ and assume that it is non-trivial. Consider
$(\delta^i_k)_{k=1,\dots ,l}$ a basis of $H^1(M_i)$. Since $\Omega_i$ are non-degenerate, one can construct
a family of vector fields $Y^\lambda_i$ on $M_i$ with indexes $i \in \J$ such that 
\begin{equation*}
i_{Y^\lambda_i} \Omega_i =\sum_{k=1}^{l} \lambda�_k \delta^i_k.
\end{equation*}
Note that $Y^\lambda_i$ only depends on $x_i\in M$.
Now we introduce the   vector-field $X_\lambda$ on the lattice $\M$
defined by
\begin{equation*}
(X_\lambda)_j(x)=\left \{
\begin{array}[c]{cl}
0 \qquad & \mbox{if  $j \notin \J$},\\
Y^\lambda_j(\pi_j(x))&\mbox{if  $j \in \J$}.
\end{array} \right.
\end{equation*}

By construction, we have $X_0=0$. Furthermore, the family of
vector-fields $X_\lambda$ is symplectic for all $\lambda$. Indeed,
consider a decay function $\psi$ and compute
\begin{equation*}
\L_{X_\lambda} \Omega_\infty=d \sum_{j \in \ZZ^N} i_{(X_\lambda)_j} \Omega_j
=d \sum_{j \in \J} \sum_{k=1}^{l} \lambda_k \delta^j_k =0,
\end{equation*}
where we have used that the last sum is finite by construction of
$X_\lambda$ and $(\delta^i_k)_{k=1,\dots ,l}$ are closed forms. We
obtain that  $X_\lambda$ is symplectic. Notice also that all 
but a finite number of the components of 
$X_\lambda$ are zero and so $DX_\lambda (x)\in
\L_\Gamma$, i.e. $X_\lambda$ is a decay vector field. If $x \in \M$,
denote $\varphi(s,\lambda,x)$  the flow generated by
$X_\lambda$. That is:
\begin{equation*}
\frac{d}{ds} \varphi(s,\lambda,x)= X_{\lambda} (\varphi(s,\lambda,x)),\qquad \varphi(0,\lambda,x)=x.
\end{equation*}
The existence and uniqueness of $\varphi(s,\lambda,x)$ is ensured by 
the theorem of existence and uniqueness of solutions for
Lipschitz differential equations in  Banach spaces. 
See \cite{Hale80}  for instance.
Here the Banach space is $\ell^\infty(\ZZ^N)$.

Given an exact symplectic map $F$ satisfying $F(0)=0$ and $F\in C^1_\Gamma$, we  define the family of maps we want to construct by
\begin{equation*}
F_\lambda=\varphi(\lambda,\lambda,.) \circ F.
\end{equation*}
 We have the following easy lemma
\begin{lemma}
For all $s \in \RR$, we have
\begin{enumerate}
\item $\varphi(s,0,x)=x$.
\item For all $j\in\ZZ^N$,
$\varphi_j(s,\lambda,.)$ only depends on $x_j$ and
\begin{equation*}
\varphi_j(s,\lambda,x)=x_j, \qquad j\notin\J.
\end{equation*}
\end{enumerate}
\end{lemma}
\begin{proof}
(1) follows directly from the fact that $X_0=0$. The first part of
(2) follows from the fact that $(X_\lambda)_j$ only depends on
$x_j$. Moreover if $j\notin \J$, $(X_\lambda)_j=0$ and then
$\varphi_j$ is constant in $s$. Therefore
$\varphi_j(s,\lambda,x)=\varphi_j(0,\lambda,x)=x_j$.
\end{proof}
As a consequence we have that
$$
F_0=F
$$
and
$$
(F_\lambda)_j= F_j,\qquad \mbox{ for } \lambda \in \RR^l,\quad j\notin \J.
$$

Since $\varphi _j$ is not constant for only 
a finite set of indexes, for $\lambda$ small
$\varphi(1,\lambda,.)$ is well-defined on the range of $F$. Moreover, since
$\varphi $ is uncoupled, i.e. $\pi_i  \varphi(s,\lambda,.)$ depends only 
on $x_i$ we have that
$\varphi(\lambda,\lambda,.) \in C^1_\Gamma$.
On the other hand, $(F_\lambda)_j(0)=F_j(0) =0$ for $\lambda\notin \J$.
Therefore $F$ satisfies the assumptions of Lemma \ref{compo2}.

Finally, the following lemma ends the details of the construction.
\begin{lemma}
For all $\lambda $, the map $F_\lambda \in C^1_\Gamma$ is symplectic, but it is
not exact symplectic for $\lambda\neq 0$.

Indeed, we have that if $\Psi$ is 
the embedding given by the coordinates in  
$\J$, and $[ \cdot ]$ denotes the cohomology class on the torus
expressed in the basis of  the forms $\delta_k$, we have 
\begin{equation}\label{cohomologychange} 
[ \Psi^* F^*_\lambda \alpha_\infty ] = \lambda .
\end{equation}
\end{lemma}
\begin{proof}
Let $\psi \in \mathcal{A}_{\rho,\underline{c},\Gamma}$ be a decay function. We want to prove that for any $\lambda$
\begin{equation*}
\psi^* F_\lambda ^* \Omega_\infty=\psi^* \Omega_\infty.
\end{equation*}

By construction of $F_\lambda$ we have
\begin{equation*}
\begin{split}
\psi^*F_\lambda^* \alpha_\infty&-\psi^*F^*_0 \alpha_\infty=
\psi^*F_0^* \varphi^*(\lambda,\lambda,.)\alpha_\infty- \psi^*F^*_0 \varphi^*(0,\lambda,.) \alpha_\infty \\
&= \int_0^1 \frac{d}{ds} \big(
\psi^*F^*_{0} \varphi^*(s,\lambda,.) \alpha_\infty \big) \,ds\\
&= \int_0^1   \psi^*F^*_{0} \Big((\frac{d}{ds} \varphi^*i (s,\lambda,.) \alpha_\infty \Big)  \,ds\\
&= \int_0^1   \psi^*F^*_{0}   \L_{X_{\lambda}(\varphi(s,\lambda,\cdot))} \alpha_\infty \,ds\\
&= \int_0^1   \psi^*F^*_{0} [d({i}_{X_{\lambda}}\alpha_\infty )+ {i}_{X_{\lambda}}d\alpha_\infty  ]\,ds\\
&=d \int_0^1  \psi^*F^*_{0}
({i}_{X_{\lambda}}\alpha_\infty) \,ds +\int_0^1
\psi^*F^*_{0}  \sum_{j \in \J} \sum_{k=1}^{l} \lambda_k \delta^j_k
\,ds\,.
\end{split}
\end{equation*}
Since $\delta_k^j$ are closed forms, taking exterior differential
at both sides of the previous formula we get $\psi^*F_\lambda^*
\Omega_\infty-\psi^*F^*_0 \Omega_\infty=0$. Finally, using that
$F_0$ is symplectic we get that $F_\lambda$ is symplectic.

Moreover, if $\lambda\ne 0$,
$$
\psi^*F_\lambda^* \alpha_\infty-\psi^*F^*_0 \alpha_\infty= dW_\infty + E,
$$
where
$W_\infty=\int_0^\lambda   \psi^*F^*_{0}  ({i}_{X_{\lambda}}\alpha_\infty) \,ds $
and $E$ is not a differential.
The formula \eqref{cohomologychange}, follows easily from the 
expression for $E$ above. 
We note that that 
\begin{equation*}
\begin{split} 
[\Psi^*F_\lambda^* \alpha_\infty] & = 
[\Psi^*F_\lambda^* \alpha_\infty -\Psi^*F^*_0 \alpha_\infty]  \\
&=[ d \int_0^1  \Psi^*F^*_{0}
({i}_{X_{\lambda}}\alpha_\infty) \,ds +\int_0^1
\Psi^*F^*_{0}  \sum_{j \in \J} \sum_{k=1}^{l} \lambda_k \delta^j_k
\,ds\, ] \\
&= 0 + \int_0^1 
[\Psi^*F^*_{0}  \sum_{j \in \J} \sum_{k=1}^{l} \lambda_k \delta^j_k]
\,ds \\
&= 
\int_0^1 
\Psi^* (F_{0})^* \sum_{j \in \J} \sum_{k=1}^{l} \lambda_k [\delta^j_k]
\,ds \\
&= \lambda
\end{split}
\end{equation*}

\end{proof}

\bibliographystyle{alpha}
\bibliography{whiskers}

\end{document}